\documentclass[11pt]{article}
\usepackage{amsmath}
\usepackage{epsfig, amssymb, amsfonts, dsfont}
\usepackage{amsthm}
\usepackage{amstext}
\usepackage{amsopn}
\usepackage{mathrsfs}
\usepackage{subfigure}
\usepackage{graphicx}
\usepackage[left=2.3cm,top=2cm,right=2.3cm,bottom=2cm, nohead,foot=1cm]{geometry}
\usepackage[makeroom]{cancel}
\usepackage{color}
\usepackage{enumerate}
\usepackage{comment}
\usepackage{stackrel}
\usepackage{stmaryrd}

\def\stackrel#1#2{\mathrel{\mathop{#2}\limits^{#1}}}

\numberwithin{equation}{section}

\usepackage{stmaryrd}
\SetSymbolFont{stmry}{bold}{U}{stmry}{m}{n}

\newtheorem{theorem}{Theorem}[section]

\newtheorem{lemma}[theorem]{Lemma}

\newtheorem{corollary}[theorem]{Corollary}
\newtheorem{proposition}[theorem]{Proposition}

\theoremstyle{definition}

\newtheorem{remark}[theorem]{Remark}

\newcommand{\R}{{\mathbb R}}

\input xy
\xyoption{all}

\DeclareFontFamily{U}{mathx}{\hyphenchar\font45}
\DeclareFontShape{U}{mathx}{m}{n}{
      <5> <6> <7> <8> <9> <10>
      <10.95> <12> <14.4> <17.28> <20.74> <24.88>
      mathx10
      }{}
\DeclareSymbolFont{mathx}{U}{mathx}{m}{n}
\DeclareMathSymbol{\bigtimes}{1}{mathx}{"91}

\usepackage{relsize}

\newcommand{\vsm}{{\mathsmaller{\mathsmaller{V}}}}

\newcommand{\Z}{{\mathbb Z}}

\date{\vspace{-9ex}}

\title{Continuum  polymer measures      corresponding \\ to the critical  2d stochastic heat flow}

\date{  }

 \author{\textbf{Jeremy  Clark}\footnote{ {\tt jeremy@olemiss.edu}}\hspace{.8cm}and\hspace{.8cm}\textbf{Barkat Mian}\footnote{ {\tt bmian@go.olemiss.edu}} \vspace{.1cm}  \\  University of Mississippi, Department of Mathematics   }

\begin{document}
\maketitle

\begin{abstract}  We construct continuum directed polymer measures corresponding to the   \textit{critical 2d stochastic heat flow} (2d SHF) introduced by Caravenna, Sun, and Zygouras in their recent article [Inventiones mathematicae \textbf{233},
 325--460 (2023)].  For this purpose, we prove a Chapman-Kolmogorov relation for the 2d SHF along with a related elementary conditional expectation formula.  We explore some basic properties of the continuum polymer measures, with our main focus being on their second moments. In particular, we  show that the form of their second moments is consistent with the family of  continuum  polymer measures, indexed by a disorder strength parameter, having a conditional Gaussian multiplicative chaos distributional interrelationship similar to that previously found in an analogous hierarchical toy model.  
\end{abstract}

\section{Introduction} The \textit{critical two-dimensional stochastic heat flow} (2d SHF) of index $\vartheta\in \R$ is a two-parameter process of random Borel measures $\mathscr{Z}^{\mathsmaller{\vartheta}}=\{\mathscr{Z}^{\mathsmaller{\vartheta}}_{s,t}(dx,dy)\}_{0\leq s < t <\infty}$ on  $(\R^2)^2$ derived 
by  Caravenna, Sun, and Zygouras in~\cite{CSZ5} as a universal distributional limit of point-to-point partition functions for (1+2)-dimensional models of a directed polymer in a random environment (DPRE) within a critical weak-coupling scaling regime.  Our scaling convention for the 2d SHF  will be such that  $\mathscr{Z}^{\vartheta}_{s,t}$ is equal to what in~\cite{CSZ5} is expressed by $ 2\mathscr{Z}^{\vartheta'}_{2s,2t} $  for $\vartheta':=\vartheta+\gamma_{\mathsmaller{\textup{EM}}}-\log 2$, where $\gamma_{\mathsmaller{\textup{EM}}}$ is the Euler-Mascheroni constant.
The parameter $\vartheta$ modulates between the  weak and strong disorder phases in the limits  $\vartheta\downarrow -\infty$ and  $\vartheta\uparrow \infty$, respectively.   The expectation of the random measure  $\mathscr{Z}^{\mathsmaller{\vartheta}}_{s,t}$ is the  Borel measure $\mathbf{U}_{t-s}$ on $(\R^2)^2$ with Lebesgue density 
\begin{align}\label{Gee}
\text{}\hspace{4.5cm}\mathbf{\dot{U}}_{t-s}(x,y)\,=\,g_{t-s}(x-y)\,,  \hspace{2cm}x,y\in \R^2\,,
\end{align}
in which $g_r:\R^2\rightarrow [0,\infty)$ is  the Gaussian function $g_r(x)=\frac{ 1 }{2\pi r}\exp\{-\frac{1}{2r} |x|^2  \} $.  In  loose terms,  we can characterize $\mathscr{Z}^{\mathsmaller{\vartheta}}$ as a  consistent randomization  of  the family $\{\mathbf{U}_{t-s}(dx,dy)\}_{0\leq s < t <\infty}$, where the consistency  is reflected in a Chapman-Kolmogorov type relation, one formulation of which we give in Proposition~\ref{PropChapman}.  After recalling some of the key properties of the 2d SHF below, we will review its connection with the two-dimensional stochastic heat equation. \vspace{.2cm}

The 2d SHF has a basic temporal independence property, namely that the random measures $\mathscr{Z}^{\mathsmaller{\vartheta}}_{s_1,t_1}$, \ldots, $\mathscr{Z}^{\mathsmaller{\vartheta}}_{s_n,t_n}$ are mutually  independent when the intervals  $(s_1,t_1),\ldots, (s_n,t_n)\subset [0,\infty)$ are disjoint.  Moreover, there is equality in distribution under  time and space shifts:
\[
\mathscr{Z}^{\mathsmaller{\vartheta}} \,\stackrel{\textup{d}}{=}\,\Big\{\,\mathscr{Z}^{\mathsmaller{\vartheta}}_{s+r,t+r}\big(d(x+a),d(y+a)\big)\,\Big\}_{0\leq s < t <\infty}
\]
for any $r\in [0,\infty)$ and $a\in \R^2$. 
A diffusive time-space rescaling results in a distributional equality with a shifted value of $\vartheta$:  
\begin{align}\label{Shifts}
\mathscr{Z}^{\mathsmaller{\vartheta}+\log \lambda} \,\stackrel{\textup{d}}{=}\,\Big\{\,\frac{1}{\lambda} \mathscr{Z}^{\mathsmaller{\vartheta}}_{\lambda s,\lambda t}\big(d(\sqrt{\lambda}x),d(\sqrt{\lambda}y)\big)\,\Big\}_{0\leq s < t <\infty}
\end{align}
for any $\lambda>0$.
While, as mentioned above,  the random measure $\mathscr{Z}^{\mathsmaller{\vartheta}}_{s,t} $ has expectation $\mathbf{U}_{t-s}$,  its variance $\mathbf{K}_{t-s}^{\mathsmaller{\vartheta}}$ has the 
 closed form 
\begin{align}\label{VarZee}
\mathbf{K}_{t-s}^{\mathsmaller{\vartheta}}(dx,dx';dy,dy') \,:=\,&\,\mathbb{E}\Big[ \, \mathscr{Z}^{\mathsmaller{\vartheta}}_{s,t}(dx,dy) \,\mathscr{Z}^{\mathsmaller{\vartheta}}_{s,t}(dx',dy') \,\Big] \,- \, \mathbf{U}_{t-s}(dx,dy)\, \mathbf{U}_{t-s}(dx',dy') \nonumber  \\ \,=\,&\,g_{t-s}\bigg(\frac{x+x'}{\sqrt{2}}-\frac{y+y'}{\sqrt{2}} 
  \bigg)\,K_{t-s}^{\mathsmaller{\vartheta}}\bigg(\frac{x-x'}{\sqrt{2}}, \frac{y-y'}{\sqrt{2}}  \bigg) \,dx\,dx'\,dy\,dy'
\,,
 \end{align}
where, in terms of the derivative $\nu'$ of the Volterra function $\nu(a):=\int_0^{\infty} \frac{a^r}{\Gamma(r+1)}  dr$, we define $ K_t^{\mathsmaller{\vartheta}}:(\R^2)^2\rightarrow [0,\infty]$ by
 \begin{align}\label{DefK}
K_t^{\mathsmaller{\vartheta}}(x,y)\,:=\, &\,2\,\pi \,e^{\vartheta}\,\int_{0<r<s<t   }\,  g_r(x)\,\nu'\big((s-r)e^{\vartheta}\big)\,g_{t-s}(y)\,ds\,dr\,. 
 \end{align}
 The  Volterra function and its multivariate generalizations are discussed in Apelblat's monograph~\cite{Apelblat0} and the survey~\cite{Garrappa} by Garrappa and Mainardi. Notably, $K_t^{\mathsmaller{\vartheta}}(x,y)$ diverges logarithmically  as $x$ or $y$ approaches the origin. 
The  function $ P_t^{\mathsmaller{\vartheta}}(x,y)  := g_t(x-y)+K_t^{\mathsmaller{\vartheta}}(x,y) $ is the integral kernel for a bounded operator on the Hilbert space $L^2(\R^2)$ of the form $\textup{exp}\big\{ -t H^\mathsmaller{\vartheta}\big\}$, in which $H^\mathsmaller{\vartheta}$ is a two-dimensional Schr\"odinger operator with a point potential at the origin; see~\cite[Chapter I.5]{AGHH} for a summary of these operators and~\cite[Equation (3.11)]{Albeverio} for an equivalent exponential kernel formula. This implies the semigroup property $P_{s+t}^{\mathsmaller{\vartheta}}(x,z)=\int_{\R^2}P_s^{\mathsmaller{\vartheta}}(x,y)P_t^{\mathsmaller{\vartheta}}(y,z) dy$ for $s,t\geq 0$. An interesting diffusion representation of the kernel $ P_t^{\mathsmaller{\vartheta}}(x,y)$ was found by  Chen in~\cite{Chen2}.\vspace{.2cm}

The 2d SHF  is the natural candidate for being the distributional limit law of regularized solutions of the two-dimensional stochastic heat equation  (2d SHE) within a critical scaling window.  The 2d SHE is the ill-defined stochastic partial differential equation (SPDE)
\begin{align}\label{SHE}\frac{\partial}{\partial t}\,U_t(x)\,=\,\frac{1}{2}\,\Delta\, U_t(x)\,+\, \sqrt{\vsm }\,\xi (t,x)\,U_t(x) \,, \hspace{.5cm}t\in [0,\infty)\,, \hspace{.3cm}x=(x_1,x_2)\in \R^2\,,  \end{align}
in which  $\Delta$ is the two-dimensional Laplacian, $\mathsmaller{V}>0$ is a coupling strength parameter, and  $\xi(t,x)$ is a time-space white noise on $[0,\infty)\times \R^2$, meaning a Gaussian field having the zero-range correlations $\mathbb{E}\big[ \xi(t,x) \xi(t',x')    \big]=\delta(t-t')\delta(x-x')$. 
In~\cite{BC} Bertini and Cancrini  
 proposed a delicate regularization scheme for generating a ``solution" to the 2d SHE, wherein the white noise   in~(\ref{SHE}) is replaced by a spatially mollified Gaussian field $\xi_{\varepsilon}(t,x)$ that converges to $\xi(t,x)$ as the parameter $\varepsilon>0 $ vanishes, and the coupling $\mathsmaller{V}\equiv \mathsmaller{V}^{\mathsmaller{\vartheta},\varepsilon}$ converges to zero as $\varepsilon \searrow 0$ under an asymptotics~(\ref{VSM}) that depends at higher order  on a refining  parameter $\vartheta \in \R$.  The  mollified Gaussian field is defined  by $\xi_{\varepsilon}(t,x):=  \frac{1}{\varepsilon^2}\int_{\R^2} \xi(t,y)j\big(\frac{y-x}{\varepsilon}\big)dy$ for some  function $j\in C_c^{\infty}(\R^2)$ with integral one, which has the correlations
$$ \mathbb{E}\big[\, \xi_{\varepsilon}(t,x)\,\xi_{\varepsilon}(t',x') \,\big]\,=\,\delta(t-t')\,J_{\varepsilon}( x-x' ) \,,
 $$
where  $J_{\varepsilon}(x):=\frac{1}{\varepsilon^2} J\big(\frac{x}{\varepsilon} \big) $ for $J( x ):=\int_{\R^2}j(y)j(x+y)dy$. The coupling strength $\mathsmaller{V}^{\mathsmaller{\vartheta},\varepsilon}>0$ is taken to  vanish with  small $\varepsilon$ under the asymptotics 
\begin{align}\label{VSM}
\vsm^{\mathsmaller{\vartheta},\varepsilon}\,\stackrel{\varepsilon\rightarrow 0}{=}\,\frac{ 2\pi  }{ \log \frac{1}{\varepsilon} }\bigg(1+ \frac{ \vartheta +2I_{J} }{ 2\log \frac{1}{\varepsilon} }\bigg)\, +\,\mathit{o}\bigg(\frac{ 1}{ \log^2 \frac{1}{\varepsilon} } \bigg)\,, 
\end{align}
in which $I_{J} $ has the following form  in terms of the function $J$: 
\begin{align*}
I_{J}\,:=\,\gamma_{\mathsmaller{\textup{EM}}}\,-\,\log 2\,+\,\int_{(\R^2)^2}\log |x-y| \,J(x)\,J(y)\,dx\,dy \,.
\end{align*}
Then, for an initial time $s\geq 0$ and data $\varphi\in C_c(\R^2)$, consider   the informal SPDE
\begin{align}\label{SHEII}\frac{\partial}{\partial t}\,U_{s,t}^{\mathsmaller{\vartheta},\varepsilon}(x,\varphi)\,=\,\frac{1}{2}\,\Delta\, U_{s,t}^{\mathsmaller{\vartheta},\varepsilon}(x,\varphi)\,+\, \sqrt{\vsm^{\mathsmaller{\vartheta},\varepsilon} }\,\xi_{\varepsilon} (t,x)\,U_{s,t}^{\mathsmaller{\vartheta},\varepsilon}(x,\varphi) \,,  \hspace{.5cm} U_{s,s}^{\mathsmaller{\vartheta},\varepsilon}(x,\varphi)\,=\,\varphi(x) \,,  \hspace{.5cm}t\,\geq \,s\,. \end{align}
  A solution to~(\ref{SHEII}) can be expressed through the Feynman-Kac formalism 
\begin{align}\label{DefU}
    U_{s,t}^{\mathsmaller{\vartheta},\varepsilon}(x,\varphi)\,=\,\int_{\boldsymbol{\Upsilon}_s} \,\textup{exp}\Bigg\{ \,\sqrt{\vsm^{\mathsmaller{\vartheta},\varepsilon} } \int_s^{t}\xi_{\varepsilon} \big(r, \,p(r) \big)\, dr     \,-\, \vsm^{\mathsmaller{\vartheta},\varepsilon} \,\frac{(t-s)\|j\|^2_2}{2\varepsilon^2} \,\Bigg\}    \,\varphi\big(p(t)\big)  \,\mathbf{P}_{s,x}(dp) \,,
\end{align}
 in which $\mathbf{P}_{s,x}$ denotes Wiener measure on the path space $\boldsymbol{\Upsilon}_s=C\big([s,\infty),\R^2   \big)$ for initial position $x\in \R^2$.  Note that for any $p\in \boldsymbol{\Upsilon}_s$ the process $\{B_t^{p}\}_{t\in [0,\infty)} $ defined by $ B_{t}^{p}:=\int_s^{s+t}\xi_{\varepsilon} \big(r, p(r) \big) dr$ is a 1d Brownian motion with diffusion rate $  \frac{1}{\varepsilon^2}\|j\|^2_2  $. The random positive linear functional that maps $\varphi \in C_c(\R^2)$ to $U_{s,t}^{\mathsmaller{\vartheta},\varepsilon}(x,\varphi)\in \R$ defines a random Borel measure $U_{s,t}^{\mathsmaller{\vartheta},\varepsilon}(x,\cdot)$ on $\R^2$, and we define the random Borel measure $\mathscr{Z}^{\mathsmaller{\vartheta},\varepsilon}_{s,t}$ on $(\R^2)^2$ by 
\begin{align}\label{DefZEpsilon}
    \text{}\hspace{1cm} \mathscr{Z}^{\mathsmaller{\vartheta},\varepsilon}_{s,t}(dx,dy)\,=\,dx\, U_{s,t}^{\mathsmaller{\vartheta},\varepsilon}(x,dy) \,, \hspace{1cm}x,y\in \R^2\,.
\end{align}
We can view the random element $ \mathscr{Z}^{\mathsmaller{\vartheta},\varepsilon}_{s,t}$ as taking values in the   space of locally finite Borel measures on $(\R^2)^2$, which is endowed with the vague topology.  
The expectation of the random measure $\mathscr{Z}^{\mathsmaller{\vartheta},\varepsilon}_{s,t}$ is  equal to $\mathbf{U}_{t-s}$, and the fine-tuning of the coupling parameter $\mathsmaller{V}^{\mathsmaller{\vartheta},\varepsilon} $ in~(\ref{VSM}) is such that the
variance of $\mathscr{Z}^{\mathsmaller{\vartheta},\varepsilon}_{s,t}$ converges vaguely to~(\ref{VarZee})  with small $\varepsilon$, as first observed in~\cite[Theorem 3.2]{BC}. \vspace{.2cm}

The process $ \mathscr{Z}^{\mathsmaller{\vartheta},\varepsilon}:=\big\{\mathscr{Z}^{\mathsmaller{\vartheta},\varepsilon}_{s,t}\big\}_{0\leq s<t<\infty}$ is expected to converge in distribution as $\varepsilon\rightarrow 0$ to the 2d SHF  $\mathscr{Z}^{\mathsmaller{\vartheta}}$, meaning that  for any $0\leq s_j<t_j<\infty$ and $\varphi_j\in  C_c\big( (\R^2)^2\big)$ with $j\in \{1,\ldots, N\}$ there is convergence in distribution
\begin{align}\label{MultZ}
\Big( \mathscr{Z}^{\mathsmaller{\vartheta},\varepsilon}_{s_1,t_1}(\varphi_1),\ldots, \mathscr{Z}^{\mathsmaller{\vartheta},\varepsilon}_{s_N,t_N}(\varphi_N)  \Big)  \hspace{.5cm}  \stackbin[\varepsilon\rightarrow 0  ]{\textup{d}}{\Longrightarrow}  \hspace{.5cm}  \Big( \mathscr{Z}^{\mathsmaller{\vartheta}}_{s_1,t_1}(\varphi_1),\ldots,  \mathscr{Z}^{\mathsmaller{\vartheta}}_{s_N,t_N}(\varphi_N)  \Big) \,.
\end{align}
Indeed, the  scaling regime in \cite{CSZ5} yielding the 2d SHF as its distributional limit can be understood as an alternative method for regularizing the 2d SHE, using models that are discrete in time and space.
In~\cite[Remark 1.4]{CSZ5} the authors give a plausible outline for how this distributional convergence might be proven in  analogy to the discrete polymer model counterparts, but a fuller discussion of the proof has not been provided in the literature yet to our knowledge.  It follows from results in~\cite{BC} that there exists a vanishing sequence $\{\epsilon_j\}_1^{\infty}$ and a limit law  $\hat{\mathscr{Z}}^{\mathsmaller{\vartheta}}$ such that there is convergence  in distribution
 \begin{align}\label{ZLimit}
\mathscr{Z}^{\mathsmaller{\vartheta},\varepsilon_j}  \hspace{.5cm}  \stackbin[j\rightarrow \infty  ]{\textup{d}}{\Longrightarrow}  \hspace{.5cm} \hat{\mathscr{Z}}^{\mathsmaller{\vartheta}}
 \end{align}
 in the same sense as~(\ref{MultZ}), leaving the question of uniqueness of the distributional limits open.  While Bertini and Cancrini identified $\mathbf{K}_{t-s}^{\mathsmaller{\vartheta}}$ in~(\ref{DefK}) as the limit of the variance of $\mathscr{Z}^{\mathsmaller{\vartheta},\varepsilon}_{s,t}$ as $\varepsilon\rightarrow 0$, a rigorous proof that any limit~(\ref{ZLimit}) does indeed have  variance  $\mathbf{K}_{t-s}^{\mathsmaller{\vartheta}}$ did not come until  Caravenna,  Sun, and Zygouras  provided a uniform integrability argument in~\cite{CSZ4} through a uniform bound on the third moments of the family  $\{\mathscr{Z}^{\mathsmaller{\vartheta},\varepsilon}\}_{\epsilon>0} $. In~\cite{GQT} Gu, Quastel, and Tsai proved that each higher integer moment $ \mathbb{E}\big[ (\mathscr{Z}^{\mathsmaller{\vartheta},\varepsilon})^m  \big] $ converges vaguely as $\varepsilon\rightarrow 0$ to a limit, which they  characterize using a theory of  $m$-particle Schr\"odinger operators with point interactions, thus uniquely determining the  higher moments of any  distributional limit~(\ref{ZLimit}); see also Chen's recent work~\cite{Chen1} for a probabilistic perspective on the convergence of the moments, giving a useful technical extension of the main result in~\cite{GQT}. 
 This does not entail distributional uniqueness for limits~(\ref{ZLimit}) because the sequence $ \mathbb{E}\big[ (\hat{\mathscr{Z}}^{\mathsmaller{\vartheta}}(\varphi) )^m  \big] $ in $m\in \mathbb{N}$ is expected to increase too quickly  for nonzero $\varphi\in C_c\big( (\R^2)^2\big)$. \vspace{.2cm}

The aim of this article is to construct and initiate  a study of  random Borel measures $\mathscr{Z}^{\mathsmaller{\vartheta}}_{[s,t]}$ on the path space $\boldsymbol{\Upsilon}_{[s,t]}:=   C\big(  [s,t],\R^2   \big)$ having a canonical relation to the SHF $\mathscr{Z}^{\mathsmaller{\vartheta}} $. These random path measures have a distributional invariance under spatial shifts, inherited from the 2d SHF~(\ref{Shifts}), and are almost surely infinite. We do not address the question of constructing a normalized version of these measures in which they  are supported on a set of paths with fixed starting point $x\in \R^2$ and have total mass one (i.e., are probability measures) here.  Our approach takes  inspiration from the article~\cite{AKQ2} by Alberts, Khanin, and Quastel formulating a 1d continuum directed polymer model; see also~\cite[Section 2.10]{Quastel} for a very brief overview of the model.
For a partition $P=\{t_0,\ldots, t_m\}$ of the interval $[s,t]$, the random measure $\mathscr{Z}^{\mathsmaller{\vartheta}}_{[s,t]}$   formally satisfies
\begin{align}\label{fdd}
\mathscr{Z}^{\mathsmaller{\vartheta}}_{[s,t]}\Big(&\Big\{\, p \in  \boldsymbol{\Upsilon}_{[s,t]}\,:\,p(t_0)\in dx_0\,,\,\, p(t_1)\in dx_1 \,,\, \ldots \,,\, p(t_m)\in dx_m \,\Big\}\Big) \nonumber   \\ &\,=\,\mathscr{Z}^{\mathsmaller{\vartheta}}_{t_0,t_1}(dx_0,dx_1)\,\frac{\mathscr{Z}^{\mathsmaller{\vartheta}}_{t_1,t_2}(dx_1,dx_2)}{dx_1}\,\cdots\, \frac{\mathscr{Z}^{\mathsmaller{\vartheta}}_{t_{m-1},t_m}(dx_{m-1},dx_m)}{dx_{m-1}}\,,\hspace{.8cm} x_j\in \R^2\,.
\end{align}
The right side above is ill-defined since the projection 
$\overline{\mathscr{Z}}^{\mathsmaller{\vartheta}}_{a,b}(dx):=\mathscr{Z}^{\mathsmaller{\vartheta}}_{a,b}(dx,\R^2)$ is almost surely singular with respect to Lebesgue measure on $\R^2$, and thus $\mathscr{Z}^{\mathsmaller{\vartheta}}_{a,b}$ does not have a disintegration analogous to~(\ref{DefZEpsilon}).\footnote{The almost sure singularity of $\overline{\mathscr{Z}}^{\mathsmaller{\vartheta}}_{a,b}$  with respect to Lebesgue measure on $\R^2$ is a result in a manuscript~\cite{CSZ8} that Caravenna, Sun, and Zygouras are preparing.}  However,  a simple  limit scheme through Gaussian averaging can be used to rigorously define the right side of~(\ref{fdd}).  The first moment $\mathbf{U}_{[s,t]}=\mathbb{E}\big[\mathscr{Z}^{\mathsmaller{\vartheta}}_{[s,t]}\big]$ is Wiener measure on $\boldsymbol{\Upsilon}_{[s,t]}$ with initial position having Lebesgue ``distribution" on $\R^2$.  The second moment $\mathbf{Q}^{\mathsmaller{\vartheta}}_{[s,t]}=\mathbb{E}\big[(\mathscr{Z}^{\mathsmaller{\vartheta}}_{[s,t]})^2\big]$ was effectively the subject of our previous  work~\cite{CM}, and we explain the connection between $\mathbf{Q}^{\mathsmaller{\vartheta}}_{[s,t]}$ and the path measures considered there in Section~\ref{SubsectionSM}.  For any $\mathsmaller{\vartheta}'\in \R$ the measures $\mathbf{Q}^{\mathsmaller{\vartheta}'}_{[s,t]}$ and $\mathbf{Q}^{\mathsmaller{\vartheta}}_{[s,t]}$ are equivalent with Radon-Nikodym derivative
\begin{align}\label{RNDeriv}
\frac{ d\mathbf{Q}^{\mathsmaller{\vartheta}'}_{[s,t]} }{d\mathbf{Q}^{\mathsmaller{\vartheta}}_{[s,t]}  }(p,q)\,=\,\textup{exp}\Big\{\, (\vartheta'-\vartheta)\,\mathbf{I}_{[s,t]}(p,q) \,\Big\}\,,\hspace{1cm}p,q\in \boldsymbol{\Upsilon}_{[s,t]}\,,
\end{align}
where $\mathbf{I}_{[s,t]}(p,q)\in [0,\infty]$ is an \textit{intersection time} between the paths $p,q$, that is, a certain measure of the set of intersection times $\{ r\in [s,t]: p(r)=q(r)\}$. As we argue heuristically  in Section~\ref{SubsectCGMC}, the formula~(\ref{RNDeriv})  suggests that when $\vartheta'>\vartheta$ the law of $\mathscr{Z}^{\mathsmaller{\vartheta}'}_{[s,t]}$ can be constructed as a Gaussian multiplicative chaos (GMC) with random reference measure $\mathscr{Z}^{\mathsmaller{\vartheta}}_{[s,t]}$.   A result of this type was previously obtained in~\cite{Clark4} for an analogous toy model defining a critical continuum directed polymer defined on a diamond fractal~\cite{Clark3}.   \vspace{.2cm}

In the recent  work~\cite{QRV} Quastel, Ram\'irez, and Vir\'ag studied a   2d continuum polymer model associated with a  2d SHE driven by a spatial white noise, that is, differing from~(\ref{SHE}) in that the multiplicative  time-space white noise term $\xi(t,x)$ is replaced by a spatial one $\xi(x)$. This translates to their continuum polymer being \textit{undirected} because, heuristically, a polymer  may return to a spatial point $x\in \R^2$, thus sampling the white noise variable $\xi(x)$ with  `multiplicity.' The authors showed that their polymer measures can be constructed as GMCs using Shamov's formulation of subcritical GMC in~\cite{Shamov}.  As we explain in Section~\ref{SubsectCGMC}, this is not possible for the continuum polymer measures that we consider here, except in the relative sense mentioned above of $\mathscr{Z}^{\mathsmaller{\vartheta}'}_{[s,t]}$ being constructible (in law) as a GMC with random reference measure $\mathscr{Z}^{\mathsmaller{\vartheta}}_{[s,t]}$ when $\vartheta'>\vartheta$.
In~\cite{Berger0,Berger1}  Berger and Lacoin considered another family of continuum polymer models corresponding to the SHE, valid for any dimension, in which the Gaussian white noise $\xi(t,x)$ is replaced by a   time-space white  L\'evy noise making positive jumps,  which includes the one-sided $\alpha$-stable  as a special case.  Realizations of their polymer measures are extremely singular with respect to Wiener measure  when the  jump size is not integrable near zero under the L\'evy measure \cite[Proposition 2.18(iii)]{Berger1}. See also the related work~\cite{Berger2} by Berger, Chong, and Lacoin, which obtains some strong localization results through control of the  moments of the partition functions over long time scales.\vspace{.2cm}

Rather than using the 2d SHF $\mathscr{Z}^{\mathsmaller{\vartheta}}$---understood in the sense of~\cite{CSZ5} as the universal distributional limit of discrete 2d directed polymer models within their critical scaling window---for the foundation
of our construction of the random path measure $\mathscr{Z}^{\mathsmaller{\vartheta}}_{[s,t]}$, our analysis will start with a distributional limit~(\ref{ZLimit}) generated from the regularized 2d SHE.  This difference comes up only in the proof of  Proposition~\ref{PropCond} in Section~\ref{SubsectionCond}, where we are deducing a basic conditional expectation property for the limiting object using its counterpart for the prelimits. Our reason for this choice is to avoid the extra set of notations needed for  the discrete polymer models, and the analysis here can be carried out from the discrete  models using the techniques in~\cite{CSZ5}  without any essential change. Of course, as mentioned above, the  distributional limits~(\ref{ZLimit}) are expected to be equal  to $\mathscr{Z}^{\mathsmaller{\vartheta}}$, and their known properties coincide.  In the sequel, we will drop the hat from the symbol $\hat{\mathscr{Z}}^{\mathsmaller{\vartheta}}$ in~(\ref{ZLimit})  and refer to its law  as the 2d SHF.

\subsection{Article organization and a brief summary of the main results}
In Section~\ref{SectResults} we state our main results, from which we highlight the following in summary form:
\begin{itemize}

\item Proposition~\ref{PropCond} presents a conditional expectation formula for $\mathscr{Z}^{\mathsmaller{\vartheta}}_{r,u}$ conditional on information generated by the 2d SHF over the intervals $[0,s]$ and $[t,\infty)$ for  $r<s<t<u$, meaning the $\sigma$-algebra $\sigma\big\{ \mathscr{Z}^{\mathsmaller{\vartheta}}_{a,b} : (a,b)\subset [0,s]\cup [t,\infty) \big\}$.

\item Proposition~\ref{PropChapman} states a Chapman-Kolmogorov type relation for the 2d SHF, constructing $\mathscr{Z}^{\mathsmaller{\vartheta}}_{r,t}$ from the pair  $\mathscr{Z}^{\mathsmaller{\vartheta}}_{r,s}$ and $\mathscr{Z}^{\mathsmaller{\vartheta}}_{s,t}$ when $r<s<t$.

\item Theorem~\ref{ThmSHFExtension} formulates random measures $\mathscr{Z}^{\mathsmaller{\vartheta}}_{P}$ on $(\R^2)^{m+1}$ corresponding to  partitions $P=\{t_0,t_1,\ldots, t_m\}$ of an interval $[s,t]$ that are heuristically expressible by the right side of~(\ref{fdd}).  This is a multi-interval extension of the 2d SHF, and their projective (marginal) consistency in terms of removing intermediate times $t_j$, $1\leq j<m$ generalizes the Chapman-Kolmogorov relation.

\item Theorem~\ref{ThmPathMeasConst} concerns the existence of a continuum polymer measure $\mathscr{Z}^{\mathsmaller{\vartheta}}_{[s,t]}$ consistent with the  measures $\mathscr{Z}^{\mathsmaller{\vartheta}}_{P}$ in the sense of~(\ref{fdd}). 

\item Theorem~\ref{ThmRNDer} lists some  key properties of the second moment $\mathbf{Q}^{\mathsmaller{\vartheta}}_{[s,t]}$ of $\mathscr{Z}^{\mathsmaller{\vartheta}}_{[s,t]}$, including~(\ref{RNDeriv}).

\end{itemize}

Proofs are placed in  Sections~\ref{SecChapmanProof}--\ref{SecCPM}.  Many of our arguments are simply based on second moment calculations.  The conditional expectation in the first bullet  eases many of these computations, and it underlies a martingale property that we use to construct $\mathscr{Z}^{\mathsmaller{\vartheta}}_{P}$.  The construction of $\mathscr{Z}^{\mathsmaller{\vartheta}}_{[s,t]}$ relies on Kolomogorov's extension theorem through the consistency of the measures $\mathscr{Z}^{\mathsmaller{\vartheta}}_{P}$ referred to in the third bullet. As we mentioned above, there is a direct connection between the second moment measure $\mathbf{Q}^{\mathsmaller{\vartheta}}_{[s,t]}$ and the path measures studied in~\cite{CM}, which we elaborate in Section~\ref{SubsectionSM}. \vspace{.2cm}

Appendix~\ref{SecAppendix}  includes a remark on the existence of a subsequential distributional limit~(\ref{ZLimit}) of the regularized 2d SHE.

\subsection{Notation}

We will use the  notational conventions below. Every  measure on a topological space is assumed to be a Borel measure.   We use $X$ to denote a generic Polish space below.
\begin{itemize}

\item $C_b(X)$ denotes the set of real-valued bounded continuous functions on $X$.

\item $\hat{C}_b(X)$ denotes the set of functions in $C_b(X)$  with bounded support.

\item  $\widehat{\mathcal{M}}(X)$ denotes the set of finite  measures on $X$, which is endowed with the weak topology, meaning the topology generated by the family of maps   $\mu \rightarrow  \mu(\varphi)$ for $\varphi\in C_b(X)$, $\mu\in \widehat{\mathcal{M}}(X) $.

\item $\mathcal{M}(X)$  denotes the set of locally finite measures on $X$, that is measures that are finite-valued on bounded subsets of $X$.  We equip $\mathcal{M}(X)$   with the topology of vague convergence, meaning the topology generated by the family of maps   $\mu \rightarrow  \mu(\varphi)$ for $\varphi\in \hat{C}_b(X)$, $\mu\in \mathcal{M}(X) $.

\item If $\boldsymbol{\mu}_{j}$ for $j\in \mathbb{N}\cup\{\infty\}$ are $\mathcal{M}(X)$-valued random elements  on some probability space, then  we say that $\boldsymbol{\mu}_j$ \textit{converges to $\boldsymbol{\mu}_\infty$ vaguely in $L^p$}  when $\boldsymbol{\mu}_j(\varphi)\rightarrow \boldsymbol{\mu}_{\infty}(\varphi)$ in $L^p$ for each $\varphi\in \hat{C}_b(X)$.

\item Given a measure $\mu$  and $n\in \mathbb{N}$, we denote its $n$-fold product by $\mu^n$. 

\item When $\mu$ is a measure on $(\R^2)^m$, we interpret $\mu^2$ as a measure on  $(\R^2\times \R^2)^m$ as follows:
\[
\mu^2\big(dx_1,dx_1';\ldots; dx_m,dx_m'\big)\,=\,\mu\big(dx_1,\ldots, dx_m\big)\,\mu\big(dx_1',\ldots, dx_m'\big)\,,\hspace{.5cm}x_j,x_j'\in \R^2\,.
\]

\item When $\mu$ is a measure on $\boldsymbol{\Upsilon}_{[s,t]}:= C\big([s,t],\mathbb{R}^2\big)$, we interpret   $\mu^2$ as a measure on   $C\big([s,t], (\mathbb{R}^2)^2\big)$ through the canonical  identification   $\boldsymbol{\Upsilon}_{[s,t]}^2\equiv C\big([s,t],(\mathbb{R}^2)^2\big)$.

\item If a measure $\mu$ on $\R^n$ is absolutely continuous with respect to Lebesgue measure, we denote its density by $\dot{\mu}$.

\end{itemize}

\section{From the 2d SHF to continuum  polymer measures}\label{SectResults}
In Section~\ref{SubsectSubsequentialLimit} we review some relevant known properties of a distributional limit law~(\ref{ZLimit}).  We state a Chapman-Kolmogorov relation for the 2d SHF in Section~\ref{SecChapmanState}, which  underlies the consistency of a multi-interval extension of the 2d SHF that we introduce in Section~\ref{SubsectMTSHF}.  The continuum polymer measure and its second moment are discussed in Sections~\ref{SubsectCPM}--\ref{SubsectCGMC}.

\subsection{On subsequential distributional limits of the regularized 2d SHE}\label{SubsectSubsequentialLimit}
Put  $S:=\big\{(s,t): 0\leq s<t<\infty\big\}$ and $\mathcal{M}_m:=\mathcal{M}\big( (\R^2)^{m+1}\big)$,  omitting the subscript when $m=1$. The  starting point for our analysis is the limiting distribution on $\mathcal{M}^S$ in the lemma below, which is implied by Bertini and Cancrini's analysis in~\cite{BC}.  We include a remark on its proof in Appendix~\ref{SecAppendix}.  
\begin{lemma}\label{LemmaSubseqLimit}  Fix $\vartheta\in \R$.  For $\epsilon>0$   put $\mathscr{Z}^{\mathsmaller{\vartheta},\epsilon}=\big\{ \mathscr{Z}^{\mathsmaller{\vartheta},\epsilon}_{s,t}\big\}_{(s,t)\in S} $, where $\mathscr{Z}^{\mathsmaller{\vartheta},\epsilon}_{s,t}$ is the $\mathcal{M}$-valued random element defined in~(\ref{DefZEpsilon}).  There exists a vanishing sequence $\{\epsilon_j\}_1^{\infty}\subset (0,\infty)$ such that the sequence of $\mathcal{M}^S$-valued random elements
$\{\mathscr{Z}^{\mathsmaller{\vartheta},\epsilon_j}\}_1^{\infty}$ converges in the sense of finite-dimensional distributions.
\end{lemma}

For $t>0$ let  $\mathbf{Q}_{t}^{\mathsmaller{\vartheta}}$ denote  the measure on $\big(\R^2\times \R^2\big)^2$ with Lebesgue density
\begin{align}\label{QDensity}
\mathbf{\dot{Q}}_{t}^{\mathsmaller{\vartheta}}\big(x,x';y,y'\big)\,:=\,&\,g_{t}\bigg(\frac{x+x'}{\sqrt{2}}-\frac{y+y'}{\sqrt{2}} 
  \bigg)\,P_{t}^{\mathsmaller{\vartheta}}\bigg(\frac{x-x'}{\sqrt{2}}, \frac{y-y'}{\sqrt{2}}  \bigg) \,,\hspace{1cm}x,x',y,y'\in \R^2\,,
\end{align}
where $P_{t}^{\mathsmaller{\vartheta}}$ is defined below~(\ref{DefK}).  Note that  $\mathbf{K}_{t}^{\mathsmaller{\vartheta}}:=\mathbf{Q}_{t}^{\mathsmaller{\vartheta}}-\mathbf{U}_{t}^2$ is expressed in~(\ref{VarZee}).  The next proposition lists several of the basic properties of a limit law in Lemma~\ref{LemmaSubseqLimit}.
\begin{proposition}\label{PropBasicProperties}Fix $\vartheta\in \R$. If $\mathscr{Z}^{\mathsmaller{\vartheta}}=\big\{ \mathscr{Z}^{\mathsmaller{\vartheta}}_{s,t}\big\}_{(s,t)\in S} $ is an  $\mathcal{M}^S$-valued random element having a limiting distribution from Lemma~\ref{LemmaSubseqLimit}, then (i)--(v) below hold.
\begin{enumerate}[(i)]

\item The first and second moments of $\mathscr{Z}^{\mathsmaller{\vartheta}}_{s,t}$ are $\mathbf{U}_{t-s}$ and $\mathbf{Q}_{t-s}^{\mathsmaller{\vartheta}}$, respectively.

\item The $m^{\textup{th}}$ moment $\mathbb{E}\big[(\mathscr{Z}^{\mathsmaller{\vartheta}}_{s,t})^m\big]$ is locally finite for each $m\in \mathbb{N}$.

\item $\mathscr{Z}^{\mathsmaller{\vartheta}}_{s_1,t_1}$, \ldots, $\mathscr{Z}^{\mathsmaller{\vartheta}}_{s_n,t_n}$ are mutually independent when the intervals $(s_1,t_1)$,\ldots,$(s_n,t_n)$ are disjoint.

\item $\mathscr{Z}^{\mathsmaller{\vartheta}}_{s,t}$ and $\mathscr{Z}^{\mathsmaller{\vartheta}}_{r+s,r+t}$ are equal in distribution for each $r>0$.

\item $\mathscr{Z}^{\mathsmaller{\vartheta}}_{s,t}$ and $\mathscr{Z}^{\mathsmaller{\vartheta}}_{s,t}\circ \tau_a^{-1}$ are equal in distribution for each $a\in \R^2$, where $\tau_a:(\R^2)^2\rightarrow (\R^2)^2$ is the shift map $(x,y)\rightarrow (x+a,y+a)$ .

\end{enumerate}
\end{proposition}
The claim in (i) that $\mathbb{E}\big[\mathscr{Z}^{\mathsmaller{\vartheta}}_{s,t}\big]=\mathbf{U}_{t-s}$   follows from Lemma~\ref{LemmaSubseqLimit} and a uniform integrability argument using  uniform bounds on the second moments of the family $\{ \mathscr{Z}^{\mathsmaller{\vartheta},\epsilon}_{s,t}  \}_{\varepsilon\in (0,1]}$ in the sense of~\cite[Lemma 4.1]{BC}, for instance.  Similarly, $\mathbb{E}\big[(\mathscr{Z}^{\mathsmaller{\vartheta}}_{s,t})^2\big]=\mathbf{Q}_{t-s}^{\mathsmaller{\vartheta}}$  was established in~\cite[Theorem 1.9]{CSZ4} using  uniform bounds on the third moments  of the family $\{ \mathscr{Z}^{\mathsmaller{\vartheta},\epsilon}_{s,t}  \}_{\varepsilon\in (0,1]}$. Statement (ii) follows from~\cite[Theorem 1.1]{GQT}, which implies  vague convergence of  each positive integer moment $\mathbb{E}\big[(\mathscr{Z}^{\mathsmaller{\vartheta},\varepsilon}_{s,t})^m\big]$ to $\mathbb{E}\big[(\mathscr{Z}^{\mathsmaller{\vartheta}}_{s,t})^m\big]$ as $\varepsilon\rightarrow 0$.  The remaining properties (iii)--(v) are straightforwardly inherited from $\mathscr{Z}^{\mathsmaller{\vartheta},\epsilon}$.

\subsection{A conditional expectation formula and  Chapman-Kolmogorov type relation}\label{SecChapmanState}

As a preliminary for constructing  a continuum polymer measure associated with  $\mathscr{Z}^{\mathsmaller{\vartheta}}$, we require a Chapman-Kolmogorov type relation between the random measures $ \mathscr{Z}^{\mathsmaller{\vartheta},\varepsilon}_{r,s}$,  $ \mathscr{Z}^{\mathsmaller{\vartheta},\varepsilon}_{s,t}$, and $ \mathscr{Z}^{\mathsmaller{\vartheta},\varepsilon}_{r,t}$ for $r<s<t$. We formulate this Chapman-Kolmogorov relation using a binary operation for measures  on $(\R^2)^2 $, defined in greater generality below.  Let $X$ and $X'$ be  Polish spaces.  Given  $\varsigma>0$ and $\sigma$-finite measures  $\mu_1$ and $\mu_2$  on  $X\times \R^n$ and $ \R^n\times X'$, respectively,   we define $\mu_1 \mathlarger{\mathlarger{\bullet}}_{\varsigma} \mu_2 $ as the  measure on $X\times X'$  given by 
\begin{align}\label{MeasComb}
\text{}\hspace{1.1cm}\mu_1 \,\mathlarger{\mathlarger{\bullet}}_{\varsigma}\, \mu_2 (dx,dx')\,=\,&\int_{a,b\in \R^n } \,
 \mu_1(dx,da)\,g_{\varsigma}(a-b)\,\mu_2(db,dx')\,,  \hspace{1cm}x\in X,\,x'\in X' \,,
\end{align}
for the $n$-dimensional Gaussian $g_{\varsigma}(a)=(2\pi \varsigma)^{-\frac{n}{2}} e^{-\frac{ |a|^2 }{ 2\varsigma  } } $. We use $\mu_1 \,\mathlarger{\mathlarger{\bullet}}_{\varsigma}$ as a shorthand for the measure on $X\times \R^n$  with
\begin{align}\label{MeasComb2}
\text{}\hspace{1.1cm}\mu_1 \,\mathlarger{\mathlarger{\bullet}}_{\varsigma} (dx,dx')\,=\,&\int_{a\in \R^n } \,
 \mu_1(dx,da)\,g_{\varsigma}(a-x')\,dx'\,,  \hspace{1cm}x\in X,\,x'\in \R^n \,,
\end{align}
and $\mathlarger{\mathlarger{\bullet}}_{\varsigma}\,\mu_2$ is the analogously defined measure on $\R^n\times X'$.  When $\mu_2$ has a disintegration with respect to Lebesgue measure on $\R^n$, that is $\mu_2(dx,ds')=dx\,\dot{\mu}_2(x,ds')$, for a kernel $\dot{\mu}_2$ from  $\R^n$ to $X'$, we define 
\[
\mu_1 \,\mathlarger{\mathlarger{\bullet}}\, \mu_2(dx,dx')\,=\,\int_{a\in \R^n}\, \mu_1(dx,da)\, \dot{\mu}_2(a,dx')\,,
\]
and we define $\mu_1 \,\mathlarger{\mathlarger{\bullet}}\, \mu_2$ analogously when $\mu_1$ has a disintegration with respect to Lebesgue measure on $\R^n$. By virtue of the semigroup property held by the families of integral kernels $\big\{g_a (x-y)\big\}_{a\in [0,\infty)}$ and $\big\{P_{a}^{\mathsmaller{\vartheta}}(x,y)\big\}_{a\in [0,\infty)}$, we have
 $\mathbf{U}_{a}\,\mathlarger{\mathlarger{\bullet}}\, \mathbf{U}_{b}=\mathbf{U}_{a+b}$ and $\mathbf{Q}_{a}^{\mathsmaller{\vartheta}}\,\mathlarger{\mathlarger{\bullet}}\, \mathbf{Q}_{b}^{\mathsmaller{\vartheta}}=\mathbf{Q}_{a+b}^{\mathsmaller{\vartheta}}$, where for the latter the measures are understood to be on $(\R^4)^2\equiv (\R^2\times \R^2)^2$.
 Our definitions trivially yield that 
\begin{align*}
 &\, \mu_1 \,\mathlarger{\mathlarger{\bullet}}_{\varsigma}\, \,\mu_2 \,=\,\mu_1 \,\mathlarger{\mathlarger{\bullet}}\, (\mathlarger{\mathlarger{\bullet}}_{\varsigma}\,\mu_2)\,=\, (\mu_1 \,\mathlarger{\mathlarger{\bullet}}_{\varsigma})\,\mathlarger{\mathlarger{\bullet}}\,\mu_2 \,.
\end{align*}
Part of our motivation for defining the operation~(\ref{MeasComb}) comes from the conditional expectation identity in Proposition~\ref{PropCond} below.

The proofs of the three propositions below are placed in Section~\ref{SecChapmanProof}.  For $A\subset [0,\infty)$ define the $\sigma$-algebra $\mathscr{F}_{A}:= \sigma\big\{ \mathscr{Z}^{\mathsmaller{\vartheta}}_{s,t} \,:\, (s,t) \subset  A     \big\} $.  
\begin{proposition} \label{PropCond}  Fix $\vartheta\in \R$.  Let $\mathscr{Z}^{\mathsmaller{\vartheta}}$  be an $\mathcal{M}^S$-valued random element with a limit distribution from Lemma~\ref{LemmaSubseqLimit}.  For any  $0\leq  r<s<t<u$, we have 
 \[
\mathbb{E}\Big[\,\mathscr{Z}^{\mathsmaller{\vartheta}}_{r,u}\,\Big|\, \mathscr{F}_{[0,s]\cup [t,\infty) }  \,\Big]    \,=\,\mathscr{Z}^{\mathsmaller{\vartheta}}_{r,s}  \,\mathlarger{\mathlarger{\bullet}}_{t-s}\,  \mathscr{Z}^{\mathsmaller{\vartheta}}_{t,u} \,.
 \]
When $s=r$ or $t=u$, the right side above  reduces to $\mathlarger{\mathlarger{\bullet}}_{t-s}\,  \mathscr{Z}^{\mathsmaller{\vartheta}}_{t,u} $ and $\mathscr{Z}^{\mathsmaller{\vartheta}}_{r,s}\, \mathlarger{\mathlarger{\bullet}}_{t-s}$ respectively.
\end{proposition}

Proposition~\ref{PropCond} and a short computation yields the following.
\begin{proposition}\label{PropCond2}  Fix $\vartheta\in \R$.  Let $\mathscr{Z}^{\mathsmaller{\vartheta}}$  be an $\mathcal{M}^S$-valued random element with a limit distribution from Lemma~\ref{LemmaSubseqLimit}. For any  $ 0\leq r<s<t<u$, the second moment of the difference between $\mathscr{Z}^{\mathsmaller{\vartheta}}_{r,u}$ and $\mathscr{Z}^{\mathsmaller{\vartheta}}_{r,s}  \,\mathlarger{\mathlarger{\bullet}}_{t-s}\,  \mathscr{Z}^{\mathsmaller{\vartheta}}_{t,u}$ is the following positive measure:
\begin{align}\label{QKQ}
\mathbb{E}\Big[ \,\big(\mathscr{Z}^{\mathsmaller{\vartheta}}_{r,u}-\mathscr{Z}^{\mathsmaller{\vartheta}}_{r,s}  \,\mathlarger{\mathlarger{\bullet}}_{t-s}\,  \mathscr{Z}^{\mathsmaller{\vartheta}}_{t,u}\big)^2\,\Big]\,=\,\mathbf{Q}_{s-r}^{\mathsmaller{\vartheta}}  \,\mathlarger{\mathlarger{\bullet}}\,    \mathbf{K}_{t-s}^{\mathsmaller{\vartheta}}  \,\mathlarger{\mathlarger{\bullet}}\, \mathbf{Q}_{u-t}^{\mathsmaller{\vartheta}} \,.
\end{align}
\end{proposition}

The right side of~(\ref{QKQ}) vanishes vaguely as $t\searrow s$, implying that $\mathscr{Z}^{\mathsmaller{\vartheta}}_{r,s}  \,\mathlarger{\mathlarger{\bullet}}_{t-s}\,  \mathscr{Z}^{\mathsmaller{\vartheta}}_{t,u}$ converges vaguely in $L^2$ to $\mathscr{Z}^{\mathsmaller{\vartheta}}_{r,u}$.  Likewise, the difference between  $\mathscr{Z}^{\mathsmaller{\vartheta}}_{r,s}  \,\mathlarger{\mathlarger{\bullet}}_{\varsigma}\,  \mathscr{Z}^{\mathsmaller{\vartheta}}_{s+\varsigma,t}$ and $\mathscr{Z}^{\mathsmaller{\vartheta}}_{r,t} $ vanishes vaguely in $L^2$ as $\varsigma\rightarrow 0$, bringing us close to proving the next proposition, as only  the error induced from replacing $\mathscr{Z}^{\mathsmaller{\vartheta}}_{s+\varsigma,t}$ by $\mathscr{Z}^{\mathsmaller{\vartheta}}_{s,t}$ needs to be controlled.
\begin{proposition}[Chapman-Kolmogorov Type Relation] \label{PropChapman}  Fix $\vartheta\in \R$ and $0\leq r<s<t$. Let $\mathscr{Z}^{\mathsmaller{\vartheta}}$  be an $\mathcal{M}^S$-valued random element with a limit distribution from Lemma~\ref{LemmaSubseqLimit}.  Then  the random measure $\mathscr{Z}^{\mathsmaller{\vartheta}}_{r,s}  \,\mathlarger{\mathlarger{\bullet}}_{\varsigma}\,  \mathscr{Z}^{\mathsmaller{\vartheta}}_{s,t}$ converges vaguely in $L^2$ to $\mathscr{Z}^{\mathsmaller{\vartheta}}_{r,t}$  as $\varsigma\rightarrow 0$.
\end{proposition}

 Note in particular that $\mathscr{Z}^{\mathsmaller{\vartheta}}_{r,t}$ depends measurably on the pair $(\mathscr{Z}^{\mathsmaller{\vartheta}}_{r,s},  \mathscr{Z}^{\mathsmaller{\vartheta}}_{s,t})$.  The independence property in (iii) of Proposition~\ref{PropBasicProperties}  combined with  Proposition~\ref{PropChapman} imply the following unsurprising lemma, which we prove at the end of Section~\ref{SecChapmanProof}.
\begin{lemma}\label{LemInd} The $\sigma$-algebras $\mathscr{F}_{[s,t]}$ and  $\mathscr{F}_{[0,s]\cup [t,\infty) } $ are independent. 
\end{lemma}

\begin{remark}\label{RemarkL^2Conv} Let $\widetilde{C}$ denote the set of real-valued continuous functions $\varphi(x,y)$ on $(\R^2)^2$ that are bounded in absolute value by a constant multiple of  $ \textup{exp}\big\{-\frac{1}{n}|x|+n|y|\big\} $ for some $n\in \mathbb{N}$.  In Section~\ref{SubsecSpecSpace} we consider a topology for measures on $(\R^2)^2$ such that $\mu_j\rightarrow \mu_{\infty}$  provided that $\mu_j(\varphi)\rightarrow \mu_{\infty}(\varphi)$ for all  $\varphi\in \widetilde{C}$, which is stronger than the vague topology.  We can strengthen  Proposition~\ref{PropChapman} to the statement that
  $\mathscr{Z}^{\mathsmaller{\vartheta}}_{r,s}  \,\mathlarger{\mathlarger{\bullet}}_{\varsigma}\,  \mathscr{Z}^{\mathsmaller{\vartheta}}_{s,t}(\varphi)$ converges to  $\mathscr{Z}^{\mathsmaller{\vartheta}}_{r,t}(\varphi)$ in $L^2$  for every $\varphi\in \widetilde{C}$.
 \end{remark}

\subsection{Multi-interval extension of the 2d SHF}\label{SubsectMTSHF}

 As before, let $X$ and $X'$ be Polish spaces and  $\mu_1$, $\mu_2$  be $\sigma$-finite measures on  $X\times \R^n$ and $ \R^n\times X'$,  respectively.  For $\varsigma>0$  let $\mu_1\, \mathlarger{\mathlarger{\circ}}_{\varsigma} \mu_2  $ denote the measure on $X\times\R^n \times X'$ given by
\begin{align}\label{StarOperation}
  \mu_1 \,\mathlarger{\mathlarger{\circ}}_{\varsigma}\, \mu_2 (dx,da,dx')\,=\,&\int_{ b\in \R^n } \, \mu_1(dx,da)\,g_{\varsigma}(a-b)\, \mu_2(db,dx')\,.
\end{align}
The projection of $ \mu_1 \mathlarger{\mathlarger{\circ}}_{\varsigma}  \mu_2 $  on $X\times X'$ (by integrating out $a\in \R^n$) yields the operation~(\ref{MeasComb}): 
\[
 \mu_1 \,\mathlarger{\mathlarger{\bullet}}_{\varsigma}\, \mu_2(dx,dx')\,=\, 
 \mu_1 \,\mathlarger{\mathlarger{\circ}}_{\varsigma}\, \mu_2 \big(dx,\R^n,dx'\big)  \, .
 \]
When $\mu_2$ has a disintegration with respect to Lebesgue measure on $\R^n$, that is $\mu_2(db,dx')=db\,\dot{\mu}_2(b,dx')$, for a kernel $\dot{\mu}_2$ from  $\R^n$ to $X'$, we define 
\[
\mu_1 \,\mathlarger{\mathlarger{\circ}} \, \mu_2(dx,da,dx')\,=\,\mu_1(dx,da)\, \dot{\mu}_2(a,dx')\,.
\]

If  $\mu_1,\ldots, \mu_m$ are $\sigma$-finite measures on $\R^n\times \R^n$,  then we can define a  measure on  $(\R^n)^{m+1} $ through iterative use of the operation $\mathlarger{\mathlarger{\circ}}_{\varsigma}$, which is associative:
\[
\mu_1 \,\mathlarger{\mathlarger{\circ}}_{\varsigma}\,\mu_2 \,\mathlarger{\mathlarger{\circ}}_{\varsigma} \, \cdots \,  \mathlarger{\mathlarger{\circ}}_{\varsigma}\, \mu_m \,.
\]
The proofs of the following theorem and the propositions below are placed in Section~\ref{SecGenSHF}.
\begin{theorem}[Multi-interval Extension of the 2d SHF] \label{ThmSHFExtension} Fix $\vartheta\in \R$.  Let $\mathscr{Z}^{\mathsmaller{\vartheta}}$  be an $\mathcal{M}^S$-valued random element with the limit distribution in Lemma~\ref{LemmaSubseqLimit}, defined on a probability space $(\Omega, \mathcal{F},\mathbb{P})$.   Given a partition $P=\{ t_0 , t_1 ,\ldots ,t_m\}$ of the interval $[s,t]$, the $ \mathcal{M}_{m }$-valued random element
\begin{align}\label{MartMulti}
\mathscr{Z}^{\mathsmaller{\vartheta},\varsigma}_{P}\,:=\,\mathscr{Z}^{\mathsmaller{\vartheta}}_{t_0,t_1} \,\mathlarger{\mathlarger{\circ}}_{\varsigma}\,\mathscr{Z}^{\mathsmaller{\vartheta}}_{t_1,t_2} \,\mathlarger{\mathlarger{\circ}}_{\varsigma} \, \cdots   \,\mathlarger{\mathlarger{\circ}}_{\varsigma}\, \mathscr{Z}^{\mathsmaller{\vartheta}}_{t_{m-1},t_m} 
\end{align}
converges vaguely in $L^2$ to a limit $ \mathscr{Z}^{\mathsmaller{\vartheta}}_{P} $ as $\varsigma\rightarrow 0$.
\end{theorem}
 Define $\mathbf{U}_{P}:=\mathbf{U}_{t_1-t_0}\mathlarger{\mathlarger{\circ}}\cdots \mathlarger{\mathlarger{\circ}}\,\mathbf{U}_{t_m-t_{m-1}} $ and $\mathbf{Q}_P^{\mathsmaller{\vartheta}}:=\mathbf{Q}_{t_1-t_0}^{\mathsmaller{\vartheta}}\mathlarger{\mathlarger{\circ}}\cdots \mathlarger{\mathlarger{\circ}}\,\mathbf{Q}_{t_m-t_{m-1}}^{\mathsmaller{\vartheta}} $.  In other terms, $\mathbf{U}_{P}$ is the 
  measure on $(\R^2)^{m+1}$ with Lebesgue density
\begin{align*}
\text{}\hspace{.5cm} \mathbf{\dot{U}}_{P}\big(x_0,\,\ldots\,, x_m\big)\,:=\,\prod_{j=1}^m \,g_{t_j-t_{j-1}}(x_j-x_{j-1}) \,,
\end{align*}
and $\mathbf{Q}_P^{\mathsmaller{\vartheta}}$ is the measure on $ (\R^4)^{m+1}\equiv (\R^2\times \R^2)^{m+1}$ with Lebesgue density
\begin{align}\label{QDef}
\mathbf{\dot{Q}}_P^{\mathsmaller{\vartheta}} \big(x_0, x'_0; \,\ldots\,; x_m, x'_m\big)\,:=\,&\,\prod_{j=1}^m  \,g_{t_j-t_{j-1}}\bigg(\frac{x_j+x'_j}{\sqrt{2}}-\frac{x_{j-1}+x'_{j-1}}{\sqrt{2}} 
  \bigg)\nonumber \\  &\,\,\,\,\cdot\,P_{t_j-t_{j-1}}^{\mathsmaller{\vartheta}}\bigg(\frac{x_{j-1}-x'_{j-1}}{\sqrt{2}} ,\,\frac{x_j-x'_j}{\sqrt{2}}  \bigg)\,.
\end{align}
We also put $\mathbf{K}_P^{\mathsmaller{\vartheta}}:=\mathbf{Q}_P^{\mathsmaller{\vartheta}}-\mathbf{U}_P^2$.  The next proposition lists some elementary properties of the multi-interval 2d SHF.  The consistency property in (v) generalizes the Chapman-Kolmogorov relation in Proposition~\ref{PropChapman} and is crucial for the construction of the continuum polymer measures presented in the next subsection.
\begin{proposition}[Properties of the Multi-interval 2d SHF]\label{PropSHFExtensionProp} Fix $\vartheta\in \R$.  For a partition $P=\{t_0,\ldots, t_m\}$ of $[s,t]$, let  $ \mathscr{Z}^{\mathsmaller{\vartheta}}_{P} $ be defined as in Theorem~\ref{ThmSHFExtension}.  Then (i)--(vii) below hold.  
\begin{enumerate}[(i)]

\item $\mathscr{Z}^{\mathsmaller{\vartheta}}_{P}$ has first and second moments 
$\mathbf{U}_P$ and $\mathbf{Q}_P^{\mathsmaller{\vartheta}}$, respectively.

\item $\mathscr{Z}^{\mathsmaller{\vartheta}}_{P}$ is $\mathscr{F}_{[s,t]}$-measurable.

\item  $\mathscr{Z}^{\mathsmaller{\vartheta}}_{P}$ and $\mathscr{Z}^{\mathsmaller{\vartheta}}_{r+P}$ are equal in distribution for any $r>0$.

\item  $\mathscr{Z}^{\mathsmaller{\vartheta}}_{P}$ and $\mathscr{Z}^{\mathsmaller{\vartheta}}_{P}\circ \tau_a^{-1}$  are equal in distribution for each $a\in \R^2$, where $\tau_a:(\R^2)^{m+1}\rightarrow (\R^2)^{m+1}$ is the shift map $(x_0,\ldots, x_m)\rightarrow (x_0+a,\ldots,x_m+a)$.

\item  If $P'=\{s_0,\ldots, s_{\ell}\}$ is a partition of $[s,t]$ refining $P$, then $\mathscr{Z}^{\mathsmaller{\vartheta}}_{P}$ is a projection of  $\mathscr{Z}^{\mathsmaller{\vartheta}}_{P'}$ in the sense that $\mathscr{Z}^{\mathsmaller{\vartheta}}_{P}=\mathscr{Z}^{\mathsmaller{\vartheta}}_{P'}\circ \Pi_{P',P}^{-1}$ almost surely for the map $\Pi_{P',P}:(\R^2)^{\ell+1}\rightarrow (\R^2)^{m+1}$  sending $(x_{s_0}, x_{s_1},\ldots, x_{s_{\ell}})$ to $(x_{t_0},x_{t_1},\ldots, x_{t_m})$.

\item The map $P\mapsto \mathscr{Z}^{\mathsmaller{\vartheta}}_{P}$ is continuous, where $P$ is regarded as an element of the $(m+1)$-dimensional time simplex on $[0,\infty)$ and the image space is equipped with the topology of vague convergence in $L^2$.

\item  If $P'$ is a partition of $[t,u]$, then $\mathscr{Z}^{\mathsmaller{\vartheta}}_{P} \,\mathlarger{\mathlarger{\circ}}_{\varsigma}\, \mathscr{Z}^{\mathsmaller{\vartheta}}_{P'} $ converges vaguely in $L^2$ to $\mathscr{Z}^{\mathsmaller{\vartheta}}_{P\cup P'}$ as $\varsigma\rightarrow 0$.

\end{enumerate}

\end{proposition}

For $P\subset [0,\infty)$ and $0\leq s<t$, define 
\begin{align*}
P_{\leq s}\,:=\,   \big(P\cap [0,s]\big)\cup \{ s \}\,, \hspace{.5cm} P_{\geq t}\,:=\, \big(P\cap [t,\infty)\big)\cup \{ t \}\,, \hspace{.5cm} P_{[s,t]}\,:=\, \big(P\cap [s,t]\big)\cup \{ s,t \}\,.
\end{align*}
The following generalizes Propositions~\ref{PropCond} \& \ref{PropCond2} to the multi-interval 2d SHF.
\begin{proposition} \label{PropCondGen} Fix $\vartheta\in \R$ and a partition $P$ of the interval $[r,u]$.  Let  $\mathscr{Z}^{\mathsmaller{\vartheta}}_P$ be  the $ \mathcal{M}_{m} $-valued random element defined in Theorem~\ref{ThmSHFExtension}.  Given  $s,t\in [r,u]  $  with  $s<t$ and $s,t\notin P$,
we have
\[
\mathbb{E}\Big[\,\mathscr{Z}^{\mathsmaller{\vartheta}}_{P}\,\Big|\, \mathscr{F}_{[0,s]\cup [t,\infty) } \, \Big]    \,=\, \mathscr{Z}^{\mathsmaller{\vartheta}}_{P_{\leq s}}  \,\mathlarger{\mathlarger{\bullet}}\,\mathbf{U}_{P_{[s,t]}}\,\mathlarger{\mathlarger{\bullet}}\,\mathscr{Z}^{\mathsmaller{\vartheta}}_{P_{\geq t} } \,,
\]
and the second moment of the difference between $\mathscr{Z}^{\mathsmaller{\vartheta}}_{P}$ and $\mathscr{Z}^{\mathsmaller{\vartheta}}_{P_{\leq s}}  \,\mathlarger{\mathlarger{\bullet}}\,\mathbf{U}_{P_{[s,t]}}\,\mathlarger{\mathlarger{\bullet}}\,\mathscr{Z}^{\mathsmaller{\vartheta}}_{P_{\geq t} }$ has the form
\begin{align}\label{QKQ2}
\mathbb{E}\Big[\,\big( \mathscr{Z}^{\mathsmaller{\vartheta}}_{P}  \,-\, \mathscr{Z}^{\mathsmaller{\vartheta}}_{P_{\leq s}}  \,\mathlarger{\mathlarger{\bullet}}\,\mathbf{U}_{P_{[s,t]}}\,\mathlarger{\mathlarger{\bullet}}\,\mathscr{Z}^{\mathsmaller{\vartheta}}_{P_{\geq t} }\big)^2\,\Big]\,=\,\mathbf{Q}_{P_{\leq s}}^{\mathsmaller{\vartheta}}\,\mathlarger{\mathlarger{\bullet}}\,\mathbf{K}_{P_{[s,t]}}^{\mathsmaller{\vartheta}}\,\mathlarger{\mathlarger{\bullet}}\,\mathbf{Q}^{\mathsmaller{\vartheta}}_{P_{\geq t} }\,.
\end{align}
If $s\in P$ or $t\in P$, the above equalities hold with the corresponding $\mathlarger{\mathlarger{\bullet}}$ replaced by $\mathlarger{\mathlarger{\circ}}$.
\end{proposition}
The form of the right side of~(\ref{QKQ2}) is unimportant beyond having projection~(\ref{QKQ}), which vanishes vaguely as $t\searrow s$.

\subsection{The associated continuum  polymer measures}\label{SubsectCPM}

 Recall that we define $\boldsymbol{\Upsilon}_{[s,t]}:=C\big( [s,t],\R^2 \big)$ and   equate  $\boldsymbol{\Upsilon}_{[s,t]}^2$ with $C\big( [s,t], (\R^{2})^2 \big)$.  Let $ \Lambda_{[s,t]}$ denote the set of partitions on the interval $[s,t]$.  For  $P=\{ t_0,t_1,\ldots, t_m\}\in \Lambda_{[s,t]}$, let  $\Pi_{P}: \boldsymbol{\Upsilon}_{[s,t]}\rightarrow (\R^{2})^{m+1}$ be  the $(m+1)$-point evaluation map 
\[
\Pi_{P}(p)\,= \,\big(p(t_0), p(t_1)  ,\ldots,  p(t_m) \big) \,, \hspace{1.5cm} p\in \boldsymbol{\Upsilon}_{[s,t]}\,.
\]
  Let $\mathbf{U}_{[s,t]}$ denote the  measure on $\boldsymbol{\Upsilon}_{[s,t]}$   satisfying $\mathbf{U}_{[s,t]}\circ \Pi_{P}^{-1}= \mathbf{U}_{P}$ for every $P\in \Lambda_{[s,t]}$.  In other terms,  $\mathbf{U}_{[s,t]}$ 
is the 2d Wiener measure with initial position having Lebesgue  measure. 
For a partition $P=\{t_0,\ldots, t_m\}$ of $[s,t]$, let $\mathbf{V}^{\mathsmaller{\vartheta}}_{P} $ denote the measure on $(\R^2)^{m+1}$ with Lebesgue density
\begin{align*}
\mathbf{\dot{V}}^{\mathsmaller{\vartheta}}_{P}(x_0,\ldots, x_m) \,=\,\prod_{j=1}^m\, P_{t_j-t_{j-1}}^{\mathsmaller{\vartheta}} (x_{j-1},x_j)\,, \hspace{1cm}x_j\in \R^2\,.
\end{align*}
There exists  a unique  measure $\mathbf{V}^{\mathsmaller{\vartheta}}_{[s,t]} $ on  $\boldsymbol{\Upsilon}_{[s,t]}$ satisfying $\mathbf{V}^{\mathsmaller{\vartheta}}_{[s,t]}\circ \Pi_P^{-1}= \mathbf{V}^{\mathsmaller{\vartheta}}_{P}$ for all $P\in \Lambda_{[s,t]}$; see the discussion in Section~\ref{SubsectionSM}.  Define the map $R:\boldsymbol{\Upsilon}_{[s,t]}^2\rightarrow \boldsymbol{\Upsilon}_{[s,t]}^2$ by $R(p,q)=\big(\frac{p+q}{\sqrt{2}}, \frac{p-q}{\sqrt{2}} \big) $.  Let $\mathbf{Q
}^{\mathsmaller{\vartheta}}_{[s,t]}$ be the  measure on $\boldsymbol{\Upsilon}_{[s,t]}^2$ given by pushing forward  the product measure   $\mathbf{U}_{[s,t]}\times \mathbf{V}^{\mathsmaller{\vartheta}}_{[s,t]} $ by $R$: 
\begin{align}\label{QUV}
\mathbf{Q
}^{\mathsmaller{\vartheta}}_{[s,t]}\,:=\, \big( \mathbf{U}_{[s,t]}\times \mathbf{V}^{\mathsmaller{\vartheta}}_{[s,t]} \big)\circ R^{-1}    \,.
\end{align}
Thus $\mathbf{Q}_{[s,t]}^{\mathsmaller{\vartheta}}$ satisfies $\mathbf{Q}_{[s,t]}^{\mathsmaller{\vartheta}}\circ (\Pi_{P}^2)^{-1}\,=\, \mathbf{Q}_{P}^{\mathsmaller{\vartheta}}$ for all $P\in \Lambda_{[s,t]}$, where $\mathbf{Q}_{P}^{\mathsmaller{\vartheta}}$ is defined in~(\ref{QDef}) and $\Pi_{P}^2: \boldsymbol{\Upsilon}_{[s,t]}^2\rightarrow (\R^{2}\times\R^2)^{m+1}$ is the evaluation map 
\[
\Pi_{P}^2(p,q)\,= \,\big(p(t_0), q(t_0); p(t_1) ,q(t_1) ;\ldots;  p(t_m),q(t_m) \big) \,, \hspace{1.5cm} p,q\in \boldsymbol{\Upsilon}_{[s,t]}\,.
\] 
Also define $\mathbf{K}_{[s,t]}^{\mathsmaller{\vartheta}}:=\mathbf{Q}_{[s,t]}^{\mathsmaller{\vartheta}}- \mathbf{U}_{[s,t]}^2$.  The proofs of the theorem and three propositions below are in Section~\ref{SecCPM}.
\begin{theorem}[Continuum Polymer Measure]\label{ThmPathMeasConst} Fix  $\vartheta\in \R$. Let the family of random measures  $\{\mathscr{Z}^{\mathsmaller{\vartheta}}_{P}\}_{P\in\Lambda_{[s,t]} }$ be defined as in Theorem~\ref{ThmSHFExtension}.    There exists a $\mathcal{M}(\boldsymbol{\Upsilon}_{[s,t]})$-valued random element $\mathscr{Z}_{[s,t]}^{\mathsmaller{\vartheta}}$ such that for any $P\in \Lambda_{[s,t]}$ 
\begin{align}\label{ZProject}
  \mathscr{Z}_{[s,t]}^{\mathsmaller{\vartheta}}\circ \Pi_{P}^{-1}  \, =\,   \mathscr{Z}^{\mathsmaller{\vartheta}}_{P} \hspace{.5cm}\mathbb{P}\text{-a.e.}  
\end{align}
\end{theorem}
The   proposition below follows closely from the defining property~(\ref{ZProject}) and  (i)--(iv) in Proposition~\ref{PropSHFExtensionProp}.
\begin{proposition}\label{PropPathMeasProp} Fix  $\vartheta\in \R$.    If $\mathscr{Z}_{[s,t]}^{\mathsmaller{\vartheta}}$ is the  $\mathcal{M}(\boldsymbol{\Upsilon}_{[s,t]})$-valued random element in Theorem~\ref{ThmPathMeasConst}, then (i)--(iv) below hold.  
\begin{enumerate}[(i)]

\item $\mathscr{Z}^{\mathsmaller{\vartheta}}_{[s,t]}$ has first and second moments 
$\mathbf{U}_{[s,t]}$ and $\mathbf{Q}_{[s,t]}^{\mathsmaller{\vartheta}}$, respectively.

\item $\mathscr{Z}^{\mathsmaller{\vartheta}}_{[s,t]}$ is $\mathscr{F}_{[s,t]}$-measurable.

\item  $\mathscr{Z}^{\mathsmaller{\vartheta}}_{[s,t]}$ 
 and $\mathscr{Z}^{\mathsmaller{\vartheta}}_{[u+s,u+t]}\circ \mathbf{s}_u^{-1}$   are equal in distribution for any $u>0$, where $\mathbf{s}_u: \boldsymbol{\Upsilon}_{[u+s,u+t]}\rightarrow \boldsymbol{\Upsilon}_{[s,t]}$ is defined through $(\mathbf{s}_u p)(r)=p(r+u)$.

\item $\mathscr{Z}^{\mathsmaller{\vartheta}}_{[s,t]}$  and $\mathscr{Z}^{\mathsmaller{\vartheta}}_{[s,t]}\circ \tau_a^{-1}$  are equal in distribution for each $a\in \R^2$, where $\tau_a:\boldsymbol{\Upsilon}_{[s,t]}\rightarrow \boldsymbol{\Upsilon}_{[s,t]}$ is defined through $(\tau_a p)(r)=p(r)+a$ .

\end{enumerate}

\end{proposition}

Fix  $0\leq s<t$ and $x,y\in \R^2$.  Let $\mathbf{P}_{[s,t]}^{x,y}$ denote the 2d Brownian bridge measure on the time interval $[s,t]$ for paths with $p(s)=x$ and $p(t)=y$.  In other terms, $\mathbf{P}_{[s,t]}^{x,y}$ is the probability measure on $\boldsymbol{\Upsilon}_{[s,t]}$ such that for any partition $P=\{t_0,\ldots, t_m\}$ of $[s,t]$
\[
\mathbf{P}_{[s,t]}^{x,y}\circ \Pi_{P}^{-1} \big(dx_0, dx_1,\ldots,dx_{m}\big) \,=\,  \frac{ \prod_{1}^m\,g_{t_j-t_{j-1}}(x_j-x_{j-1}) }{g_{t-s}(y-x)  }\,\delta_x(dx_0) \,  dx_1\,\cdots \,dx_{m-1}\,\delta_y(dx_m)\,.
\]
 Let $\mu_1$ and $\mu_2$ be measures on $\boldsymbol{\Upsilon}_{[r,s]}$ and $\boldsymbol{\Upsilon}_{[t,u]}$ for $r\in [0,s)$, $u\in (t,\infty)$. Consider the  measure on $\boldsymbol{\Upsilon}_{[r,s]}\times \boldsymbol{\Upsilon}_{[s,t]}\times \boldsymbol{\Upsilon}_{[t,u]}$ given by
\begin{align}\label{ToPullBack}
\lambda_{\mu_1,\mu_2}\big(dp_1,dp_2,dp_3\big)=  \mu_1(dp_{1})  \,\mu_2(dp_{3})\,\mathbf{P}_{[s,t]}^{p_1(s),p_3(t)}(dp_{2})\,, \hspace{.2cm}p_1\in \boldsymbol{\Upsilon}_{[r,s]}\,,\,p_2\in  \boldsymbol{\Upsilon}_{[s,t]}\,, \,p_3\in  \boldsymbol{\Upsilon}_{[t,u]}\,.
\end{align}
We then define $\mu_1 \odot   \mu_2$ as the measure on  $\boldsymbol{\Upsilon}_{[r,u]}$ that results from  pulling back the measure~(\ref{ToPullBack}) by the injective map 
\[
p\in \boldsymbol{\Upsilon}_{[r,u]} \hspace{.4cm}
 \longmapsto  \hspace{.4cm} \big(p_{[r,s]}, p_{[s,t]},p_{[t,u]} \big)\in \boldsymbol{\Upsilon}_{[r,s]}\times \boldsymbol{\Upsilon}_{[s,t]}\times \boldsymbol{\Upsilon}_{[t,u]} \,,
 \]
where $p_A$ denotes the restriction of $p$ to $A\subset [r,u]$.  We can characterize $\mu_1 \odot   \mu_2$ as the unique measure on $\boldsymbol{\Upsilon}_{[r,u]}$  satisfying
\[
\big(\mu_1 \odot   \mu_2\big)\circ \Pi_{P}^{-1}\,=\,\big(\mu_1 \circ \Pi_{P_{\leq s}}^{-1}\big)\,\mathlarger{\mathlarger{\circ}}\,\mathbf{U}_{P_{[s,t]}}\,\mathlarger{\mathlarger{\circ}}\,\big(\mu_2 \circ \Pi_{P_{\geq t}}^{-1}\big)
\]
for every $P\in \Lambda_{[r,u]}$ with $s,t\in P$.  The next proposition follows from~(\ref{ZProject}) and  Proposition~\ref{PropCondGen}.
\begin{proposition}\label{PropCondExpLast} Fix  $\vartheta\in \R$ and $0\leq r<s<t<u$. Let the   $\mathcal{M}(\boldsymbol{\Upsilon}_{[r,u]})$-valued random  element  $\mathscr{Z}^{\mathsmaller{\vartheta}}_{[r,u]} $ be defined as in Theorem~\ref{ThmPathMeasConst}.
  Then we have
  \begin{align*}
\mathbb{E}\Big[\,\mathscr{Z}^{\mathsmaller{\vartheta}}_{[r,u]}\,\Big|\,\mathscr{F}_{[0,s]\cup [t,\infty) } \, \Big]\,=\, \mathscr{Z}^{\mathsmaller{\vartheta}}_{[r,s]}\odot   \mathscr{Z}^{\mathsmaller{\vartheta}}_{[t,u]}\,.
  \end{align*}
 The second moment of the difference between  $\mathscr{Z}^{\mathsmaller{\vartheta}}_{[r,u]}$ and $\mathscr{Z}^{\mathsmaller{\vartheta}}_{[r,s]}\odot   \mathscr{Z}^{\mathsmaller{\vartheta}}_{[t,u]}$ has pushforward by $\Pi_{r,u}^2$ 
 \begin{align}\label{FinalVariance}
\mathbb{E}\Big[\,\big( \mathscr{Z}^{\mathsmaller{\vartheta}}_{[r,u]} \,-\, \mathscr{Z}^{\mathsmaller{\vartheta}}_{[r,s]}\odot   \mathscr{Z}^{\mathsmaller{\vartheta}}_{[t,u]}\big)^2\,\Big]\circ \big(\Pi_{r,u}^2\big)^{-1}\,=\,\mathbf{Q}_{s-r}^{\mathsmaller{\vartheta}}\,\mathlarger{\mathlarger{\bullet}}\,\mathbf{K}_{t-s}^{\mathsmaller{\vartheta}}\,\mathlarger{\mathlarger{\bullet}}\,\mathbf{Q}^{\mathsmaller{\vartheta}}_{u-t}\,.
\end{align}
\end{proposition}

Next we  extend  (vii) of Proposition~\ref{PropSHFExtensionProp} to the level of path measures by formulating how $\mathscr{Z}^{\mathsmaller{\vartheta}}_{[r,t]}$ can be approximated through a limiting procedure involving $\mathscr{Z}^{\mathsmaller{\vartheta}}_{[r,s]}$  and $\mathscr{Z}^{\mathsmaller{\vartheta}}_{[s,t]}$ for  $s\in (r,t)$.
The approximating measures that we have in mind are defined on the set of  c\`adl\`ag paths  over the time interval $[r,t]$ that have at most one  discontinuity,  occurring at time $s$. Since $\mathscr{Z}^{\mathsmaller{\vartheta}}_{[r,t]}$ is defined as a random measure on the space $\boldsymbol{\Upsilon}_{[r,t ]}$ of continuous paths,  it will be convenient to express the approximation scheme in terms of the pushforward of $\mathscr{Z}^{\mathsmaller{\vartheta}}_{[r,t]}$  
by the  map $  \mathlarger{\iota}_{r,s,t}:\boldsymbol{\Upsilon}_{[r,t ]}\rightarrow \boldsymbol{\Upsilon}_{[r,s ]}\times \boldsymbol{\Upsilon}_{[s,t ]}$ given by $ \mathlarger{\iota}_{r,s,t}(p)=(p_{[r,s]}, p_{[s,t]}  )$.  Define the maps $\Xi_{s,t,+}:\boldsymbol{\Upsilon}_{[s,t ]}\rightarrow \boldsymbol{\Upsilon}_{[s,t]}\times \R^2$  and $\Xi_{s,t,-}:   \boldsymbol{\Upsilon}_{[s,t]}\rightarrow \R^2\times \boldsymbol{\Upsilon}_{[s,t]}$  by
\[
\Xi_{s,t,+}(p)\,:=\,\big( p , p(t) \big)
\hspace{1cm}\text{and}\hspace{1cm} \Xi_{s,t,-}(p)\,:=\,\big(p(s), p  \big)  \,.
\]
If $\mu_1\in \mathcal{M}(\boldsymbol{\Upsilon}_{[r,s]})$ and $\mu_2\in \mathcal{M}(\boldsymbol{\Upsilon}_{[s,t]})$, then $\big(\mu_1\circ \Xi_{r,s,+}^{-1} \big)\,\mathlarger{\mathlarger{\bullet}}_{\varsigma}\,\big(\mu_2\circ \Xi_{s,t,-}^{-1} \big)$ for $\varsigma>0$ is the  measure on $\boldsymbol{\Upsilon}_{[r,s]}\times \boldsymbol{\Upsilon}_{[s,t]}$ satisfying
\[
\big(\mu_1\circ \Xi_{r,s,+}^{-1} \big)\,\mathlarger{\mathlarger{\bullet}}_{\varsigma}\,\big(\mu_2\circ \Xi_{s,t,-}^{-1} \big)\,(dp,dq)\,=\, \mu_1(dp)\,g_{\varsigma}\big(p(s)-q(s)\big) \,\mu_2(dq)\,, \hspace{.7cm}p\in \boldsymbol{\Upsilon}_{[r,s]}\,,\,\,q\in \boldsymbol{\Upsilon}_{[s,t]}\,.
\]
\begin{proposition}\label{PropConn} Fix  $\vartheta\in \R$ and $0\leq r<s<t$.   The    $\mathcal{M}\big(\boldsymbol{\Upsilon}_{[r,s]}\times \boldsymbol{\Upsilon}_{[s,t]}  \big)$-valued random  element $\big(\mathscr{Z}^{\mathsmaller{\vartheta}}_{[r,s]}\circ \Xi_{r,s,+}^{-1} \big)\,\mathlarger{\mathlarger{\bullet}}_{\varsigma}\,\big(\mathscr{Z}^{\mathsmaller{\vartheta}}_{[s,t]}\circ \Xi_{s,t,-}^{-1} \big)$ converges vaguely in probability to $\mathscr{Z}^{\mathsmaller{\vartheta}}_{[r,t]}\circ\mathlarger{\iota}_{r,s,t}^{-1} $ as $\varsigma\rightarrow 0$.
\end{proposition}

\subsection{The second moment of $\mathscr{Z}^{\mathsmaller{\vartheta}}_{[s,t]}$ and path intersection time}\label{SubsectCPMSM}

The theorems in this subsection concerning the properties of the second moment, $\mathbf{Q}^{\mathsmaller{\vartheta}}_{[s,t]}$,  of the random path measure $\mathscr{Z}^{\mathsmaller{\vartheta}}_{[s,t]}$  follow from results in~\cite{CM}, as we discuss further in Section~\ref{SubsectionSM}. \vspace{.2cm}

Let $\mathcal{I}$ denote the set of intersecting pairs of paths $(p,q)\in\boldsymbol{\Upsilon}_{[s,t]}^2$, i.e., such that $p(r)=q(r)$ for some $r\in [s,t]$.   Note that the product Wiener measure  $\mathbf{U}_{[s,t]}^2 $ is supported on  $ \mathcal{I}^c$ since two independent 2d Brownian motions starting from different positions  almost surely will not intersect.  For $\varepsilon>0$ define $\mathbf{I}_{[s,t]}^{\varepsilon}:\boldsymbol{\Upsilon}_{[s,t]}^2 \rightarrow [0,\infty)$  by
 $$\mathbf{I}_{[s,t]}^{\varepsilon}(p,q) \,:=\,\frac{1}{2\varepsilon^2\log^2\frac{1}{\varepsilon}   } \,\textup{meas}\bigg(  \,\bigg\{\, a\in [s,t] \,:\,  \bigg|\frac{p(a)-q(a)}{\sqrt{2}}\bigg|\leq \varepsilon   \, \bigg\}\,\bigg)\,, \hspace{.9cm}p,q\in \boldsymbol{\Upsilon}_{[s,t]}\,,
 $$
where $\text{meas}(A)$ denotes the Lebesgue measure of $A\subset \R$.  Then $\mathbf{I}_{[s,t]}^{\varepsilon}(p,q)=0$ for small enough $\varepsilon$ when $(p,q)\in \mathcal{I}^c$.
We refer to the limit of $\mathbf{I}_{[s,t]}^{\varepsilon}(p,q)$  as $\varepsilon\rightarrow 0$ in (ii) of the theorem below as the \textit{intersection time} between the paths $p$ and $q$  over the time interval $[s,t]$. 
\begin{theorem}[Properties of the Second Moment of $\mathscr{Z}^{\mathsmaller{\vartheta}}_{[s,t]}$]\label{ThmRNDer} Fix  $\vartheta\in \R$   and $0\leq s<t$.  
\begin{enumerate}[(i)]
\item The  Lebesgue decomposition of $\mathbf{Q}^{\mathsmaller{\vartheta}}_{[s,t]}$ with respect to   $\mathbf{U}_{[s,t]}^2$ has components $\mathbf{U}_{[s,t]}^2$ and $\mathbf{K}^{\mathsmaller{\vartheta}}_{[s,t]}$, which are  supported on $\mathcal{I}^c$ and $\mathcal{I}$, respectively.

\item There exists a Borel measurable function $\mathbf{I}_{[s,t]}$ on $\boldsymbol{\Upsilon}_{[s,t]}^2$ with  $\mathbf{I}_{[s,t]}=0$ on $\mathcal{I}^c$    such that $\mathbf{I}_{[s,t]}^{\varepsilon} \rightarrow \mathbf{I}_{[s,t]}$ in  $L^1_{\textup{loc} }\big(\mathbf{Q}^{\mathsmaller{\vartheta}}_{[s,t]} \big)$ as $\varepsilon\rightarrow 0$.

\item For any $\vartheta'\in \R$ the measure $\mathbf{Q}^{\mathsmaller{\vartheta}'}_{[s,t]}$ is absolutely continuous with respect to $\mathbf{Q}^{\mathsmaller{\vartheta}}_{[s,t]}$ with Radon-Nikodym derivative \vspace{-.2cm}
 \begin{align}\label{RND} \frac{d\mathbf{Q}^{\mathsmaller{\vartheta}'}_{[s,t]}}{d\mathbf{Q}^{\mathsmaller{\vartheta}}_{[s,t]}}\,=\,\textup{exp}\big\{\,(\mathsmaller{\vartheta}'-\mathsmaller{\vartheta}) \, \mathbf{I}_{[s,t]} \, \big\} \,.
 \end{align}

\item   $\mathbf{I}_{[s,t]}$ is $\mathbf{Q}^{\mathsmaller{\vartheta}}_{[s,t]}$-almost everywhere positive on  $\mathcal{I}$.

\end{enumerate}

\end{theorem}

We can use Theorem~\ref{ThmRNDer} to make some casual observations about the random  product measure $\big(\mathscr{Z}_{[s,t]}^{\mathsmaller{\vartheta}}\big)^2$, which we do in the next few remarks.  
\begin{remark}
Given $a>0$ the formula~(\ref{RND}) implies that the function
\[ (p,q)\in \boldsymbol{\Upsilon}_{[s,t]}^2 \hspace{.7cm}\longrightarrow \hspace{.7cm} \textup{exp}\big\{\,a\,\mathbf{I}_{[s,t]}(p,q) \,\big\}
\]
is almost surely locally integrable under  $\big(\mathscr{Z}_{[s,t]}^{\mathsmaller{\vartheta}}\big)^2$ since for a bounded set $A\subset \boldsymbol{\Upsilon}_{[s,t]}^2 $
\begin{align*}
\mathbb{E}\bigg[\,\int_{A} \,\textup{exp}\big\{\,a\,\mathbf{I}_{[s,t]}\,\big\}\,\big(\mathscr{Z}_{[s,t]}^{\mathsmaller{\vartheta}}\big)^2\,\bigg]\,=\,\int_{A} \,\textup{exp}\big\{\,a\,\mathbf{I}_{[s,t]}\,\big\}\,\mathbf{Q}^{\mathsmaller{\vartheta}}_{[s,t]}\,\stackrel{(\ref{RND})  }{=}\, \mathbf{Q}^{\mathsmaller{\vartheta}+a}_{[s,t]}(A)\,<\,\infty\,,
\end{align*}
where the first equality uses that  $\mathbf{Q}^{\mathsmaller{\vartheta}}_{[s,t]}$ is the second moment of $\mathscr{Z}_{[s,t]}^{\mathsmaller{\vartheta}}$. 
\end{remark}
\begin{remark} \label{CorRealizations} Similarly, parts (ii) \& (iv) of Theorem~\ref{ThmRNDer} imply that the product measure $\big(\mathscr{Z}_{[s,t]}^{\mathsmaller{\vartheta}}\big)^2$ almost surely takes full weight on the set of pairs $(p,q)\in \boldsymbol{\Upsilon}_{[s,t]}^2$ such that either $p$ and $q$ do not intersect or  $\mathbf{I}_{[s,t]}(p,q)>0$.
\end{remark}
\begin{remark}
If $\hat{\mathscr{Z}}_{[s,t]}^{\mathsmaller{\vartheta}}$ is an independent copy of $\mathscr{Z}_{[s,t]}^{\mathsmaller{\vartheta}}$, then the product measure $\mathscr{Z}_{[s,t]}^{\mathsmaller{\vartheta}}\times \hat{\mathscr{Z}}_{[s,t]}^{\mathsmaller{\vartheta}}$ is almost surely supported on $\mathcal{I}^c$ since
\[
\mathbb{E}\big[\,\mathscr{Z}_{[s,t]}^{\mathsmaller{\vartheta}}\times \hat{\mathscr{Z}}_{[s,t]}^{\mathsmaller{\vartheta}}(\mathcal{I})\,\big]\,=\, \mathbb{E}\big[\,\mathscr{Z}_{[s,t]}^{\mathsmaller{\vartheta}}\,\big]\times \mathbb{E}\big[\,\hat{\mathscr{Z}}_{[s,t]}^{\mathsmaller{\vartheta}}\,\big](\mathcal{I})
  \,=\,\mathbf{U}_{[s,t]}^2(\mathcal{I})\,=\,0\,.
\]
Thus, it seems natural to conjecture that beyond realizations of the pair of measures $\mathscr{Z}_{[s,t]}^{\mathsmaller{\vartheta}}$, $\hat{\mathscr{Z}}_{[s,t]}^{\mathsmaller{\vartheta}}$ merely being mutually singular with probability one, they are supported on sets of paths that do not even intersect.
\end{remark}

  For paths $p,q\in  \boldsymbol{\Upsilon}_{[s,t]} $ define the set of intersection times:   
  \[
 I_{p,q}\,:=\,\big\{\,a\in [s,t]\,:\,p(a)=q(a) \,  \big\}\,.
  \]  
The set  $I_{p,q}$ has Hausdorff dimension zero for $\mathbf{Q}^{\mathsmaller{\vartheta}}_{[s,t]}$-almost every pair $(p,q)\in \boldsymbol{\Upsilon}_{[s,t]}^2$, although    $I_{p,q}$ is $\mathbf{Q}^{\mathsmaller{\vartheta}}_{[s,t]}$-almost surely uncountable in the event $\mathcal{I}=\{(p,q)\in \boldsymbol{\Upsilon}_{[s,t]}^2 : I_{p,q}\neq \emptyset \}$.  These facts are implied by a more refined Hausdorff-type analysis, appropriate for  Hausdorff dimension zero sets.  For fixed $\mathfrak{h}>0$, define the Hausdorff-type outer measure  of $A\subset \R$ by
\begin{align*}
\lim_{\delta\searrow 0}  H_{\mathfrak{h},\delta}(A)\,:=\, H_{\mathfrak{h}}(A)\,,\hspace{1cm}\text{where}\hspace{1cm}   H_{\mathfrak{h},\delta}(A)\,:=\,\inf_{ \substack{A\subset \cup_n I_n  \\ |I_n|\leq \delta  } }\sum_{ n } \frac{1}{ \big|\log\big(\frac{1}{|I_n|}\big) \big|^{\mathfrak{h}}  }\,,
\end{align*}
where the infimum is over all coverings of $A$ by intervals of length less than $\delta\in (0,1)$.  Define the \textit{log-Hausdorff exponent} of $A$ as 
the infimum of the set of $\mathfrak{h}$ such that $H_{\mathfrak{h}}(A)<\infty$.  
\begin{theorem}\label{ThmIntLogHaus}  Fix $\vartheta\in \R$   and $0\leq s<t$.    The set $I_{p,q}$ has log-Hausdorff exponent one for $\mathbf{Q}^{\mathsmaller{\vartheta}}_{[s,t]}$-almost every $(p,q)\in \mathcal{I}$.  Consequently, the same holds for  almost every realization of $\big(\mathscr{Z}_{[s,t]}^{\mathsmaller{\vartheta}}\big)^2$. 
\end{theorem}

\subsection{Consistency with a conditional Gaussian multiplicative chaos structure}\label{SubsectCGMC}
We will briefly describe a natural conjecture for how the distributional law of $\mathscr{Z}_{[s,t]}^{\mathsmaller{\vartheta}'} $ can be constructed as a subcritical Gaussian multiplicative chaos (GMC) with random reference measure $\mathscr{Z}_{[s,t]}^{\mathsmaller{\vartheta}} $ when $\vartheta'>\vartheta$.  A similar result was obtained in~\cite{Clark4} for an analogous  toy model defined on  the so-called \textit{diamond hierarchical lattice}, which has a self-similar graphical structure that enables relatively straightforward computations.  The motivating intuition for that study was exactly the pattern outlined below.\vspace{.2cm}

We will first argue that the random path measure $\mathscr{Z}_{[s,t]}^{\mathsmaller{\vartheta}} $ cannot be constructed as   a GMC in the usual sense, with nonrandom reference measure given by its expectation $\mathbf{U}_{[s,t]}$.  In formal terms, for $\mathscr{Z}_{[s,t]}^{\mathsmaller{\vartheta}}$ to be a GMC with respect to $\mathbf{U}_{[s,t]}$, it would need to be expressible as
\begin{align}\label{ZGMC}
\mathscr{Z}_{[s,t]}^{\mathsmaller{\vartheta}}(dp)\,=\,\textup{exp}\Big\{\,\mathbf{W}(p)-\frac{1}{2}\mathbb{E}\big[\mathbf{W}^2(p)\big]  \, \Big\} \,\mathbf{U}_{[s,t]}(dp)\,,
\end{align}
where $\mathbf{W}=\{\mathbf{W}(p)\}_{p\in \boldsymbol{\Upsilon}_{[s,t]}}$ is a centered Gaussian field with some covariance kernel $T(p,q)=\mathbb{E}\big[ \mathbf{W}(p)\mathbf{W}(q) \big]$.  Here, in general, the individual elements $\mathbf{W}(p)$ do not  denote Gaussian random variables, but rather  $\mathbf{W}$ is understood as a bounded linear map from $L^2( \boldsymbol{\Upsilon}_{[s,t]},\mathbf{U}_{[s,t]} )$ to a Gaussian subspace of $L^2(\Omega, \mathbb{P})$, for which we write
\[
\mathbf{W}\psi\,=\,\int_{\boldsymbol{\Upsilon}_{[s,t]}} \,\mathbf{W}(p)\,\psi(p)\,\mathbf{U}_{[s,t]}(dp)\,,\hspace{1cm}\psi\in L^2( \boldsymbol{\Upsilon}_{[s,t]},\mathbf{U}_{[s,t]} )\,.
\]
Based on~(\ref{ZGMC}), the second moment of $\mathscr{Z}_{[s,t]}^{\mathsmaller{\vartheta}}$ has the form
\[
\mathbf{Q}^{\mathsmaller{\vartheta}}_{[s,t]}(dp,dq)\,=\,\mathbb{E}\Big[\,\mathscr{Z}_{[s,t]}^{\mathsmaller{\vartheta}}(dp) \,\mathscr{Z}_{[s,t]}^{\mathsmaller{\vartheta}}(dq) \,\Big]\,=\,e^{T(p,q)}\,\mathbf{U}_{[s,t]}(dp)\,\mathbf{U}_{[s,t]}(dq)\,,  \hspace{.9cm}p,q\in \boldsymbol{\Upsilon}_{[s,t]}\,;
\]
see~\cite[Lemma 34]{Shamov}.  However, the above contradicts that $\mathbf{Q}^{\mathsmaller{\vartheta}}_{[s,t]}$ is not absolutely continuous with respect to $\mathbf{U}_{[s,t]}^2$, as a consequence of the nontrivial Lebesgue decomposition that $\mathbf{Q}^{\mathsmaller{\vartheta}}_{[s,t]}$ has with respect to $\mathbf{U}_{[s,t]}^2$ in Theorem~\ref{ThmRNDer}.\vspace{.2cm}

Suppose for the purpose of our heuristic computation below that  $\big\{\mathbf{W}_{\mathscr{Z}_{[s,t]}^{\mathsmaller{\vartheta}}}(p)\big\}_{p\in \boldsymbol{\Upsilon}_{[s,t]}}$ is a field that is  Gaussian with correlation kernel $\mathbf{I}_{[s,t]}(p,q)$ when conditioned on $\mathscr{Z}_{[s,t]}^{\mathsmaller{\vartheta}}$.  In other terms, for an appropriate test function $\psi:\boldsymbol{\Upsilon}_{[s,t]}\rightarrow \R $ the random variable formally expressed by
\begin{align}\label{FieldExpr}
\int_{\boldsymbol{\Upsilon}_{[s,t]}} \,\mathbf{W}_{\mathscr{Z}_{[s,t]}^{\mathsmaller{\vartheta}}}(p)\,\psi(p)\,\mathscr{Z}_{[s,t]}^{\mathsmaller{\vartheta}}(dp) 
\end{align}
is a centered Gaussian with variance 
\[
\int_{\boldsymbol{\Upsilon}_{[s,t]}^2} \,\mathbf{I}_{[s,t]}(p,q)\,\psi(p)\,\psi(q)\,\mathscr{Z}_{[s,t]}^{\mathsmaller{\vartheta}}(dp) \,\mathscr{Z}_{[s,t]}^{\mathsmaller{\vartheta}}(dq) \,,
\]
conditional on $\mathscr{Z}_{[s,t]}^{\mathsmaller{\vartheta}}$. Next suppose that $\mathscr{M}_{[s,t]}^{\mathsmaller{\vartheta},\mathsmaller{\vartheta'}}$ is the conditional GMC formally defined by 
\begin{align*}
\mathscr{M}_{[s,t]}^{\mathsmaller{\vartheta},\mathsmaller{\vartheta'}} (dp)
\, =\, \textup{exp}\bigg\{\, \sqrt{\mathsmaller{\vartheta}'-\mathsmaller{\vartheta}}\,\mathbf{W}_{ \mathscr{Z}_{[s,t]}^{\mathsmaller{\vartheta}}}(p)\,-\,\frac{\mathsmaller{\vartheta}'-\mathsmaller{\vartheta}}{2}\,\mathbb{E}\Big[\, \mathbf{W}_{ \mathscr{Z}_{[s,t]}^{\mathsmaller{\vartheta}} }^2(p)   \,\Big|\, \mathscr{Z}_{[s,t]}^{\mathsmaller{\vartheta}} \, \Big]  \,\bigg\}    \,\mathscr{Z}_{[s,t]}^{\mathsmaller{\vartheta}}(dp)\,.
\end{align*}
Then $\mathscr{M}_{[s,t]}^{\mathsmaller{\vartheta},\mathsmaller{\vartheta'}}$ and $\mathscr{Z}_{[s,t]}^{\mathsmaller{\vartheta'}}$  have the same second moment  since
\begin{align*}
\displaystyle \mathbb{E}\Big[  \,\big(\mathscr{M}_{[s,t]}^{\mathsmaller{\vartheta},\mathsmaller{\vartheta'}} \big)^2 \, \Big] \, =\, e^{(\vartheta'-\vartheta)\,\mathbf{I}_{[s,t]} }\, \mathbb{E}\Big[  \,\big(\mathscr{Z}_{[s,t]}^{\mathsmaller{\vartheta}}\big)^2\,  \Big]\,=\, e^{(\vartheta'-\vartheta)\,\mathbf{I}_{[s,t]} }\, \mathbf{Q}^{\mathsmaller{\vartheta}}_{[s,t]} 
  \stackrel{(\ref{RND})  }{=}\mathbf{Q}^{\mathsmaller{\vartheta'}}_{[s,t]}\, \,=\,\mathbb{E}\Big[  \big(\mathscr{Z}_{[s,t]}^{\mathsmaller{\vartheta'}}\big)^2\,  \Big] \,,
  \end{align*}
which  suggests that
$\mathscr{M}_{[s,t]}^{\mathsmaller{\vartheta},\mathsmaller{\vartheta'}}$ and $\mathscr{Z}_{[s,t]}^{\mathsmaller{\vartheta'}}$  are equal in distribution. \vspace{.2cm}

In~\cite{CSZ6} Caravenna, Sun, and Zygouras proved that the projection measure $\overline{\mathscr{Z}}_{s,t}^{\mathsmaller{\vartheta}}(dx)=\mathscr{Z}_{s,t}^{\mathsmaller{\vartheta}}(dx,\R^2)$ is not a GMC by demonstrating that  the form of its first three moments exclude this possibility.   The significance of their result is that it illustrates the non-Gaussian nature of the critical regime, distinguishing it from the subcritical regime of two-dimensional polymer models where log-normals  arise as  distributional limits~\cite{CSZ1,Chatterjee,Gu,CSZ7}.  The above computation isolates where a layer of  GMC structure is to be expected, at least in the relative sense. \vspace{.2cm}

An approach to proving the distributional equality between $\mathscr{M}_{[s,t]}^{\mathsmaller{\vartheta},\mathsmaller{\vartheta'}}$ and $\mathscr{Z}_{[s,t]}^{\mathsmaller{\vartheta'}}$ in analogy to the argument in~\cite{Clark4} would require the following: 
\begin{itemize}
\item Provide a list of properties uniquely characterizing the family of laws $\big\{\mathscr{Z}_{[s,t]}^{\mathsmaller{\vartheta}}\big\}_{0\leq s<t<\infty}$ (for path measure valued random elements).  This would likely follow easily from a list of properties uniquely characterizing the distribution of the 2d SHF $\mathscr{Z}^{\mathsmaller{\vartheta}}$, which is an open question.

\item For $\vartheta<\vartheta'$ construct $\mathscr{M}_{[s,t]}^{\mathsmaller{\vartheta},\mathsmaller{\vartheta'}}$ using tools in~\cite{Shamov}.  Let $\mathbf{I}_{ \mathscr{Z}_{[s,t]}^{\mathsmaller{\vartheta}} }:L^2\big(\boldsymbol{\Upsilon}_{[s,t]}, \mathscr{Z}_{[s,t]}^{\mathsmaller{\vartheta}} \big)\rightarrow L^2\big(\boldsymbol{\Upsilon}_{[s,t]}, \mathscr{Z}_{[s,t]}^{\mathsmaller{\vartheta}} \big)$ denote the linear operator with integral kernel given by the intersection time $\mathbf{I}_{[s,t]}(p,q)$.  The technical aspect arises in constructing a linear operator $Y_{ \mathscr{Z}_{[s,t]}^{\mathsmaller{\vartheta}} }:\ell^2\rightarrow L^2\big(\boldsymbol{\Upsilon}_{[s,t]}, \mathscr{Z}_{[s,t]}^{\mathsmaller{\vartheta}} \big)$ satisfying 
\[
\mathbf{I}_{ \mathscr{Z}_{[s,t]}^{\mathsmaller{\vartheta}} }\,=\,Y_{ \mathscr{Z}_{[s,t]}^{\mathsmaller{\vartheta}} }\,Y_{ \mathscr{Z}_{[s,t]}^{\mathsmaller{\vartheta}} }^*\,.
\]
If $W:\ell^2\rightarrow L^2(\Omega, \mathbb{P})$ is a standard Gaussian random vector independent of $\mathscr{Z}_{[s,t]}^{\mathsmaller{\vartheta}} $, i.e., a linear isometry from $\ell^2$ into a  Gaussian subspace of $L^2(\Omega, \mathbb{P})$ that is independent of $\mathscr{Z}_{[s,t]}^{\mathsmaller{\vartheta}} $, then the field~(\ref{FieldExpr}) is defined by the random linear operator $\mathbf{W}_{\mathscr{Z}_{[s,t]}^{\mathsmaller{\vartheta}}}:=W Y_{ \mathscr{Z}_{[s,t]}^{\mathsmaller{\vartheta}} }^*$.

\item Show that for any $a>0$ the family $\big\{\mathscr{M}_{[s,t]}^{\mathsmaller{\vartheta}-a,\mathsmaller{\vartheta}}\big\}_{0\leq s<t<\infty}$ has the list of properties uniquely characterizing $\big\{\mathscr{Z}_{[s,t]}^{\mathsmaller{\vartheta}}\big\}_{0\leq s<t<\infty}$.

\end{itemize}

\section{Conditional expectations and the Chapman-Kolmogorov relation} \label{SecChapmanProof}
The main goals of this section are to prove Propositions~\ref{PropCond} \&~\ref{PropChapman}, which we do in Sections~\ref{SubsectionCond} \& \ref{SubSecChapman}, respectively.  One technical issue  arising here is that the operation~(\ref{MeasComb}) does not  define a map from $\mathcal{M}^2    $ to $\mathcal{M}$, where recall the  $\mathcal{M}:=\mathcal{M}(\R^2) $.   In Section~\ref{SubsecSpecSpace} we introduce a  subcollection  $\widetilde{\mathcal{M}}$ of $\mathcal{M}$ and endow it with a topology that  is stronger than the vague topology, and we show that  the map $(\mu_1,\mu_2)\mapsto \mu_1\,\mathlarger{\mathlarger{\bullet}}_{\varsigma}\, \mu_2$ from $\widetilde{\mathcal{M}}^2$ to $\mathcal{\widetilde{M}}$  is continuous in Section~\ref{SubsectCont}.  The proof of Lemma~\ref{LemInd} is in Section~\ref{SubsectLemInd}.

\subsection{The $\boldsymbol{\widetilde{C}_m}$-weak topology for measures on $\boldsymbol{(\R^2)^{m+1}}$  }\label{SubsecSpecSpace}

Let $m,n\in \mathbb{N}$.  Recall that $\mathcal{M}_{m}:=\mathcal{M}\big((\R^2)^{m+1}\big)$, and  put $\widehat{\mathcal{M}}_{m}:=\widehat{\mathcal{M}}\big((\R^2)^{m+1}\big)$. Define $\phi_{m}^{n}:  (\R^2)^{m+1}\rightarrow [0,\infty) $ by 
\[
\phi^n_m(x_0,\ldots,x_m)\,=\,\textup{exp}\bigg\{-\frac{1}{n}\,|x_0|\,+\,n\,\sum_{1}^m\,|x_j| \, \bigg\}\,, \hspace{1cm} x_j\in \R^2\,.
\]
We omit the subscript $m$ on these notations when $m=1$. 
Let $\widetilde{C}_m$ denote the collection of real-valued continuous functions $f$ on $(\R^2)^{m+1}$ for which there exist $c,n>0$  such that
$|f|\leq c\phi^n_m$.  We use $\widetilde{\mathcal{M}}_m$ to denote the collection of  measures $\mu$ on $(\R^2)^{m+1}$ such that  $ \mu(\phi_m^n)<\infty  $ for all $n\in \mathbb{N}$.  Then  every $\varphi\in \widetilde{C}_m$ is $\mu$-integrable, and we equip  $\widetilde{\mathcal{M}}_m$ with the weak topology induced by the family $\widetilde{C}_m$, meaning the coarsest topology such that the map $\mu\in \widetilde{\mathcal{M}}_m \rightarrow \mu(\varphi)$ is continuous for all $\varphi\in \widetilde{C}_m$. Equivalently, we can characterize this as the  topology generated by the family of maps $\{\Phi_m^n\}_{n\in \mathbb{N}}$ for $\Phi_m^n: \widetilde{\mathcal{M}}_m\rightarrow \widehat{\mathcal{M}}_{m}   $ defined by $\Phi^n_m(\mu)= \phi^n_m\,\mu $, denoting here the measure having Radon-Nikodym derivative $\phi^n_m$ with respect to $\mu$. The map $\Phi_m:\widetilde{\mathcal{M}}_m\rightarrow \widehat{\mathcal{M}}_{m}^{\infty}$ defined by $\Phi_m(\mu)= \big(\Phi^1_m(\mu), \Phi^2_m(\mu),\ldots  \big)  $ is a closed  embedding by the proposition below.  Note that the map $\Phi_m$ is injective since, in fact, each $\Phi^n_m$ is injective.   We will refer to this topology on $\widetilde{\mathcal{M}}_m$ as the $\widetilde{C}_m$\textit{-weak} topology, which is stronger than the  vague topology induced by the inclusion  $\widetilde{\mathcal{M}}_m\subset \mathcal{M}_{m} $.  \vspace{.2cm}

The $\widetilde{C}_m$-weak topology inherits all of the basic properties of the weak topology on $\widehat{\mathcal{M}}_{m} $.  By Prohorov's theorem, there exists a complete metric  $\rho_m$  on $\widehat{\mathcal{M}}_{m}$ inducing the weak topology.  The product topology on  $\widehat{\mathcal{M}}_{m}^{\infty}$ is induced by the 
complete metric $\boldsymbol{\rho}_m$ given by
\begin{align}\label{RhoBold}
\boldsymbol{\rho}_m\big(\mu,\lambda\big)\,:=\,\sum_{n=1}^{\infty}\,\frac{1}{2^n}\wedge \rho_m\big( \mu^{(n)}  ,  \lambda^{(n)}\big) \,,\hspace{1.2cm}\,\mu=\big(\mu^{(1)},\mu^{(2)},\ldots\big)\,,\,\,\, \lambda=\big(\lambda^{(1)},\lambda^{(2)},\ldots\big)\,.
\end{align}
We can then induce the topology on  $\widetilde{\mathcal{M}}_m$   using the  metric  $\widetilde{\rho}_m(\mu,\nu):=\boldsymbol{\rho}_m\big(\Phi_m (\mu) ,\Phi_m(\nu)\big)  $, which is complete by the following proposition. 
\begin{proposition} \label{PropMTilde} Fix $m\in \mathbb{N}$. The space $\widetilde{\mathcal{M}}_m$ is Polish under the metric  $\widetilde{\rho}_m$. In particular, the embedding map $\Phi_m$ is closed, and we have (i)--(iii) below. 
\begin{enumerate}[(i)]

\item A subset  $M$ of   $\widetilde{\mathcal{M}}_m$ is precompact iff  $\Phi^n_m[M]$ is  precompact in $ \widehat{\mathcal{M}}_{m}$ for each $n\in \mathbb{N}$.

\item  A family $\{\boldsymbol{\mu}_{\alpha}\}_{\alpha\in I}$ of  $\widetilde{\mathcal{M}}_m$-valued random elements is tight provided that  $\big\{\Phi^n_m(\boldsymbol{\mu}_{\alpha})\big\}_{\alpha\in I}$ is tight in $\widehat{\mathcal{M}}_{m}$ for each $n\in \mathbb{N}$.

\item A family $\{\boldsymbol{\mu}_{\alpha}\}_{\alpha\in I}$ of $\widetilde{\mathcal{M}}_m$-valued random elements is tight provided that $\big\{\mathbb{E}[\boldsymbol{\mu}_{\alpha}]:\alpha \in I \big\}$ is precompact in $\widetilde{\mathcal{M}}_m$. 
\end{enumerate}

\end{proposition}
\begin{proof} Note that $\widetilde{\mathcal{M}}_m$ is separable because $\Phi_m$ embeds $\widetilde{\mathcal{M}}_m$ in the separable space $\widehat{\mathcal{M}}_{m}^{\infty}$.  The space $ \widetilde{\mathcal{M}}_m$ is complete under  $\widetilde{\rho}_m$ if and only if   $\Phi[\widetilde{\mathcal{M}}_m]$ is sequentially closed under the metric $\boldsymbol{\rho}_m$. Suppose that for $\{ \mu_j\}_1^{\infty}\subset \widetilde{\mathcal{M}}_m$ the sequence $\{\Phi_m(\mu_j)\}_{j=1}^{\infty}$ converges to a limit $\big(\lambda^{(1)}, \lambda^{(2)},\ldots\big)\in \widehat{\mathcal{M}}_{m}^{\infty}$ under $\boldsymbol{\rho}_m$. Then,  for any $n\in \mathbb{N}$, the sequence $\{ \Phi^n_m(\mu_j) \}_{j=1}^{\infty}$  converges  to   $\lambda^{(n)}\in \widehat{\mathcal{M}}_{m}$ under $ \rho_m $.  Define the  measure $\eta^{(n)}:= \big(\phi^{n}_m\big)^{-1} \,\lambda^{(n)}  $.  We can see that $ \eta^{(n)}$ does not depend on $n$ through the simple computation below. Given any $k<n$ and  $\varphi\in C_b\big( (\R^2)^{m+1} \big)$,  we have
\begin{align}
\Phi_m^{k}(\mu_j)(\varphi)\,=\, \big( \phi_{m}^{k} \,\mu_j \big)(\varphi) \,=\, \big( \phi^{n}_m \,\mu_j \big)\bigg(\frac{ \phi^{k}_m }{ \phi_m^{n} }\,\varphi\bigg) \,=\, \Phi^n_m (\mu_j)\bigg(\frac{ \phi^{k}_m }{ \phi^n_{m} }\,\varphi\bigg)\,.
\end{align}
Note that $\frac{ \phi^k_{m} }{ \phi^{n}_m }\varphi \in C_b\big( (\R^2)^{m+1} \big)$.   As $j\rightarrow \infty$ the above yields the equality $\lambda^{(k)}(\varphi)= \lambda^{(n)}\big(\frac{ \phi^{k}_m }{ \phi^{n}_m }\varphi\big)$  for  any $\varphi$, and thus we can conclude that  $\lambda^{(k)}= \frac{ \phi^{k}_m }{ \phi^n_{m} }\lambda^{(n)}$.  Hence $\eta^{(k)}=\eta^{(n)}$ holds for all $k, n\in \mathbb{N}$, and so the superscript on $\eta^{(n)}\equiv \eta$ can be omitted, and we have  $\big(\lambda^{(1)}, \lambda^{(2)},\ldots\big) =\Phi_m(\eta)\in  \Phi_m[\widetilde{\mathcal{M}}]$.   Therefore $ \Phi_m[\widetilde{\mathcal{M}}]$ is sequentially closed.  \vspace{.2cm}

\noindent Part (i): The ``only if" direction is immediate since each $\Phi^n_m$ is continuous.  Suppose that   $\Phi_m^n[M]  $ is precompact for each $n$.  Then  the product set $\prod_{n=1}^{\infty} \overline{\Phi^n_m[M]}  $ is a compact subset of  $\widehat{\mathcal{M}}_{m}^{\infty}$ by Tychonoff's theorem.  It follows that $\Phi_m[M]$ is  precompact, as it is a subset of  $ \prod_{n=1}^{\infty} \overline{\Phi^n_m[M]}  $. Since $\Phi_m $ is a closed embedding, the set $M$ is precompact.\vspace{.2cm}

\noindent Part (ii):  If $\big\{\Phi^n_m(\boldsymbol{\mu}_{\alpha})\big\}_{\alpha\in I}$ is tight in $ \widehat{\mathcal{M}}_{m}$ for each $n$, then for any $\epsilon>0$ we can pick a sequence $\{K_n\}_1^{\infty}$ of compact subsets of $\widehat{\mathcal{M}}_{m}$ such that $\mathbb{P}\big[  \Phi^n_m[\boldsymbol{\mu}_{\alpha} ]\in K_n^c \big]\leq \frac{\epsilon}{2^n}  $ for all $\alpha$ and $n$.   The set $K=\prod_{1}^{\infty}K_n   $ is compact in $\widehat{\mathcal{M}}_{m}^{\infty}$ by Tychonoff's theorem, and hence $ \Phi_m^{-1}[K]$ is compact in $\widetilde{\mathcal{M}}_m$, as $\Phi_m$ is a closed embedding. For any $\alpha$, we have
\begin{align*}
\mathbb{P}\left[  \boldsymbol{\mu}_{\alpha} \in \Phi_m^{-1}[K] \right]\,= \,&\,\mathbb{P}\big[ \Phi^n_m(\boldsymbol{\mu}_{\alpha} )\in K_n \text{ for each $n$} \big]\\   \,\geq\, & \,1\,-\,\sum_{n=1}^{\infty}\mathbb{P}\big[ \Phi^n_m(\boldsymbol{\mu}_{\alpha} )\in K_n^c \big] \,\geq \,1\,-\,\sum_{n=1}^{\infty} \frac{\epsilon}{2^n}\,=  \,1\,-\,\epsilon \,,
\end{align*}
and hence the family $\big\{\Phi_m(\boldsymbol{\mu}_{\alpha})\big\}_{\alpha\in I}$ is tight.\vspace{.2cm}

\noindent Part (iii): Suppose that $M:=\big\{\mathbb{E}[\boldsymbol{\mu}_{\alpha}]:\alpha \in I \big\}$ is precompact in $\widetilde{\mathcal{M}}_m$. By (i) the set 
\[
\Phi^n_m[M]\,=\,\Big\{\Phi^n_m\big(\mathbb{E}[\boldsymbol{\mu}_{\alpha}]\big)\,:\,\alpha \in I \Big\} \,=\,\Big\{\mathbb{E}\big[\Phi^n_m(\boldsymbol{\mu}_{\alpha})\big]\,:\,\alpha \in I \Big\}
\]
is  precompact in $ \widehat{\mathcal{M}}_{m}$ for each $n\in \mathbb{N}$. It follows from~\cite[Theorem 4.10]{Kallenberg} that the family $\big\{  \Phi^n_m(\boldsymbol{\mu}_{\alpha})\big\}_{\alpha\in I}$ is tight.
\end{proof}

Since the $\widetilde{C}_m$-weak topology is stronger than the vague topology on $\widetilde{\mathcal{M}}_m$,  part (i) of Proposition~\ref{PropMTilde} implies the following.

\begin{corollary} A sequence $\{ \mu_j\}_{1}^{\infty}\subset \widetilde{\mathcal{M}}_m$ converges to $\mu\in \widetilde{\mathcal{M}}_m$ in the $\widetilde{C}_m$-weak topology if and only if $  \mu_j  \rightarrow \mu$ vaguely and  $ \sup_{ j}\,\mu_{j}(\phi^n_m) <\infty       $ for each $n\in \mathbb{N}$.

\end{corollary}

The following basic lemma states that sequential convergence under the $\widetilde{C}_m$-topology only needs to be checked for functions  $\varphi\in \widetilde{C}_m$ having a  product form.
\begin{lemma}\label{LemmaSuffice} A sequence $\{ \mu_j\}_{1}^{\infty} \subset \widetilde{\mathcal{M}}_m$ converges to $\mu\in \widetilde{\mathcal{M}}_m$ provided that $\mu_j(\varphi)\rightarrow \mu(\varphi)$ for all $\varphi\in \widetilde{C}_m$ of the product form below  for some  $n\in \mathbb{N}$ and  $\psi^{(j)}\in C_b(\R^2)$:
\[
\varphi(x_0,\ldots, x_m)\,=\,\textup{exp}\bigg\{-\frac{1}{n}|x_0|\, +\,n\,\sum_1^m\,  |x_j| \,\bigg\}\,\prod_{j=0}^m \psi^{(j)}(x_j)\,, \hspace{1cm}x_j\in \R^2\,.
\]

\end{lemma}

\begin{proof} This is equivalent to asserting that the sequence of finite measures $\{\lambda_{k}^{(n)}\}_{k=1}^{\infty}  $ defined by $\lambda_{k}^{(n)}:=\Phi^{n}_m(\mu_k)$ converges weakly to  $\lambda^{(n)}:=\Phi^{n}_m(\mu)$ provided that  $\lambda_{k}^{(n)}\big(\psi^{(0)}\otimes \cdots \otimes \psi^{(m)}\big)\rightarrow \lambda^{(n)}\big(\psi^{(0)}\otimes \cdots \otimes \psi^{(m)}\big)$ as $k\rightarrow \infty$ for any $\psi^{(j)}\in C_b(\R^2)$, which is true.
\end{proof}

\subsection{Continuity of the $\mathlarger{\mathlarger{\bullet}}_{\varsigma}$ operation }\label{SubsectCont}

The aim of this subsection is to prove the proposition below.  
\begin{proposition}\label{PropCont}  Fix $\varsigma>0$ and $k,m\in \mathbb{N}$. If
$\mu\in \widetilde{\mathcal{M}}_k$ and $\lambda\in \widetilde{\mathcal{M}}_m$, then  $\mu\,\mathlarger{\mathlarger{\bullet}}_{\varsigma}\,\lambda \in \widetilde{\mathcal{M}}_{k+m}$.  Furthermore, the map from $\widetilde{\mathcal{M}}_k\times \widetilde{\mathcal{M}}_m$ to $ \widetilde{\mathcal{M}}_{k+m}$ defined by $\mathlarger{\mathlarger{\bullet}}_{\varsigma}$ is continuous.

\end{proposition}

We will organize the proof of Proposition~\ref{PropCont} around checking the conditions of the following  technical lemma.

\begin{lemma}\label{LemmaUnif} Let $(X,\rho_X)$ and $(Y,\rho_Y)$ be Polish spaces.  Suppose that conditions (I)--(III) below hold. 
\begin{enumerate}[(I)]

\item  $\mu_1$, $\mu_2$,\ldots and $\mu_{\infty}$   are finite  measures on $Y$ such that $\mu_j\rightarrow \mu_{\infty}$ weakly.

\item  $\varphi$ is a real-valued continuous function on  $X\times Y$ such that $\varphi(x,\cdot)\in C_b(Y)  $ for all $x\in X$, and the map $x\mapsto \varphi(x,\cdot) $ from $X$ to $C_b(Y)$ is continuous with respect to the uniform norm on $C_b(Y)$.

\item For every $\epsilon>0$ there exists a compact set $K\subset X$ such that 
\[
C_K\,:=\,\sup_{x\in K^c}\,\sup_{j\in \mathbb{N} }\,\int \big|\varphi(x,y)\big|\,\mu_j(dy)\,<\,\epsilon \,.
\]

\end{enumerate}
Define $\widetilde{\varphi}_j:X\rightarrow \R$ by $\widetilde{\varphi}_j(x):= \int \varphi(x,y)\,\mu_j(dy) $ for $j\in \mathbb{N}\cup \{\infty\}$.  Then the sequence $\{\widetilde{\varphi}_j\}_{1}^{\infty}$ converges uniformly to $\widetilde{\varphi}_\infty$.

\end{lemma}

\begin{proof} By condition (III), for any  given $\epsilon>0$ we can pick a compact set $K\subset X$ such that $C_K<\epsilon$.  Note that since $\mu_j\rightarrow \mu_{\infty}$ weakly by condition (I), we also have $\sup_{x\in K^c}\,\int \big|\varphi(x,y)\big|\,\mu_{\infty}(dy)\,\leq \,C_K$, and consequently $|\widetilde{\varphi}_j|\leq \epsilon$ on $K^c$ for all $j\in \mathbb{N}\cup \{\infty\}$.
Thus it suffices to show that $\{\widetilde{\varphi}_j\}_{1}^{\infty}$ converges uniformly to $\widetilde{\varphi}_\infty$ on any compact subset $K$ of $X$.  The map $x\mapsto \varphi(x,\cdot)$, which is continuous by condition (II), is uniformly continuous when its domain is restricted to $ K$.  Thus for any $\varepsilon>0$ there exists $\delta>0$ such that for all $x,a\in K$ with $\rho_X(x,a)<\delta$ we have
\begin{align*}
  \big\| \varphi(x,\cdot)-\varphi(a,\cdot ) \big\|_{\infty} \,\leq \,\frac{\varepsilon}{1+\sup_{ k\in \mathbb{N} }\mu_k(Y)    } \,,
\end{align*}
and so for all $j\in \mathbb{N}$
\begin{align*}
\big|\widetilde{\varphi}_j(x)-\widetilde{\varphi}_j(a)\big|\,\leq \, \big\| \varphi(x,\cdot)-\varphi(a,\cdot) \big\|_{\infty}\,\sup_{k\in \mathbb{N}} 
 \mu_k(Y)    
 \, < \,\varepsilon \,.
\end{align*}
Hence the family of functions $\{\widetilde{\varphi}_j : j\in \mathbb{N}\}$ is uniformly equicontinuous on $K$, and it is also pointwise bounded since
\[
 \sup_{ j\in \mathbb{N}  } \big| \widetilde{\varphi}_j(x) \big|\,\leq \,\big\|\varphi(x,\cdot)\big\|_{\infty}\,\sup_{j\in \mathbb{N} }\mu_j(Y)\,.
\]
Since the sequence $\{\widetilde{\varphi}_j\}_{1}^{\infty}$ converges pointwise to $\widetilde{\varphi}_\infty$ by (I),
  it follows from Arzel\`a-Ascoli's theorem that $\{\widetilde{\varphi}_j\}_{1}^{\infty}$ converges uniformly  to $\widetilde{\varphi}_\infty$ on $K$, which completes the proof.
\end{proof}

\begin{proof}[Proof of Proposition~\ref{PropCont}]  We will specialize our proof to the case $k=m=1$, as the general case involves less wieldy notation but no additional difficulty.  Let $\mu,\lambda\in \widetilde{\mathcal{M}}$.  For $n\in \mathbb{N}$ recall that we define $\phi^n(x,y)=e^{-\frac{1}{n}|x| +n|y| }$. We can write
\begin{align*}
\mu\, \mathlarger{\mathlarger{\bullet}}_{\varsigma}\,\lambda \,( \phi^n)\,=\,&\,\int_{(\R^2)^4}\,e^{-\frac{1}{n}|x|  }\,\mu(dx,da)\,\frac{ e^{- \frac{ |a-b|^2  }{ 2\varsigma }   }
 }{ 2\pi \varsigma } \,\lambda(db,dy)\, e^{n|y|  }\\
 \,=\,&\,\frac{1}{2\pi \varsigma} \int_{(\R^2)^4}\,\phi^n(x,a) \,\mu(dx,da)\, \textup{exp}\bigg\{-n|a|+\frac{1}{n}|b|- \frac{ |a-b|^2  }{ 2\varsigma } \,\bigg\}   \,\lambda(db,dy)\, \phi^n(b,y) \,.
\end{align*}
For $n\geq 1$ we have the following trivial upper bound for the exponent above:
\begin{align}\label{Trivy}
-n|a|\,+\,\frac{1}{n}|b|\,  -\, \frac{ |a-b|^2  }{ 2\varsigma } \,\leq \,\frac{1}{n}(-|a|\,+\,|b|)\,   -\, \frac{ |a-b|^2  }{ 2\varsigma } \,\leq \,\frac{1}{n}|a-b|\,   -\, \frac{ |a-b|^2  }{ 2\varsigma } \,\leq \,\frac{\varsigma}{2n^2}\,,
\end{align}
where the last inequality results from maximizing the quadratic polynomial $p(x)=\frac{x}{n}-\frac{x^2}{2\varsigma}$.
It follows that
\begin{align*}
\mu\, \mathlarger{\mathlarger{\bullet}}_{\varsigma}\,\lambda\,(\phi^n) \,\leq \,\frac{1}{2\pi \varsigma}\,e^{ 
 \frac{\varsigma}{2n^2} }\mu(\phi^n) \,\lambda (\phi^n) \,<\,\infty 
\end{align*}
 for any $n\in \mathbb{N}$, implying that $\mu\, \mathlarger{\mathlarger{\bullet}}_{\varsigma}\,\lambda \in \widetilde{\mathcal{M}} $. \vspace{.2cm}

For sequences $\{\mu_j\}_1^{\infty},\{\lambda_j\}_1^{\infty}\subset \widetilde{\mathcal{M}}$ with $\mu_j\rightarrow \mu_{\infty}$ and $\lambda_j\rightarrow \lambda_{\infty}$, we wish to show that $\mu_j\, \mathlarger{\mathlarger{*}}_{\varsigma}\,\lambda_j \rightarrow \mu_{\infty}\, \mathlarger{\mathlarger{\bullet}}_{\varsigma}\,\lambda_{\infty}$.  For this, by Lemma~\ref{LemmaSuffice}, it suffices to show that  $\mu_j\, \mathlarger{\mathlarger{\bullet}}_{\varsigma}\,\lambda_j(\varphi)\rightarrow \mu_{\infty}\, \mathlarger{\mathlarger{\bullet}}_{\varsigma}\,\lambda_{\infty}(\varphi)$ for all $\varphi\in \widetilde{C}$ of the form $\varphi(x,y)=\psi^{(1)}(x)\psi^{(2)}(y)e^{-\frac{1}{n}|x| +n|y| }$ for some $\psi^{(1)},\psi^{(2)}\in C_b(\R^2)$ and $n\in \mathbb{N}$.  We can write
\begin{align*}
\mu_j\, \mathlarger{\mathlarger{\bullet}}_{\varsigma}\,\lambda_j( \varphi)\,=\,&\,\int_{(\R^2)^4}\,\psi^{(1)}(x)\,e^{-\frac{1}{n}|x|  }\,\mu_j(dx,da)\,\frac{ e^{- \frac{ |a-b|^2  }{ 2\varsigma }   }
 }{ 2\pi \varsigma } \,\lambda_j(db,dy)\, e^{n|y|  } \,\psi^{(2)}(y) \\
 \, = \,&\,\int_{(\R^2)^2}\, \psi^{(1)}(x)\, \Phi^n(\mu_j)(dx,da)\,\varphi_{j,n}(a)\, ,
\end{align*}
in which $\widetilde{\varphi}_{j,n}:\R^2\rightarrow \R$ is defined by
\begin{align*}
\widetilde{\varphi}_{j,n}(a)\, =\, \int_{ (\R^2)^2}\, \,\varphi_n\big(a; b,y\big)\, \Phi^n(\lambda_j)(db,dy) \hspace{.7cm}
\text{for} \hspace{.7cm}
\varphi_n\big(a; b,y\big)\,=\,\frac{ e^{- n|a|+\frac{1}{n}|b|  - \frac{ |a-b|^2  }{ 2\varsigma }  } }{ 2 \pi \varsigma  } \,\psi^{(2)}(y)   \,.
\end{align*}
Since  $\{\Phi^n(\mu_j)\}_{j=1}^{\infty}$ converges weakly to $\Phi^n(\mu_\infty)$, we have  $\mu_j\, \mathlarger{\mathlarger{\bullet}}_{\varsigma}\,\lambda_j( \varphi)\rightarrow \mu_{\infty}\, \mathlarger{\mathlarger{\bullet}}_{\varsigma}\,\lambda_{\infty}( \varphi)$ provided that $\widetilde{\varphi}_{j,n} \in C_b(\R^2) $ and $\widetilde{\varphi}_{j,n} \rightarrow \widetilde{\varphi}_{\infty,n}$ uniformly.  We can verify this by checking conditions (I)--(III) of Lemma~\ref{LemmaUnif} for the sequence of measures $\big\{  \Phi^n(\lambda_j)\big\}_{j=1}^{\infty}$ on $(\R^2)^2$ and the  function $\varphi_n:\R^2\times (\R^2)^2\rightarrow \R $. \vspace{.2cm}

\noindent Condition  (I).  Since $\lambda_j\rightarrow \lambda_{\infty}$ in $\widetilde{\mathcal{M}}$ by assumption,  the sequence of measures  $\big\{  \Phi^n(\lambda_j)\big\}_{j=1}^{\infty}$ converges weakly to $\Phi^n(\lambda_{\infty})$ for each $n\in \mathbb{N}$.   \vspace{.3cm}

\noindent Condition (II). We  have $ \varphi_n(a; \cdot, \cdot)\in C_b\big( (\R^2)^2\big)$ for every $a\in \R^2$ since $ \varphi_n(a;b,y)$ is jointly continuous and  
\begin{align}\label{Trivy2}
\sup_{a,b,y\in \R^2}\big|  \varphi_n\big(a; b,y\big) \big|  \,\stackrel{ (\ref{Trivy}) }{ \leq}  \, \frac{ e^{\frac{\varsigma}{2n^2} }}{ 2 \pi \varsigma }  \,\|\psi^{(2)}\|_{\infty}\, < \infty\,.
\end{align}
 The supremum of the gradient $ \nabla_a \, \varphi_n\big(a; b,y\big) =-\big(n\frac{a}{|a|}+\frac{1}{\varsigma}\,(a-b)   \big) \varphi_n\big(a; b,y\big)  $ satisfies
\begin{align*}
\sup_{a,b,y\in \R^2}\big|  \nabla_a \varphi_n\big(a; b,y\big) \big| \,\leq\,\sup_{a,b,y\in \R^2} \bigg( n\,+\,\frac{1}{\varsigma}\,|a-b| \bigg)\,\big|  \varphi_n\big(a; b,y\big) \big|\,<\,\infty  \,,
\end{align*}
and hence the map $a\mapsto  \varphi_n(a; \cdot,\cdot ) $  from $ \R^2$  to $  C_b\big( (\R^2)^2 \big)$ is continuous (in fact,  Lipschitz). \vspace{.3cm}

\noindent Condition (III).  Since  $\Phi^n(\lambda_j)\rightarrow \Phi^n(\lambda_\infty)$ weakly, the family $\{\Phi^n(\lambda_j)\}_{ j\in \mathbb{N} }$ is tight.  Thus for any $\epsilon>0$ there exists an $R>0$ such that, for  the open ball  $B_R$  of radius $R$ centered at the origin in $(\R^2)^2\equiv \R^4$, we have
\begin{align}\label{Trivy3}
\sup_{j\in \mathbb{N} }\,\Phi^n(\lambda_j)\big(B_R^c\big)\,\leq \,\frac{\epsilon}{2\|\psi^{(2)}\|_{\infty} } \frac{1}{1+\frac{1}{2\pi \varsigma} e^{\frac{\varsigma}{2n^2} } }\,. 
\end{align}
 Pick $L>0$ large enough so that 
\begin{align}\label{Boog}
\frac{e^{ \frac{1}{n}L - \frac{ L^2  }{ 2\varsigma }  } }{2\pi \varsigma} \leq \frac{\epsilon}{2} \frac{1}{1+  \|\psi^{(2)}\|_{\infty}\sup_{j} \Phi^n(\lambda_j)(\R^2\times \R^2) }  \,.
\end{align}
Define the compact set $K=\big\{x\in \R^2:|x|\leq L+R  \big\}$.
For any $a\in K^c$ and $j\in  \mathbb{N} $,
\begin{align*}
 \int_{(\R^2)^2}\, \big|\varphi_n\big(a; b,y\big)\big|\,\Phi^n(\lambda_j)(db,dy)  \,\leq \,&\, \|  \varphi_n\|_{\infty}\,\Phi^n(\lambda_j)\big(B_R^c\big)    \,+\,  \int_{B_R}\,\big|\varphi_n\big(a; b,y\big)\big|\,\,\Phi^n(\lambda_j)(db,dy)   \\
 \,\leq \,&\,    \frac{\epsilon}{2}\,+\,\|\psi^{(2)}\|_{\infty} \int_{B_R}\frac{ e^{ \frac{1}{n}|a-b| - \frac{ |a-b|^2  }{ 2\varsigma }  }  }{ 2 \pi \varsigma  } \,\Phi^n(\lambda_j)(db,dy) 
   \\ \,\leq \,&\, \frac{\epsilon}{2}\,+\,\frac{\epsilon}{2}\,=\,\epsilon\,.
\end{align*}  
The second inequality applies (\ref{Trivy2})--(\ref{Trivy3}) to the first term.  The third inequality applies~(\ref{Boog}) and that if $(b,y)\in B_R$ and $a\in K^c$ then $|a-b|\geq L$.  Thus we have verified conditions (I)--(III) of Lemma~\ref{LemmaUnif}.
\end{proof}

\subsection{Proof of Proposition~\ref{PropCond}}  \label{SubsectionCond}

 Given $\vartheta,\varepsilon>0$ let $ \mathscr{Z}^{\mathsmaller{\vartheta},\varepsilon}=\big\{\mathscr{Z}^{\mathsmaller{\vartheta},\varepsilon}_{s,t}\big\}_{0\leq s<t<\infty} $ denote the process~(\ref{DefZEpsilon}), defined on a probability space $(\boldsymbol{\Omega}, \mathds{F},  \mathds{P})$.  For $S\subset [0,\infty)$ with nonempty interior, define the $\sigma$-algebra
\[
\mathds{F}_{S}\,:=\,\sigma\bigg\{ \int \, f(t,x) \, \xi(t,x)\,dt\,dx\,:\,f\in L^2\big([0,\infty)\times \R^2\big) \text{ with }\textup{supp}(f)\subset S^{o}\times \R^2 \bigg\}\, .
\]
Thus, we can formally express $\mathds{F}_{S}$ as $\sigma\big\{\xi(t,x):t\in S,\,x\in \R^2\big\} $.  The following trivial lemma is the counterpart for $\mathscr{Z}^{\mathsmaller{\vartheta},\varepsilon}$ of the conditional expectation formula in Proposition~\ref{PropCond}. Naturally, our proof of  Proposition~\ref{PropCond} entails showing that the formula is inherited by the distributional limit $\mathscr{Z}^{\mathsmaller{\vartheta}}$.
\begin{lemma} \label{LemmaEpsilonCase} Fix $\vartheta\in \R$ and $\varepsilon>0$.   For any $0\leq r<s<t<u $, we have
\begin{align}\label{CondExpEpsilon}
\mathds{E}\Big[\,\mathscr{Z}^{\mathsmaller{\vartheta},\varepsilon}_{r,u}\, \Big|\, \mathds{F}_{[0,s]\cup [t,\infty)} \,\Big]\,=\,\mathscr{Z}^{\mathsmaller{\vartheta},\varepsilon}_{r,s}\,\mathlarger{\mathlarger{\bullet}}_{t-s}\,\mathscr{Z}^{\mathsmaller{\vartheta},\varepsilon}_{t,u}\,.
\end{align}
When $r=s$ or $u=t$, the right side above becomes respectively $\mathlarger{\mathlarger{\bullet}}_{t-s}\,\mathscr{Z}^{\mathsmaller{\vartheta},\varepsilon}_{t,u}$ and $\mathscr{Z}^{\mathsmaller{\vartheta},\varepsilon}_{r,s}\,\mathlarger{\mathlarger{\bullet}}_{t-s}$.
\end{lemma}

\begin{proof}  Since $\mathscr{Z}^{\mathsmaller{\vartheta},\varepsilon}_{s,t}(dx,dy)=dx\,U_{s,t}^{\mathsmaller{\vartheta},\varepsilon}(x,dy)$ for the kernel $U_{s,t}^{\mathsmaller{\vartheta},\varepsilon}$  defined in~(\ref{DefU}), establishing our desired result amounts to showing that for any $x\in \R^2$ 
\begin{align}\label{GOAL}
\mathds{E}\Big[\,U_{r,u}^{\mathsmaller{\vartheta},\varepsilon}(x,\cdot)\, \Big|\, \mathds{F}_{[0,s]\cup [t,\infty)}\, \Big]\,=\, \int_{a,b\in \R^2 }\,U_{r,s}^{\mathsmaller{\vartheta},\varepsilon}(x,da)\,g_{t-s }(a-b) \,db\, U_{t,u}^{\mathsmaller{\vartheta},\varepsilon}(b,\cdot )\,.
\end{align}
Given some $\varphi\in C_c(\R^2)$,  the conditional expectation of  $U_{r,u}^{\mathsmaller{\vartheta},\varepsilon}(x,\varphi) $ with respect to $\mathds{F}_{(0,s)\cup(t,\infty)}$ can be expressed as
\begin{align}\label{IntAbs}
\mathds{E}\Big[\,U_{r,u}^{\mathsmaller{\vartheta},\varepsilon}(x,\varphi)\, \Big|\, \mathds{F}_{[0,s]\cup[t,\infty)}\, \Big]\,=\,\int_{\boldsymbol{\Upsilon}_r} \,  \mathds{E}\Big[\,\mathscr{E}^{\mathsmaller{\vartheta},\varepsilon}_{r,u}(p)  \, \Big|\, \mathds{F}_{[0,s]\cup[t,\infty)} \,\Big] \,\varphi(p_{u})  \, \mathbf{P}_{r,x}(dp) \,,
 \end{align}
where for $0\leq a<b$ we define
$$\mathscr{E}^{\mathsmaller{\vartheta},\varepsilon}_{a,b}(p)\,:=\, \textup{exp}\Bigg\{\sqrt{\vsm^{\mathsmaller{\vartheta},\varepsilon} } \int_{a}^{b}\xi_{\varepsilon} \big(r, \,p(r) \big)dr  -\vsm^{\mathsmaller{\vartheta},\varepsilon} \frac{(b-a)\|j\|^2_2}{2\varepsilon^2} \Bigg\}\,.$$
 
For fixed $p \in \boldsymbol{\Upsilon}_r$,  we can write 
\[
\int_r^{u}\,\xi_{\varepsilon} \big(a, \, p(a) \big)\,da\,=\,\int_r^{s}\,\xi_{\varepsilon} \big(a, \, p(a) \big)\,da\,+\,\int_{s}^{t}\,\xi_{\varepsilon} \big(a, \, p(a) \big)\,da\,+\,\int_{t}^{u}\,\xi_{\varepsilon} \big(a, \,p(a) \big)\,da\,,
\]
where the first and last terms on the right side are $\mathds{F}_{[0,s]\cup[t,\infty)} $-measurable, and the middle term is independent of $\mathds{F}_{[0,s]\cup[t,\infty)} $ and has distribution $\mathcal{N}\Big(0,\frac{(t-s)\|j\|^2}{\varepsilon^2}\Big)$. 
Thus we get
 \begin{align}\label{ProdU}
&\mathds{E}\Big[\,\mathscr{E}^{\mathsmaller{\vartheta},\varepsilon}_{r,u}(p)  \, \Big|\, \mathds{F}_{[0,s]\cup[t,\infty)} \,\Big]  \,=\, \mathscr{E}^{\mathsmaller{\vartheta},\varepsilon}_{r,s}(p)\, \mathscr{E}^{\mathsmaller{\vartheta},\varepsilon}_{t,u}(p) \, .
 \end{align}
Subbing~(\ref{ProdU}) into~(\ref{IntAbs}) and applying the Markov property  with  respect to its natural filtration $F_b:=\sigma\big\{ p(a) : a\in [0,b]  \big\}$ twice yields 
 \begin{align*}
\mathds{E}\Big[\,U_{r,u}^{\mathsmaller{\vartheta},\varepsilon}(x,\varphi)\, \Big|\, \mathds{F}_{[0,s]\cup[t,\infty)}\, \Big]\,=\,&\,\int_{\boldsymbol{\Upsilon}_r} \, \mathscr{E}^{\mathsmaller{\vartheta},\varepsilon}_{r,s}(p)\, \mathscr{E}^{\mathsmaller{\vartheta},\varepsilon}_{t,u}(p) \, \varphi\big(p(u)\big)  \, \mathbf{P}_{r,x}(dp) \\ \,=\,&\int_{\boldsymbol{\Upsilon}_r} \, \mathscr{E}^{\mathsmaller{\vartheta},\varepsilon}_{r,s}(p)\, 
 \bigg(\int_{\boldsymbol{\Upsilon}_{t}} \mathscr{E}^{\mathsmaller{\vartheta},\varepsilon}_{t,u}(q) \, \varphi\big(q(u)\big) \,\mathbf{P}_{t,p(t)}(dq)\bigg) \, \mathbf{P}_{r,x}(dp)\\
 \,=\, &\int_{\boldsymbol{\Upsilon}_r} \, \mathscr{E}^{\mathsmaller{\vartheta},\varepsilon}_{r,s}(p)\, 
U_{t,u}^{\mathsmaller{\vartheta},\varepsilon}\big(p(t),\varphi\big)\, \mathbf{P}_{r,x}(dp) \\
 \,=\, &\,\int_{\boldsymbol{\Upsilon}_r} \, \mathscr{E}^{\mathsmaller{\vartheta},\varepsilon}_{r,s}(p)\, \bigg(\int_{\boldsymbol{\Upsilon}_{s}}  \, U_{t,u}^{\mathsmaller{\vartheta},\varepsilon}\big(q(t),\varphi\big)\,\mathbf{P}_{s,q(s)}(dq)\bigg)
\, \mathbf{P}_{r,x}(dp) \\
\,=\,&\, \int_{\boldsymbol{\Upsilon}_r} \, \mathscr{E}^{\mathsmaller{\vartheta},\varepsilon}_{r,s}(p)\, \bigg(\int_{\R^2}  \, g_{t-s}\big(p(s)-a \big) \,U_{t,u}^{\mathsmaller{\vartheta},\varepsilon}(a,\varphi)\,da\bigg)
\, \mathbf{P}_{r,x}(dp)  \\
\,=\,&\, \int_{\boldsymbol{\Upsilon}_r} \, \mathscr{E}^{\mathsmaller{\vartheta},\varepsilon}_{r,s}(p)\,\widehat{\varphi}\big(p(s)\big)
\, \mathbf{P}_{r,x}(dp) \,=:\,\,U_{r,s}^{\mathsmaller{\vartheta},\varepsilon}(x,\widehat{\varphi}) \,,
 \end{align*}
where $\widehat{\varphi}(y)\,:=\,\int_{\R^2} g_{t-s}(y-b)\,U_{t,u}^{\mathsmaller{\vartheta},\varepsilon}(b,\varphi)\,db$ for $y\in \R^2$. Since the above holds for any $\varphi\in C_c(\R^2)$, we have proven~(\ref{GOAL}).
\end{proof}

Fix $n\in \mathbb{N}$.  Given $\varphi_j \in C_c\big( (\R^2)^2\big)$ and $0\leq s_j<t_j<\infty$ for each $j\in\{1,\ldots, n\}$, let $ \mathscr{F}_{s_1,t_1;\ldots; s_n,t_n}^{\varphi_1,\ldots, \varphi_n}$ denote the $\sigma$-algebra generated by the random vector $\big(  \mathscr{Z}^{\mathsmaller{\vartheta}}_{s_1,t_1}(\varphi_1),\ldots, \mathscr{Z}^{\mathsmaller{\vartheta}}_{s_n,t_n}(\varphi_n) \big)$, that is, the collection of  events of the form
\[
E_A\,:=\, \Big\{ \,\omega\in \Omega\,:\,\Big(  \mathscr{Z}^{\mathsmaller{\vartheta}}_{s_1,t_1}(\omega;\varphi_1),\ldots, \mathscr{Z}^{\mathsmaller{\vartheta}}_{s_n,t_n}(\omega;\varphi_n) \Big)\in A   \,    \Big\} 
\]
for some Borel set $A\subset \R^n$. Given $S\subset [0,\infty)$ the $\sigma$-algebra $\mathscr{F}_{S}:=\sigma\big\{ \mathscr{Z}^{\mathsmaller{\vartheta}}_{s,t}  : (s,t)\subset S\big\}$ is generated by the algebra given by
\begin{align*}
 \mathscr{A}_S \,:=\, \bigcup_{n=1}^{\infty} \,\bigcup_{\substack{(s_1,t_1),\ldots, (s_n,t_n) \subset S  \\ \varphi_1,\ldots, \varphi_n\in C_c( (\R^2)^2)  }    }  \mathscr{F}_{s_1,t_1; \ldots ; s_n,t_n }^{\varphi_1,\ldots, \varphi_n}\,.
\end{align*}
 
\begin{lemma}\label{LemmaTrivial} For any $E\in \mathscr{F}_{s_1,t_1;\ldots; s_n,t_n}^{\varphi_1,\ldots, \varphi_n}$ there exists a uniformly bounded sequence of functions $\{h_j\}_1^{\infty}\subset C_b(\R^n) $ such that $h_j\big(  \mathscr{Z}^{\mathsmaller{\vartheta}}_{s_1,t_1}(\varphi_1),\ldots, \mathscr{Z}^{\mathsmaller{\vartheta}}_{s_n,t_n}(\varphi_n) \big)\rightarrow 1_E$ almost surely as $j\rightarrow \infty$.

\end{lemma}
\begin{proof} Define the  probability measure $\mu$ on $\R^n$ by $\mu(A)=\mathbb{P}\big[ E_A\big]$.  Pick a Borel set $A\subset \R^n$ such that $E=E_A$.   Since finite  measures on $\R^n$ are  regular, we can pick   sequences of open sets $\{\mathcal{O}_j\}_1^{\infty}\subset \R^n $ and closed sets $\{\mathcal{K}_j\}_1^{\infty}\subset \R^n $ such that $\mathcal{K}_j\subset A \subset \mathcal{O}_j$ and $\mu(\mathcal{O}_j \backslash \mathcal{K}_j)\rightarrow 0$. By  Urysohn's lemma, for each $j$ there exists a  continuous function $h_j:\R^n\rightarrow [0,1]$ such that $h_j=1$ on $\mathcal{K}_j$ and $h_j=0$ on $\mathcal{O}_j^c$.  Then $h_j\big(  \mathscr{Z}^{\mathsmaller{\vartheta}}_{s_1,t_1}(\varphi_1),\ldots, \mathscr{Z}^{\mathsmaller{\vartheta}}_{s_n,t_n}(\varphi_n) \big)$ and $1_{E_A}$ are equal outside of the set $ E_{\mathcal{O}_j \backslash K_j}$, which has the vanishing probability $\mathbb{P}\big[  E_{\mathcal{O}_j \backslash K_j}\big]=\mu(\mathcal{O}_j \backslash K_j)$.
\end{proof}

\begin{proof}[Proof of Proposition~\ref{PropCond}] It suffices to show that for any $\varphi\in C_c\big((\R^2)^2\big)$ 
\begin{align}\label{Tribe}
\mathbb{E}\Big[\,\mathscr{Z}^{\mathsmaller{\vartheta}}_{r,u}(\varphi)\, \Big|\, \mathscr{F}_{[0,s]\cup[t,\infty)} \,\Big]\,=\,\mathscr{Z}^{\mathsmaller{\vartheta}}_{r,s}\,\mathlarger{\mathlarger{\bullet}}_{t-s}\,\mathscr{Z}^{\mathsmaller{\vartheta}}_{t,u}(\varphi)\,.
\end{align}
Since the $\sigma$-algebra $ \mathscr{F}_{[0,s]\cup [t,\infty)}$ is generated by the algebra $\mathscr{A}_{[0,s]\cup [t,\infty)}  $, we can check~(\ref{Tribe}) by verifying that 
\begin{align*}
\mathbb{E}\Big[\,\Big(\mathscr{Z}^{\mathsmaller{\vartheta}}_{r,u}(\varphi)-\mathscr{Z}^{\mathsmaller{\vartheta}}_{r,s}\,\mathlarger{\mathlarger{\bullet}}_{t-s}\,\mathscr{Z}^{\mathsmaller{\vartheta}}_{t,u}(\varphi)\Big)\,1_{E} \,\Big]\,=\,0
\end{align*}
for each $E\in \mathscr{A}_{[0,\mathbf{s}]\cup[\mathbf{t},\infty)}$.  As a consequence of Lemma~\ref{LemmaTrivial}, we only need to show that 
\begin{align}\label{NeedThis}
\mathbb{E}\bigg[\,\Big(\mathscr{Z}^{\mathsmaller{\vartheta}}_{r,u}(\varphi)-\mathscr{Z}^{\mathsmaller{\vartheta}}_{r,s}\,\mathlarger{\mathlarger{\bullet}}_{t-s}\,\mathscr{Z}^{\mathsmaller{\vartheta}}_{t,u}(\varphi)\Big)\,h\Big(  \mathscr{Z}^{\mathsmaller{\vartheta}}_{s_1,t_1}(\varphi_1),\ldots, \mathscr{Z}^{\mathsmaller{\vartheta}}_{s_n,t_n}(\varphi_n) \Big) \,\bigg]\,=\,0
\end{align}
for all $h\in C_b(\R^n)$,
  $\varphi_j\in C_c\big( (\R^2)^2\big)$,  and  $0\leq s_j<t_j<\infty$ with $(s_j,t_j)\subset [0,s]\cup [t,\infty]$ for $j\in \{1,\ldots, n\}$.  However, since $\mathds{E}\big[\mathscr{Z}^{\mathsmaller{\vartheta},\varepsilon}_{r,u}(\varphi)\, \big|\, \mathds{F}_{[0,s]\cup[t,\infty)} \big]=\mathscr{Z}^{\mathsmaller{\vartheta},\varepsilon}_{r,s}\,\mathlarger{\mathlarger{\bullet}}_{t-s}\,\mathscr{Z}^{\mathsmaller{\vartheta},\varepsilon}_{t,u}(\varphi)$ holds for any  $\varepsilon>0$ by Lemma~\ref{LemmaEpsilonCase} and the random variables $\mathscr{Z}^{\mathsmaller{\vartheta},\varepsilon}_{s_1,t_1}(\varphi_1)$,\ldots, $\mathscr{Z}^{\mathsmaller{\vartheta},\varepsilon}_{s_n,t_n}(\varphi_n) $ are $\mathds{F}_{[0,s]\cup[t,\infty)}$-measurable, we have
\begin{align}\label{ToGetThis}
\mathds{E}\bigg[\,\Big(\mathscr{Z}^{\mathsmaller{\vartheta},\varepsilon}_{r,u}(\varphi)-\mathscr{Z}^{\mathsmaller{\vartheta},\varepsilon}_{r,s}\,\mathlarger{\mathlarger{\bullet}}_{t-s}\,\mathscr{Z}^{\mathsmaller{\vartheta},\varepsilon}_{t,u}(\varphi)\Big)\,h\Big(  \mathscr{Z}^{\mathsmaller{\vartheta},\varepsilon}_{s_1,t_1}(\varphi_1),\ldots, \mathscr{Z}^{\mathsmaller{\vartheta},\varepsilon}_{s_n,t_n}(\varphi_n) \Big)\, \bigg]\,=\,0\,.
\end{align}
Recall that  the $\mathcal{M}^S$-valued random element  $\mathscr{Z}^{\mathsmaller{\vartheta}}$ is defined to have law arising as the distributional  
 limit of a sequence of $\mathcal{M}^S$-valued random elements  $\{\mathscr{Z}^{\mathsmaller{\vartheta}, \varepsilon_j }\}_1^{\infty}$ for  some vanishing sequence $\{\varepsilon_j\}_1^{\infty}\subset (0,1]$. Hence there is convergence in distribution of the following $\mathcal{M}^{n+3}$-valued random elements:
\begin{align}\label{Dist2Pointwise}
\Big(\mathscr{Z}^{\mathsmaller{\vartheta},\varepsilon_j}_{r,u},\,\mathscr{Z}^{\mathsmaller{\vartheta},\varepsilon_j}_{r,s}\,,\mathscr{Z}^{\mathsmaller{\vartheta},\varepsilon_j}_{t,u},\,  \mathscr{Z}^{\mathsmaller{\vartheta},\varepsilon_j}_{s_1,t_1},\ldots, \mathscr{Z}^{\mathsmaller{\vartheta},\varepsilon_j}_{s_n,t_n} \Big) \hspace{.5cm} \stackbin[j\rightarrow \infty ]{ \textup{d}}{\longrightarrow}  \hspace{.5cm} \Big(\mathscr{Z}^{\mathsmaller{\vartheta}}_{r,u},\,\mathscr{Z}^{\mathsmaller{\vartheta}}_{r,s}\,,\mathscr{Z}^{\mathsmaller{\vartheta}}_{t,u},\, \mathscr{Z}^{\mathsmaller{\vartheta}}_{s_1,t_1},\ldots, \mathscr{Z}^{\mathsmaller{\vartheta}}_{s_n,t_n} \Big)\,.
\end{align}
Since 
$  \mathbb{E}\big[\mathscr{Z}^{\mathsmaller{\vartheta}}_{a,b}\big]=\mathbf{U}_{b-a} $ is in $\widetilde{\mathcal{M}}$,
it follows from  (iii) of Proposition~\ref{PropMTilde} that $\big\{\mathscr{Z}^{\mathsmaller{\vartheta},\varepsilon_j}_{a,b}\big\}_{j=1}^{\infty}$ is a tight sequence of  $\widetilde{\mathcal{M}}$-valued random elements for any $0\leq a <b $. Thus the convergence~(\ref{Dist2Pointwise}) also holds in the   $\big(\widetilde{\mathcal{M}}\big)^{n+3}$-topology.  By Skorokhod's representation theorem, we can take the limiting and limit  $\big(\widetilde{\mathcal{M}}\big)^{n+3}$-valued random elements in~(\ref{Dist2Pointwise}) to be defined on  the same probability space and with the convergence~(\ref{Dist2Pointwise}) holding almost surely.  As $j\rightarrow \infty$ we then have that
\begin{align}
\mathscr{Z}^{\mathsmaller{\vartheta},\varepsilon_j}_{r,u}(\varphi)  & \hspace{.5cm}\stackbin[ ]{\textup{a.s.} }{\longrightarrow}  \hspace{.5cm}\mathscr{Z}^{\mathsmaller{\vartheta}}_{r,u}(\varphi) \,, \nonumber  \\
\mathscr{Z}^{\mathsmaller{\vartheta},\varepsilon_j}_{r,s}\,\mathlarger{\mathlarger{\bullet}}_{t-s}\,\mathscr{Z}^{\mathsmaller{\vartheta},\varepsilon_j}_{t,u}(\varphi) & \hspace{.5cm}\stackbin[ ]{\textup{a.s.} }{\longrightarrow}  \hspace{.5cm}\mathscr{Z}^{\mathsmaller{\vartheta}}_{r,s}\,\mathlarger{\mathlarger{\bullet}}_{t-s}\,\mathscr{Z}^{\mathsmaller{\vartheta}}_{t,u}(\varphi)\,,  \nonumber \\
h\Big(  \mathscr{Z}^{\mathsmaller{\vartheta},\varepsilon_j}_{s_1,t_1}(\varphi_1),\ldots, \mathscr{Z}^{\mathsmaller{\vartheta},\varepsilon_j}_{s_n,t_n}(\varphi_n) \Big) & \hspace{.5cm}\stackbin[ ]{\textup{a.s.} }{\longrightarrow}\hspace{.5cm}   h\Big(  \mathscr{Z}^{\mathsmaller{\vartheta}}_{s_1,t_1}(\varphi_1),\ldots, \mathscr{Z}^{\mathsmaller{\vartheta}}_{s_n,t_n}(\varphi_n) \Big) \,, \label{ThreeConv}
\end{align}
where the convergence on the second line follows from Proposition~\ref{PropCont}. Furthermore, the family of random variables
\begin{align*}
\bigg\{\Big(\mathscr{Z}^{\mathsmaller{\vartheta},\varepsilon}_{r,u}(\varphi)-\mathscr{Z}^{\mathsmaller{\vartheta},\varepsilon}_{r,s}\,\mathlarger{\mathlarger{\bullet}}_{t-s}\,\mathscr{Z}^{\mathsmaller{\vartheta},\varepsilon}_{t,u}(\varphi)\Big)\,h\Big(  \mathscr{Z}^{\mathsmaller{\vartheta},\varepsilon}_{s_1,t_1}(\varphi_1),\ldots, \mathscr{Z}^{\mathsmaller{\vartheta},\varepsilon}_{s_n,t_n}(\varphi_n) \Big) \bigg\}_{\varepsilon\in (0,1)}
\end{align*}
is uniformly integrable since
\begin{align*}
\bigg| &\Big(\mathscr{Z}^{\mathsmaller{\vartheta},\varepsilon}_{r,u}(\varphi)-\mathscr{Z}^{\mathsmaller{\vartheta},\varepsilon}_{r,s}\,\mathlarger{\mathlarger{\bullet}}_{t-s}\,\mathscr{Z}^{\mathsmaller{\vartheta},\varepsilon}_{t,u}(\varphi)\Big)\,h\Big(  \mathscr{Z}^{\mathsmaller{\vartheta},\varepsilon}_{s_1,t_1}(\varphi_1),\ldots, \mathscr{Z}^{\mathsmaller{\vartheta},\varepsilon}_{s_n,t_n}(\varphi_n) \Big)   \bigg|\\ & \,\leq\, \|h\|_{\infty}\,\Big|\mathscr{Z}^{\mathsmaller{\vartheta},\varepsilon}_{r,u}(\varphi)-\mathscr{Z}^{\mathsmaller{\vartheta},\varepsilon}_{r,s}\,\mathlarger{\mathlarger{\bullet}}_{t-s}\,\mathscr{Z}^{\mathsmaller{\vartheta},\varepsilon}_{t,u}(\varphi)\Big| \,\leq\, \|h\|_{\infty}\,\Big( \mathscr{Z}^{\mathsmaller{\vartheta},\varepsilon}_{r,u}\big(|\varphi|\big) \,+\, \mathscr{Z}^{\mathsmaller{\vartheta},\varepsilon}_{r,s}\,\mathlarger{\mathlarger{\bullet}}_{t-s}\,\mathscr{Z}^{\mathsmaller{\vartheta},\varepsilon}_{t,u}\big(|\varphi|\big)  \Big)
\end{align*}
and the families  $\big\{ \mathscr{Z}^{\mathsmaller{\vartheta},\varepsilon}_{r,u}\big(|\varphi|\big)  \big\}_{\varepsilon\in (0,1]}$ and  $\big\{ \mathscr{Z}^{\mathsmaller{\vartheta},\varepsilon}_{r,s}\,\mathlarger{\mathlarger{\bullet}}_{t-s}\,\mathscr{Z}^{\mathsmaller{\vartheta},\varepsilon}_{t,u}\big(|\varphi|\big)  \big\}_{\varepsilon\in (0,1]}$ are uniformly integrable. Here, the family  $\big\{ \mathscr{Z}^{\mathsmaller{\vartheta},\varepsilon}_{r,u}\big(|\varphi|\big)  \big\}_{\varepsilon\in (0,1]}$ is uniformly integrable because the second moments are uniformly bounded, which holds as a consequence of~\cite[Proposition 4.1]{BC}, for instance.  We can then conclude that  $\big\{ \mathscr{Z}^{\mathsmaller{\vartheta},\varepsilon}_{r,s}\,\mathlarger{\mathlarger{\bullet}}_{t-s}\,\mathscr{Z}^{\mathsmaller{\vartheta},\varepsilon}_{t,u}\big(|\varphi|\big)  \big\}_{\varepsilon\in (0,1]}$ is uniformly integrable in consequence of Lemma~\ref{LemmaEpsilonCase}. With this uniform integrability,  we can deduce~(\ref{NeedThis}) from~(\ref{ToGetThis}) and~(\ref{ThreeConv}), which completes the proof.
\end{proof}

\subsection{Proof of Proposition~\ref{PropCond2}}

\begin{proof}Since $\mathscr{Z}^{\mathsmaller{\vartheta}}_{r,s}  \,\mathlarger{\mathlarger{\bullet}}_{t-s}\,  \mathscr{Z}^{\mathsmaller{\vartheta}}_{t,u}$ is a conditional expectation of $\mathscr{Z}^{\mathsmaller{\vartheta}}_{r,u}$ by Proposition~\ref{PropCond}, we have the first equality below.
\begin{align*}
\mathbb{E}\Big[\, \big(\mathscr{Z}^{\mathsmaller{\vartheta}}_{r,u}-\mathscr{Z}^{\mathsmaller{\vartheta}}_{r,s}  \,\mathlarger{\mathlarger{\bullet}}_{t-s}\,  \mathscr{Z}^{\mathsmaller{\vartheta}}_{t,u}\big)^2 \,\Big]    & =\,\mathbb{E}\Big[ \big(\mathscr{Z}^{\mathsmaller{\vartheta}}_{r,u}\big)^2\Big]\,-\,\mathbb{E}\Big[\, \big(\mathscr{Z}^{\mathsmaller{\vartheta}}_{r,s}  \,\mathlarger{\mathlarger{\bullet}}_{t-s}\,  \mathscr{Z}^{\mathsmaller{\vartheta}}_{t,u}\big)^2 \, \Big] 
 \\   & =\,\mathbb{E}\Big[\big( \mathscr{Z}^{\mathsmaller{\vartheta}}_{r,u}\big)^2\Big]\,-\,\mathbb{E}\Big[ \big(\mathscr{Z}^{\mathsmaller{\vartheta}}_{r,s}\big)^2\Big] \,\mathlarger{\mathlarger{\bullet}}_{t-s}\,  \mathbb{E}\Big[\big( \mathscr{Z}^{\mathsmaller{\vartheta}}_{t,u}\big)^2\Big] \\
 & =\, \mathbf{Q}_{u-r}^{\mathsmaller{\vartheta}} \,-\,\mathbf{Q}_{s-r}^{\mathsmaller{\vartheta}} \,\mathlarger{\mathlarger{\bullet}}_{t-s}\,  \mathbf{Q}_{u-t}^{\mathsmaller{\vartheta}}
\end{align*}
Furthermore, the above can be expressed as 
\begin{align*}
 \mathbf{Q}_{u-r}^{\mathsmaller{\vartheta}} \,-\,\mathbf{Q}_{s-r}^{\mathsmaller{\vartheta}} \,\mathlarger{\mathlarger{\bullet}}_{t-s}\,  \mathbf{Q}_{u-t}^{\mathsmaller{\vartheta}}\, =\,&\, \mathbf{Q}_{s-r}^{\mathsmaller{\vartheta}}\,\mathlarger{\mathlarger{\bullet}}\,  \mathbf{Q}_{t-s}^{\mathsmaller{\vartheta}}\,\mathlarger{\mathlarger{\bullet}}\,   \mathbf{Q}_{u-t}^{\mathsmaller{\vartheta}} \,-\,\mathbf{Q}_{s-r}^{\mathsmaller{\vartheta}}\,\mathlarger{\mathlarger{\bullet}}\,  \mathbf{U}_{t-s}^2 \,\mathlarger{\mathlarger{\bullet}}\,   \mathbf{Q}_{u-t}^{\mathsmaller{\vartheta}} \\
 \, =\,&\, \mathbf{Q}_{s-r}^{\mathsmaller{\vartheta}}\,\mathlarger{\mathlarger{\bullet}}\,\big(  \mathbf{Q}_{t-s}^{\mathsmaller{\vartheta}}\,-\,  \mathbf{U}_{t-s}^2 \big)\,\mathlarger{\mathlarger{\bullet}}\,   \mathbf{Q}_{u-t}^{\mathsmaller{\vartheta}} \\
    \, =\,&\,\mathbf{Q}_{s-r}^{\mathsmaller{\vartheta}}  \,\mathlarger{\mathlarger{\bullet}}\,    \mathbf{K}_{t-s}^{\mathsmaller{\vartheta}}  \,\mathlarger{\mathlarger{\bullet}}\, \mathbf{Q}_{u-t}^{\mathsmaller{\vartheta}} \,,
\end{align*}
where we have used the semigroup property for $\big\{\mathbf{Q}_{a}^{\mathsmaller{\vartheta}} \big\}_{a\in [0,\infty)} $ and  that $\mathbf{K}_{a}^{\mathsmaller{\vartheta}}:=\mathbf{Q}_{a}^{\mathsmaller{\vartheta}}-\mathbf{U}_{a}^2$.
\end{proof}

\subsection{Proof of Proposition~\ref{PropChapman}}\label{SubSecChapman}

In several instances, our variance calculations lead to measures of the form $
\mathbf{Q}^{\mathsmaller{\vartheta}}_{s(\varsigma)}\,\mathlarger{\mathlarger{\bullet}}\, \mathbf{K}^{\mathsmaller{\vartheta}}_{\varsigma}     \,\mathlarger{\mathlarger{\bullet}}     \,\mathbf{Q}^{\mathsmaller{\vartheta}}_{t(\varsigma)}$, where  $s(\varsigma)\rightarrow s$ and $t(\varsigma)\rightarrow t$ as $\varsigma\rightarrow 0$.  These  measures vanish vaguely with small $\varsigma$, and in a slightly stronger topology, by the lemma below. 
Let $\widetilde{C}^{2}$ denote the set of real-valued continuous functions on $(\R^2\times \R^2)^2$ that are bounded in absolute value by a constant multiple of $\phi^n\otimes\phi^n $ for large enough $n\in \mathbb{N}$
\begin{lemma}\label{LemmaSqueeze} Fix $\vartheta\in \R$.   If the sequence $\{\varsigma_j\}_1^{\infty}\subset (0,\infty)$ vanishes and the sequences $\{s_j\}_1^{\infty},\{t_j\}_1^{\infty}\subset [0,\infty)$ are bounded, then  $
\mathbf{Q}^{\mathsmaller{\vartheta}}_{s_j}\,\mathlarger{\mathlarger{\bullet}}\, \mathbf{K}^{\mathsmaller{\vartheta}}_{\varsigma_j}     \,\mathlarger{\mathlarger{\bullet}}     \,\mathbf{Q}^{\mathsmaller{\vartheta}}_{t_j}$ converges to zero   $\widetilde{C}^{2}$-weakly. 
\end{lemma}
\begin{proof}
Since every element in $\widetilde{C}^2$ is bounded in absolute value by a multiple of $\phi^n\otimes \phi^n $ for large enough $n$, it is enough to show that $\mathbf{Q}^{\mathsmaller{\vartheta}}_{s_j}\,\mathlarger{\mathlarger{\bullet}}\, \mathbf{K}^{\mathsmaller{\vartheta}}_{\varsigma_j}     \,\mathlarger{\mathlarger{\bullet}}     \,\mathbf{Q}^{\mathsmaller{\vartheta}}_{t_j}(\phi^n\otimes\phi^n)$ vanishes as $j\rightarrow \infty$ for each $n$.    We can write
\begin{align}\label{QForm}
\mathbf{Q}^{\mathsmaller{\vartheta}}_{s_j}\,\mathlarger{\mathlarger{\bullet}}\, \mathbf{K}^{\mathsmaller{\vartheta}}_{\varsigma_j}     \,\mathlarger{\mathlarger{\bullet}}     \,\mathbf{Q}^{\mathsmaller{\vartheta}}_{t_j}\big(\phi^n\otimes\phi^n\big)\,=\,&\int_{(\R^2)^8}\,\phi^n(x_1,x_4)\,\phi^n(x_1',x_4')\,
\mathbf{\dot{Q}}^{\mathsmaller{\vartheta}}_{s_j}\big(x_1,x_1';x_2,x_2'\big)\,\mathbf{\dot{K}}^{\mathsmaller{\vartheta}}_{\varsigma_j} \big(x_2,x_2';x_3,x_3'\big)  \,  \nonumber \\ &\,\mathbf{\dot{Q}}^{\mathsmaller{\vartheta}}_{t_j}\big(x_3,x_3';x_4,x_4'\big)\,dx_1\,dx_1'\,dx_2\,dx_2'\,dx_3\,dx_3'\,dx_4\,dx_4'\,,
\end{align}
where $\mathbf{\dot{Q}}^{\mathsmaller{\vartheta}}_t$ and $\mathbf{\dot{K}}^{\mathsmaller{\vartheta}}_t$ are the Lebesgue densities of  $\mathbf{Q}^{\mathsmaller{\vartheta}}_t$ and $\mathbf{K}^{\mathsmaller{\vartheta}}_t$.  Recall that 
$\mathbf{\dot{Q}}^{\mathsmaller{\vartheta}}_t$ has  density~(\ref{QDensity}) and that  $\mathbf{K}^{\mathsmaller{\vartheta}}_{t}=\mathbf{Q}^{\mathsmaller{\vartheta}}_{t}-\mathbf{U}_{t}^2$ has  density
\[
\mathbf{\dot{K}}^{\mathsmaller{\vartheta}}_{t}\big(x,y;x',y'\big)\,=\,g_{t}\bigg(\frac{x+x'}{\sqrt{2}}-\frac{y+y'}{\sqrt{2}} 
  \bigg)\,K_{t}^{\mathsmaller{\vartheta}}\bigg(\frac{x-x'}{\sqrt{2}}, \frac{y-y'}{\sqrt{2}}  \bigg)\,,\hspace{.7cm} x,y,x',y'\in \R^2\,.
  \]
Using the inequality
\[
\phi^n(x,y)\,\phi^n(x',y') \,\leq\,   \phi^{2n}\bigg(\frac{x+x'}{\sqrt{2}} ,\frac{y+y'}{\sqrt{2}} \bigg)\,\phi^{2n}\bigg(\frac{x-x'}{\sqrt{2}} , \frac{y-y'}{\sqrt{2}} \bigg)  
\]
and switching to the integration variables $X_j=\frac{x_j-x_j'}{\sqrt{2}}$ and  $X_j'=\frac{x_j+x_j'}{\sqrt{2}}$ yields that~(\ref{QForm}) is bounded by
\begin{align}\label{NextBound}
C_j\,\int_{(\R^2)^4}\,\phi^{2n}(X_1, X_4)\,
 P^{\mathsmaller{\vartheta}}_{s_j}(X_1,X_2)\,K^{\mathsmaller{\vartheta}}_{\varsigma_j} (X_2, X_3)  \,  \,P^{\mathsmaller{\vartheta}}_{t_j}(X_3,X_4)\,dX_1\,dX_2\,dX_3\,dX_4\,,
\end{align}
where $\{C_j\}_1^{\infty}$ is the bounded sequence given by
\begin{align*}
C_j\,:=\, &\,\int_{(\R^2)^4}\,\phi^{2n}(X_1', X_4') \,
 g_{s_j}(X_1'-X_2')\,g_{\varsigma_j}(X_2'-X_3')  \,  g_{t_j}(X_3'-X_4')\,dX_1'\,dX_2'\,dX_3'\,dX_4' \\ \,=\,&\, \int_{(\R^2)^2}\,\phi^{2n}(X_1', X_4') \,
 g_{s_j+\varsigma_j+t_j}(X_1'-X_4')\,dX_1'\,dX_4'\,. 
 \end{align*}
For any $T>0$ there exists  a constant $\mathcal{C}_{T,\mathsmaller{\vartheta}}>0$ such that for all $\varsigma\in (0,T]$ and $x,y\in \R^2$ 
\begin{align}\label{KBound}
K_{\varsigma}^{\mathsmaller{\vartheta}}(x,y)\,\leq \, \frac{\mathcal{C}_{T,\mathsmaller{\vartheta}}}{\varsigma\big(1+\log^+\frac{1}{\varsigma}\big) }  \bigg( e^{-\frac{|x|^2}{2\varsigma}  }\,+\,\log^+ \frac{\varsigma}{|x|^2}    \bigg) \, \bigg( e^{-\frac{|y|^2}{2\varsigma}  }\,+\,\log^+ \frac{\varsigma}{|y|^2}    \bigg)   \,,
\end{align}
where $\log^+  $ is the positive component in the polar decomposition of $\log$; see~\cite[Proposition 8.17]{CM}, for instance.  The integral in~(\ref{NextBound})  vanishes as $j\rightarrow \infty$ in consequence of the bound~(\ref{KBound}) and the equality $P^{\mathsmaller{\vartheta}}_t(x,y)=g_t(x-y)+K^{\mathsmaller{\vartheta}}_t(x,y)$, where the the factor $ \big(1+\log^+\frac{1}{\varsigma}\big)^{-1}$ in~(\ref{KBound}) is the source of the decay.
\end{proof}

\vspace{.2cm}

\begin{proof}[Proof of Proposition~\ref{PropChapman}]
It follows from Proposition~\ref{PropCond2} and Lemma~\ref{LemmaSqueeze}  that $\mathscr{Z}^{\mathsmaller{\vartheta}}_{r,s}  \,\mathlarger{\mathlarger{\bullet}}_{\varsigma}\,  \mathscr{Z}^{\mathsmaller{\vartheta}}_{s+\varsigma,t}
$ converges to $
\mathscr{Z}^{\mathsmaller{\vartheta}}_{r,t}$ vaguely in $L^2$ as $\varsigma\rightarrow 0$.
We can thus focus on showing that the difference between $\mathscr{Z}^{\mathsmaller{\vartheta}}_{r,s}  \,\mathlarger{\mathlarger{\bullet}}_{\varsigma}\,  \mathscr{Z}^{\mathsmaller{\vartheta}}_{s+\varsigma,t}
$ and $\mathscr{Z}^{\mathsmaller{\vartheta}}_{r,s}  \,\mathlarger{\mathlarger{\bullet}}_{\varsigma}\,  \mathscr{Z}^{\mathsmaller{\vartheta}}_{s,t}
$
vanishes vaguely in $L^2$.  It follows from Proposition~\ref{PropCond} that
\[
\mathbb{E}\Big[\mathscr{Z}^{\mathsmaller{\vartheta}}_{r,s}  \,\mathlarger{\mathlarger{\bullet}}_{\varsigma}\,  \mathscr{Z}^{\mathsmaller{\vartheta}}_{s,t}\,\Big|\,\mathscr{F}_{[0,s]\cup [s+\varsigma,\infty)  }  \Big]= \mathscr{Z}^{\mathsmaller{\vartheta}}_{r,s}  \,\mathlarger{\mathlarger{\bullet}}_{\varsigma}\, \mathbb{E}\Big[ \mathscr{Z}^{\mathsmaller{\vartheta}}_{s,t}\,\Big|\,\mathscr{F}_{[0,s]\cup [s+\varsigma,\infty)  }  \Big]
   =\mathscr{Z}^{\mathsmaller{\vartheta}}_{r,s}  \,\mathlarger{\mathlarger{\bullet}}_{\varsigma}\,\big(\mathlarger{\mathlarger{\bullet}}_{\varsigma}  \mathscr{Z}^{\mathsmaller{\vartheta}}_{s+\varsigma,t}\big)=\mathscr{Z}^{\mathsmaller{\vartheta}}_{r,s}  \,\mathlarger{\mathlarger{\bullet}}_{2\varsigma}\,\mathscr{Z}^{\mathsmaller{\vartheta}}_{s+\varsigma,t}\,.
\]
Since $\mathscr{Z}^{\mathsmaller{\vartheta}}_{r,s}  \,\mathlarger{\mathlarger{\bullet}}_{2\varsigma}\,  \mathscr{Z}^{\mathsmaller{\vartheta}}_{s+\varsigma,t} $ is a conditional expectation of 
 $\mathscr{Z}^{\mathsmaller{\vartheta}}_{r,s}  \,\mathlarger{\mathlarger{\bullet}}_{\varsigma}\,  \mathscr{Z}^{\mathsmaller{\vartheta}}_{s,t}$, the second moment of their difference  can be expressed as 
\begin{align*}
\mathbb{E}\Big[\,\big( \mathscr{Z}^{\mathsmaller{\vartheta}}_{r,s}  \,\mathlarger{\mathlarger{\bullet}}_{\varsigma}\,  \mathscr{Z}^{\mathsmaller{\vartheta}}_{s,t}
 \,-\,\mathscr{Z}^{\mathsmaller{\vartheta}}_{r,s}  \,\mathlarger{\mathlarger{\bullet}}_{2\varsigma}\,  \mathscr{Z}^{\mathsmaller{\vartheta}}_{s+\varsigma,t} \big)^2  \, \Big]  =\,&\, \mathbb{E}\Big[\,\big( \mathscr{Z}^{\mathsmaller{\vartheta}}_{r,s}  \,\mathlarger{\mathlarger{\bullet}}_{\varsigma}\,  \mathscr{Z}^{\mathsmaller{\vartheta}}_{s,t}
  \big)^2 \, \Big]\,-\,\mathbb{E}\Big[\,\big( \mathscr{Z}^{\mathsmaller{\vartheta}}_{r,s}  \,\mathlarger{\mathlarger{\bullet}}_{2\varsigma}\,  \mathscr{Z}^{\mathsmaller{\vartheta}}_{s+\varsigma,t} \big)^2 \,  \Big]\\  =\,&\, \mathbb{E}\Big[\big(\mathscr{Z}^{\mathsmaller{\vartheta}}_{r,s} \big)^2    \Big]\,\mathlarger{\mathlarger{\bullet}}_{\varsigma}\,\mathbb{E}\Big[\big(\mathscr{Z}^{\mathsmaller{\vartheta}}_{s,t} \big)^2  \Big]\,-\,\mathbb{E}\Big[\big(\mathscr{Z}^{\mathsmaller{\vartheta}}_{r,s}  \big)^2    \Big]\,\mathlarger{\mathlarger{\bullet}}_{2\varsigma}\,\mathbb{E}\Big[\big(\mathscr{Z}^{\mathsmaller{\vartheta}}_{s+\varsigma,t} \big)^2   \Big]\\  =\,&\, \mathbf{Q}_{s-r}^{\mathsmaller{\vartheta}}\,\mathlarger{\mathlarger{\bullet}}_{\varsigma}\,\mathbf{Q}_{t-s}^{\mathsmaller{\vartheta}}\,-\,\mathbf{Q}_{s-r}^{\mathsmaller{\vartheta}}\,\mathlarger{\mathlarger{\bullet}}_{2\varsigma}\,\mathbf{Q}_{t-s-\varsigma}^{\mathsmaller{\vartheta}}\,,
 \end{align*}
 which   can be written in the form
 \begin{align*}
   \mathbf{Q}_{s-r}^{\mathsmaller{\vartheta}}\,\mathlarger{\mathlarger{\bullet}}_{\varsigma}\,\mathbf{Q}_{t-s}^{\mathsmaller{\vartheta}}\,-\,\mathbf{Q}_{s-r}^{\mathsmaller{\vartheta}}\,\mathlarger{\mathlarger{\bullet}}_{2\varsigma}\,\mathbf{Q}_{t-s-\varsigma}^{\mathsmaller{\vartheta}}\,=\,&\,   \mathbf{Q}_{s-r}^{\mathsmaller{\vartheta}}\,\mathlarger{\mathlarger{\bullet}}\, \mathbf{U}_{\varsigma}^2\,\mathlarger{\mathlarger{\bullet}}\,\mathbf{Q}_{t-s}^{\mathsmaller{\vartheta}}\,-\,\mathbf{Q}_{s-r}^{\mathsmaller{\vartheta}}\,\mathlarger{\mathlarger{\bullet}}\, \mathbf{U}_{2\varsigma}^2\,\mathlarger{\mathlarger{\bullet}}\,\mathbf{Q}_{t-s-\varsigma}^{\mathsmaller{\vartheta}}\\
   =\,&\, \mathbf{Q}_{s-r}^{\mathsmaller{\vartheta}}\,\mathlarger{\mathlarger{\bullet}}\, \mathbf{U}_{\varsigma}^2\,\mathlarger{\mathlarger{\bullet}}\,\mathbf{Q}_{\varsigma}^{\mathsmaller{\vartheta}}\,\mathlarger{\mathlarger{\bullet}}\,\mathbf{Q}_{t-s-\varsigma}^{\mathsmaller{\vartheta}}\,-\,\mathbf{Q}_{s-r}^{\mathsmaller{\vartheta}}\,\mathlarger{\mathlarger{\bullet}}\, \mathbf{U}_{\varsigma}^2\,\mathlarger{\mathlarger{\bullet}}\, \mathbf{U}_{\varsigma}^2\,\mathlarger{\mathlarger{\bullet}}\,\mathbf{Q}_{t-s-\varsigma}^{\mathsmaller{\vartheta}}\\
   =\,&\, \mathbf{Q}_{s-r}^{\mathsmaller{\vartheta}}\,\mathlarger{\mathlarger{\bullet}}\, \mathbf{U}_{\varsigma}^2\,\mathlarger{\mathlarger{\bullet}}\,\big(\mathbf{Q}_{\varsigma}^{\mathsmaller{\vartheta}}\,-\,\mathbf{U}_{\varsigma}^2\big)\,\mathlarger{\mathlarger{\bullet}}\,\mathbf{Q}_{t-s-\varsigma}^{\mathsmaller{\vartheta}}
   \\  =\,&\, \mathbf{Q}_{s-r}^{\mathsmaller{\vartheta}}\,\mathlarger{\mathlarger{\bullet}}\, \mathbf{U}_{\varsigma}^2\,\mathlarger{\mathlarger{\bullet}}\,\mathbf{K}_{\varsigma}^{\mathsmaller{\vartheta}}\,\mathlarger{\mathlarger{\bullet}}\,\mathbf{Q}_{t-s-\varsigma}^{\mathsmaller{\vartheta}}\\  \leq \,&\, \mathbf{Q}_{s-r}^{\mathsmaller{\vartheta}}\,\mathlarger{\mathlarger{\bullet}}\, \mathbf{Q}_{\varsigma}^{\mathsmaller{\vartheta}}\,\mathlarger{\mathlarger{\bullet}}\,\mathbf{K}_{\varsigma}^{\mathsmaller{\vartheta}}\,\mathlarger{\mathlarger{\bullet}}\,\mathbf{Q}_{t-s-\varsigma}^{\mathsmaller{\vartheta}}\\  = \,&\, \mathbf{Q}_{s-r+\varsigma}^{\mathsmaller{\vartheta}}\,\mathlarger{\mathlarger{\bullet}}\,\mathbf{K}_{\varsigma}^{\mathsmaller{\vartheta}}\,\mathlarger{\mathlarger{\bullet}}\,\mathbf{Q}_{t-s-\varsigma}^{\mathsmaller{\vartheta}}\,.
 \end{align*}
The above vanishes vaguely as $\varsigma\rightarrow 0$ by Lemma~\ref{LemmaSqueeze}.
 
It remains to show that the difference between $\mathscr{Z}^{\mathsmaller{\vartheta}}_{r,s}  \,\mathlarger{\mathlarger{\bullet}}_{2\varsigma}\,  \mathscr{Z}^{\mathsmaller{\vartheta}}_{s+\varsigma,t}$ and $\mathscr{Z}^{\mathsmaller{\vartheta}}_{r,s}  \,\mathlarger{\mathlarger{\bullet}}_{\varsigma}\,  \mathscr{Z}^{\mathsmaller{\vartheta}}_{s+\varsigma,t}$ vanishes vaguely in $L^2$ as $\varsigma\rightarrow 0$.  We can write
\begin{align} \label{SumFour}
\mathbb{E}\Big[\,\big( \mathscr{Z}^{\mathsmaller{\vartheta}}_{r,s}  \,\mathlarger{\mathlarger{\bullet}}_{2\varsigma}\,  \mathscr{Z}^{\mathsmaller{\vartheta}}_{s+\varsigma,t}
 -\mathscr{Z}^{\mathsmaller{\vartheta}}_{r,s}  \,\mathlarger{\mathlarger{\bullet}}_{\varsigma}\,  \mathscr{Z}^{\mathsmaller{\vartheta}}_{s+\varsigma,t} \big)^2   \,  \Big] \,=\,\sum_{a,b\in \{1,2\}}(-1)^{a+b}\,\lambda^{a,b}_{\varsigma}  
 \end{align}
for the measures
\begin{align*}
\lambda^{a,b}_{\varsigma} \,:=\, \mathbb{E}\Big[\,\big( \mathscr{Z}^{\mathsmaller{\vartheta}}_{r,s}  \,\mathlarger{\mathlarger{\bullet}}_{a\varsigma}\,  \mathscr{Z}^{\mathsmaller{\vartheta}}_{s+\varsigma,t}
  \big)\times \big( \mathscr{Z}^{\mathsmaller{\vartheta}}_{r,s}  \,\mathlarger{\mathlarger{\bullet}}_{b\varsigma}\,  \mathscr{Z}^{\mathsmaller{\vartheta}}_{s+\varsigma,t}
 \big) \,\Big]\,.
\end{align*}
Since the sum in~(\ref{SumFour}) includes two positive  and two negative terms, the second moment of the difference between $\mathscr{Z}^{\mathsmaller{\vartheta}}_{r,s}  \,\mathlarger{\mathlarger{\bullet}}_{2\varsigma}\,  \mathscr{Z}^{\mathsmaller{\vartheta}}_{s+\varsigma,t}$ and $\mathscr{Z}^{\mathsmaller{\vartheta}}_{r,s}  \,\mathlarger{\mathlarger{\bullet}}_{\varsigma}\,  \mathscr{Z}^{\mathsmaller{\vartheta}}_{s+\varsigma,t}$ vanishes vaguely in $L^2$ provided that  each measure
$\lambda^{a,b}_{\varsigma}$ converges vaguely  to $\mathbf{Q}_{t-r}^{\mathsmaller{\vartheta}}$.   The Lebesgue density $\dot{\lambda}^{a,b}_{\varsigma}$ of $\lambda^{a,b}_{\varsigma}$ can be expressed in terms of the Lebesgue density $\mathbf{\dot{Q}}_{u}^{\mathsmaller{\vartheta}}$ of $\mathbf{Q}_{u}^{\mathsmaller{\vartheta}}$ as
\begin{align*}
\dot{\lambda}^{a,b}_{\varsigma}(x,x';w,w')\,=\,&\, \int_{ (\R^2)^4  }\, \mathbf{\dot{Q}}_{s-r}^{\mathsmaller{\vartheta}}(x,x';y,y')\, g_{a\varsigma}(y-z) \,g_{b\varsigma}(y'-z')\,\mathbf{\dot{Q}}_{t-s-\varsigma}^{\mathsmaller{\vartheta}}(z,z';w,w')\,dy\,dy'\,dz\,dz'\,.
\end{align*}
The above converges  in $L^1_{\textup{loc}}$ as $\varsigma\rightarrow 0$ to 
\begin{align*}
 \int_{ (\R^2)^2  }\, \mathbf{\dot{Q}}_{s-r}^{\mathsmaller{\vartheta}}(x,x';y,y')\,\mathbf{\dot{Q}}_{t-s}^{\mathsmaller{\vartheta}}(y,y';w,w')\,dy\,dy'
\,=\, \mathbf{\dot{Q}}_{t-r}^{\mathsmaller{\vartheta}}(x,x';w,w')\,,
\end{align*}
and hence $\lambda^{a,b}_{\varsigma}$ converges vaguely  to $\mathbf{Q}_{t-r}^{\mathsmaller{\vartheta}}$.
\end{proof}

\subsection{Proof of Lemma~\ref{LemInd}}\label{SubsectLemInd}

 Recall that $\mathscr{F}_{A}$ for $A\subset [0,\infty)$ is defined as the $\sigma$-algebra generated by the family $\big\{ \mathscr{Z}^{\mathsmaller{\vartheta}}_{a,b} :  (a,b)\subset A \big\}$. 
\begin{proof} The $\sigma$-algebras $ \mathscr{F}_{[s,t]}$ and $ \mathscr{F}_{[0,s]\cup [t,\infty)}$ are independent  provided that  $\sigma \big\{ \mathscr{Z}^{\mathsmaller{\vartheta}}_{a_j,b_j} : 1\leq j\leq m\}$ and $\sigma \big\{ \mathscr{Z}^{\mathsmaller{\vartheta}}_{a_j',b_j'} : 1\leq j\leq n\}$ are independent for any subintervals $(a_1,b_1),\ldots, (a_m,b_m)  $  of $[s,t]$ and subintervals $(a_1',b_1'),\ldots, (a_n',b_n')  $ of $[0,s]\cup [t,\infty)$.  Let $\big\{(A_j,B_j) \big\}_1^M$  and $\big\{(A_j',B_j') \}_1^N$ be collections of disjoint intervals with
\[
\bigcup_1^m \,(a_j,b_j) \,\subset  \,\bigcup_1^M\, [A_j,B_j] \,\subset [s,t] \hspace{.9cm}\text{and}\hspace{.9cm}  \bigcup_1^n \, (a_j',b_j') \,\subset  \,\bigcup_1^N\, [A_j',B_j']\,\subset \,[0,s]\cup [t,\infty)\,.
\]
Then the $\sigma$-algebras  $\sigma \big\{ \mathscr{Z}^{\mathsmaller{\vartheta}}_{A_j,B_j} : 1\leq j\leq M\}$ and $\sigma \big\{ \mathscr{Z}^{\mathsmaller{\vartheta}}_{A_j',B_j'} : 1\leq j\leq N\}$ are independent by (iii) of Proposition~\ref{PropBasicProperties}.  Observe that if $P=\{t_0,\ldots, t_{k}\}$ is a partition of an interval $(a,b)$, then $\mathscr{Z}^{\mathsmaller{\vartheta}}_{a,b}$ is   $\sigma \big\{ \mathscr{Z}^{\mathsmaller{\vartheta}}_{t_{j-1},t_j} : 1\leq j\leq k\}$-measurable in consequence  of Proposition~\ref{PropChapman}.  Hence we have
\begin{align*}
\sigma \big\{ \mathscr{Z}^{\mathsmaller{\vartheta}}_{a_j,b_j} : 1\leq j\leq m\}\,\subset \,&\,\sigma \big\{ \mathscr{Z}^{\mathsmaller{\vartheta}}_{A_j,B_j} : 1\leq j\leq M\}  \\  \sigma \big\{ \mathscr{Z}^{\mathsmaller{\vartheta}}_{a_j',b_j'} : 1\leq j\leq n\} \,\subset \,&\, \sigma \big\{ \mathscr{Z}^{\mathsmaller{\vartheta}}_{A_j',B_j'} : 1\leq j\leq N\}\,,
\end{align*}
which implies that $\sigma \big\{ \mathscr{Z}^{\mathsmaller{\vartheta}}_{a_j,b_j} : 1\leq j\leq m\}$ and $ \sigma \big\{ \mathscr{Z}^{\mathsmaller{\vartheta}}_{a_j',b_j'} : 1\leq j\leq n\} $ are independent.
\end{proof}

\section{The multi-interval 2d SHF: construction and properties}\label{SecGenSHF}

We now turn to the proofs of Theorem~\ref{ThmSHFExtension} and Propositions~\ref{PropSHFExtensionProp} \& \ref{PropCondGen}. 
 As a preliminary, in Section~\ref{SubsectionMart} we introduce a slightly altered version, $\widetilde{\mathscr{Z}}_{P}^{\mathsmaller{\vartheta},\varsigma}$, of the  $\mathcal{M}_m$-valued random element $\mathscr{Z}_{P}^{\mathsmaller{\vartheta},\varsigma}$ in~(\ref{MartMulti}) such that the process $\big\{\widetilde{\mathscr{Z}}_{P}^{\mathsmaller{\vartheta},\varsigma}\}_{\varsigma>0}$ forms a backwards martingale.   We then define $\mathscr{Z}_{P}^{\mathsmaller{\vartheta}}$ as the vague  limit in $L^2$ of $\widetilde{\mathscr{Z}}_{P}^{\mathsmaller{\vartheta},\varsigma}$ as $\varsigma\rightarrow 0$, which reduces the proof of Theorem~\ref{ThmSHFExtension} in Section~\ref{SubsectThmSHF} to showing that  the difference between $\mathscr{Z}_{P}^{\mathsmaller{\vartheta},\varsigma}$ and $\widetilde{\mathscr{Z}}_{P}^{\mathsmaller{\vartheta},\varsigma}$ vanishes  vaguely in $L^2$ as $\varsigma\rightarrow 0$. The proofs of Propositions~\ref{PropSHFExtensionProp} \& \ref{PropCondGen} are in Section~\ref{SubsectPropCondGenProof}.

\subsection{Construction of the multi-interval 2d SHF through a martingale}\label{SubsectionMart}
Let $P=\{t_0,\ldots,t_m\}$ be a partition  of the interval $[s,t]$.  Define  $\Delta :=\min_{1\leq j\leq m}( t_j-t_{j-1})$.  Given $\varsigma\in (0,\Delta]$ define   the $\sigma$-algebra $
\textup{F}_{\varsigma}\,:=\,\mathscr{F}_{S(\varsigma)}
$ 
for the set $S(\varsigma)\subset [s,t]$ given by
\[
S(\varsigma)\,:=\,[t_0,t_1]\cup [ t_1+\varsigma,t_2]\cup\cdots \cup [t_{m-1}+\varsigma,t_m]\,, 
\]
and define the  $\mathcal{M}_m$-valued random element 
\begin{align*}
\widetilde{\mathscr{Z}}_{P}^{\mathsmaller{\vartheta},\varsigma}\,:=\,\mathscr{Z}^{\mathsmaller{\vartheta}}_{t_0,t_1}\,\mathlarger{\mathlarger{\circ}}_{\varsigma}\,\mathscr{Z}^{\mathsmaller{\vartheta}}_{t_1+\varsigma,t_2} \,\mathlarger{\mathlarger{\circ}}_{\varsigma} \, \cdots   \,\mathlarger{\mathlarger{\circ}}_{\varsigma}\, \mathscr{Z}^{\mathsmaller{\vartheta}}_{t_{m-1}+\varsigma,t_m} \,.
\end{align*}
For $\varsigma>\Delta$ we can simply put $\textup{F}_{\varsigma}:=\textup{F}_{\Delta}$ and $\widetilde{\mathscr{Z}}_{P}^{\mathsmaller{\vartheta},\varsigma}:=\widetilde{\mathscr{Z}}_{P}^{\mathsmaller{\vartheta},\Delta}$.
In the proof of Lemma~\ref{LemmaMartingale}, we will show that  $\{\widetilde{\mathscr{Z}}_{P}^{\mathsmaller{\vartheta},\varsigma}\}_{\varsigma\in (0,\Delta]}$ is a backwards martingale with respect to $\{\textup{F}_{\varsigma}\}_{\varsigma\in (0,\Delta]}$, having expectation $\mathbf{U}_P$ and second moment $\mathbf{Q}_{P}^{\mathsmaller{\vartheta},\varsigma}\in\mathcal{M}\big( (\R^2\times \R^2)^{m+1}\big)$ given by
\begin{align*}
\mathbf{Q}_{P}^{\mathsmaller{\vartheta},\varsigma}\,:=\,\mathbf{Q}_{t_1-t_0}^{\mathsmaller{\vartheta}}\,\mathlarger{\mathlarger{\circ}}_{\varsigma}\,\mathbf{Q}_{t_2-t_1-\varsigma }^{\mathsmaller{\vartheta}}\,\mathlarger{\mathlarger{\circ}}_{\varsigma}\,\cdots\,\mathlarger{\mathlarger{\circ}}_{\varsigma}\mathbf{Q}_{t_m-t_{m-1}-\varsigma}^{\mathsmaller{\vartheta}}\,.
\end{align*}
Observe that $\mathbf{Q}_{P}^{\mathsmaller{\vartheta},\varsigma}\leq \mathbf{Q}_{P}^{\mathsmaller{\vartheta}} $ since 
\begin{align}\label{QIneq}
\mathbf{Q}_{P}^{\mathsmaller{\vartheta},\varsigma}\,=\,&\,\mathbf{Q}_{t_1-t_0}^{\mathsmaller{\vartheta}}\,\mathlarger{\mathlarger{\circ}}\,\big(\mathbf{U}_{\varsigma}^2\,\mathlarger{\mathlarger{\bullet}} \, \mathbf{Q}_{t_2-t_1-\varsigma }^{\mathsmaller{\vartheta}} \big)\,\mathlarger{\mathlarger{\circ}}\,\cdots\,\mathlarger{\mathlarger{\circ}}\,\big(\mathbf{U}_{\varsigma}^2\,\mathlarger{\mathlarger{\bullet}} \, \mathbf{Q}_{t_m-t_{m-1}-\varsigma}^{\mathsmaller{\vartheta}}\big)\, \nonumber \\
\,\leq \,&\,\mathbf{Q}_{t_1-t_0}^{\mathsmaller{\vartheta}}\,\mathlarger{\mathlarger{\circ}}\,\big(\mathbf{Q}_{\varsigma}^{\mathsmaller{\vartheta}}\,\mathlarger{\mathlarger{\bullet}} \, \mathbf{Q}_{t_2-t_1-\varsigma }^{\mathsmaller{\vartheta}} \big)\,\mathlarger{\mathlarger{\circ}}\,\cdots\,\mathlarger{\mathlarger{\circ}}\,\big(\mathbf{Q}_{\varsigma}^{\mathsmaller{\vartheta}}\,\mathlarger{\mathlarger{\bullet}} \, \mathbf{Q}_{t_m-t_{m-1}-\varsigma}^{\mathsmaller{\vartheta}}\big)\nonumber \\
\,=\,&\,\mathbf{Q}_{t_1-t_0}^{\mathsmaller{\vartheta}}\,\mathlarger{\mathlarger{\circ}}\, \mathbf{Q}_{t_2-t_1}^{\mathsmaller{\vartheta}} \,\mathlarger{\mathlarger{\circ}}\,\cdots\,\mathlarger{\mathlarger{\circ}}\,\mathbf{Q}_{t_m-t_{m-1}}^{\mathsmaller{\vartheta}}\,=:\,\mathbf{Q}_{P}^{\mathsmaller{\vartheta}}\,.
\end{align}
A similar calculation shows that $\mathbf{Q}_{P}^{\mathsmaller{\vartheta},\varsigma}$ is decreasing in $\varsigma$.  Finally, $\mathbf{Q}_{P}^{\mathsmaller{\vartheta},\varsigma}$ converges vaguely to $\mathbf{Q}_{P}^{\mathsmaller{\vartheta}}$ as $\varsigma\rightarrow 0$ since the difference  between their Lebesgue densities vanishes in  $L^1_{\textup{loc}}\big( (\R^2\times \R^2)^{m+1} \big)$.

Let $\boldsymbol{\mu}_{j}$ for $j\in \mathbb{N}\cup\{\infty\}$ be $\widetilde{\mathcal{M}}_{m}$-valued random elements on some probability space.  We say that the sequence $\{\boldsymbol{\mu}_j\}_1^{\infty}$ \textit{converges to} $\boldsymbol{\mu}_{\infty}$ $\widetilde{C}_m$\textit{-weakly in $L^p$}  when $\boldsymbol{\mu}_j(\varphi)\rightarrow \boldsymbol{\mu}_{\infty}(\varphi)$ in $L^p$ for all $\varphi\in \widetilde{C}_m$.  Note that this implies that $\{\boldsymbol{\mu}_j\}_1^{\infty}$ converges to $\boldsymbol{\mu}_{\infty}$  vaguely in $L^p$ since  $\widetilde{C}_m$ contains all compactly supported continuous functions on $(\R^2)^{m+1}$.
\begin{lemma}\label{LemmaMartingale} Fix $\vartheta\in \R$ and $0\leq s<t$.  Let $P$ be a partition of the interval $[s,t]$. The  $\mathcal{M}_{m }$-valued random element $\widetilde{\mathscr{Z}}_{P}^{\mathsmaller{\vartheta},\varsigma}$ converges $\widetilde{C}_m$-weakly in $L^2$ as $\varsigma\rightarrow 0$ to a limit $\mathscr{Z}_{P}^{\mathsmaller{\vartheta}}$ having first and second moments $\mathbf{U}_P$ and $\mathbf{Q}_{P}^{\mathsmaller{\vartheta}}$.
\end{lemma}

\begin{proof} For   $P=\{t_0,\ldots, t_m\}$ the first moment of $\widetilde{\mathscr{Z}}_{P}^{\mathsmaller{\vartheta},\varsigma}$ is $\mathbf{U}_P$   since  for $\varsigma\in (0,\Delta]$
\begin{align*}
\mathbb{E}\big[\,\widetilde{\mathscr{Z}}_{P}^{\mathsmaller{\vartheta},\varsigma}\,\big]\,=\,&\,\mathbb{E}\big[\,\mathscr{Z}^{\mathsmaller{\vartheta}}_{t_0,t_1}\,\big]\,\mathlarger{\mathlarger{\circ}}_{\varsigma}\,\mathbb{E}\big[\,\mathscr{Z}^{\mathsmaller{\vartheta}}_{t_1+\varsigma,t_2}\,\big] \,\mathlarger{\mathlarger{\circ}}_{\varsigma} \, \cdots   \,\mathlarger{\mathlarger{\circ}}_{\varsigma}\, \mathbb{E}\big[\,\mathscr{Z}^{\mathsmaller{\vartheta}}_{t_{m-1}+\varsigma,t_m}\,\big]\\
\,=\,&\,\mathbf{U}_{t_1-t_0} \,\mathlarger{\mathlarger{\circ}}_{\varsigma}\,\mathbf{U}_{t_2-t_1-\varsigma} \,\mathlarger{\mathlarger{\circ}}_{\varsigma} \, \cdots   \,\mathlarger{\mathlarger{\circ}}_{\varsigma}\, \mathbf{U}_{t_m-t_{m-1}-\varsigma}\\
\,=\,&\,\mathbf{U}_{t_1-t_0} \,\mathlarger{\mathlarger{\circ}}\,\big(\mathbf{U}_{\varsigma} \,\mathlarger{\mathlarger{\bullet}}\, \mathbf{U}_{t_2-t_1-\varsigma} \big)\,\mathlarger{\mathlarger{\circ}} \, \cdots   \,\mathlarger{\mathlarger{\circ}}\, \big(\mathbf{U}_{\varsigma} \,\mathlarger{\mathlarger{\bullet}}\,\mathbf{U}_{t_m-t_{m-1}-\varsigma}\big)\\
\,=\,&\,\mathbf{U}_{t_1-t_0} \,\mathlarger{\mathlarger{\circ}}\,\mathbf{U}_{t_2-t_1} \,\mathlarger{\mathlarger{\circ}} \, \cdots   \,\mathlarger{\mathlarger{\circ}}\, \mathbf{U}_{t_m-t_{m-1}} \,=:\, \mathbf{U}_P\,,
\end{align*}
in which the first equality uses that $\mathscr{Z}^{\mathsmaller{\vartheta}}_{t_0,t_1 } $, $\mathscr{Z}^{\mathsmaller{\vartheta}}_{t_1+\varsigma,t_2} $,\ldots, $\mathscr{Z}^{\mathsmaller{\vartheta}}_{t_{m-1}+\varsigma,t_m}$ are independent. Similarly, the second moment of $\widetilde{\mathscr{Z}}_{P}^{\mathsmaller{\vartheta},\varsigma}$ is
\begin{align*}
\mathbb{E}\Big[\,\big(\widetilde{\mathscr{Z}}_{P}^{\mathsmaller{\vartheta},\varsigma}\big)^2 \,\Big]\,=\,&\, \mathbb{E}\Big[\,\big(\mathscr{Z}^{\mathsmaller{\vartheta}}_{t_0,t_1} \big)^2\,\Big]\,\mathlarger{\mathlarger{\circ}}_{\varsigma} \,\mathbb{E}\Big[\,\big(\mathscr{Z}^{\mathsmaller{\vartheta}}_{t_1+\varsigma,t_2}  \big)^2\,\Big]\,\mathlarger{\mathlarger{\circ}}_{\varsigma} \, \cdots   \,\mathlarger{\mathlarger{\circ}}_{\varsigma}\, \mathbb{E}\Big[\,\big(\mathscr{Z}^{\mathsmaller{\vartheta}}_{t_{m-1}+\varsigma,t_m} \big)^2\,\Big]\\ \,=\,&\, \mathbf{Q}_{t_1-t_0}^{\mathsmaller{\vartheta}}\,\mathlarger{\mathlarger{\circ}}_{\varsigma}\,\mathbf{Q}_{t_2-t_1-\varsigma}^{\mathsmaller{\vartheta}}\,\mathlarger{\mathlarger{\circ}}_{\varsigma}\,\cdots\,\mathlarger{\mathlarger{\circ}}_{\varsigma}\mathbf{Q}_{t_m-t_{m-1}-\varsigma}^{\mathsmaller{\vartheta}} \,=:\,\mathbf{Q}_{P}^{\mathsmaller{\vartheta},\varsigma}\,.
\end{align*}
To see that    $\{\widetilde{\mathscr{Z}}_{P}^{\mathsmaller{\vartheta},\varsigma}\}_{\varsigma\in (0,\Delta]}$ is a backwards martingale with respect to $\{\textup{F}_{\varsigma}\}_{\varsigma\in (0,\Delta]}$, observe that 
for $0<\varsigma' <\varsigma\leq \Delta$
\begin{align}
\mathbb{E}\left[\widetilde{\mathscr{Z}}_{P}^{\mathsmaller{\vartheta},\varsigma'}\,\Big|\, \textup{F}_{\varsigma}\right]\,=\,&\,\mathbb{E}\big[\,\mathscr{Z}^{\mathsmaller{\vartheta}}_{t_0,t_1} \,\big|\, \textup{F}_{\varsigma}\,\big]\,\mathlarger{\mathlarger{\circ}}_{\varsigma'}\,\mathbb{E}\big[\,\mathscr{Z}^{\mathsmaller{\vartheta}}_{t_1+\varsigma',t_2} \,\big|\, \textup{F}_{\varsigma}\,\big]\,\mathlarger{\mathlarger{\circ}}_{\varsigma'} \, \cdots   \,\mathlarger{\mathlarger{\circ}}_{\varsigma'}\, \mathbb{E}\big[\,\mathscr{Z}^{\mathsmaller{\vartheta}}_{t_{m-1}+\varsigma',t_m}  \,\big|\, \textup{F}_{\varsigma}\,\big] \nonumber   \\
\,=\,&\,\mathscr{Z}^{\mathsmaller{\vartheta}}_{t_0,t_1} \,\mathlarger{\mathlarger{\circ}}_{\varsigma'}\,\big(\mathlarger{\mathlarger{\bullet}}_{\varsigma-\varsigma'}\,\mathscr{Z}^{\mathsmaller{\vartheta}}_{t_1+\varsigma,t_2} \big)\,\mathlarger{\mathlarger{\circ}}_{\varsigma'} \, \cdots   \,\mathlarger{\mathlarger{\circ}}_{\varsigma'}\, \big(\mathlarger{\mathlarger{\bullet}}_{\varsigma-\varsigma'}\mathscr{Z}^{\mathsmaller{\vartheta}}_{t_{m-1}+\varsigma,t_m}  \big) \nonumber \\
\,=\,&\,\mathscr{Z}^{\mathsmaller{\vartheta}}_{t_0,t_1} \,\mathlarger{\mathlarger{\circ}}_{\varsigma}\,\mathscr{Z}^{\mathsmaller{\vartheta}}_{t_1+\varsigma,t_2} \,\mathlarger{\mathlarger{\circ}}_{\varsigma} \, \cdots   \,\mathlarger{\mathlarger{\circ}}_{\varsigma}\, \mathscr{Z}^{\mathsmaller{\vartheta}}_{t_{m-1}+\varsigma,t_m} \,=:\, \widetilde{\mathscr{Z}}_{P}^{\mathsmaller{\vartheta},\varsigma} \,,\label{CondExpZ}
\end{align}
where the second  equality above holds by Proposition~\ref{PropCond}.

 Next we will construct $\mathscr{Z}_{P}^{\mathsmaller{\vartheta}}$ as an almost sure limit of the sequence $\big\{\widetilde{\mathscr{Z}}_{P}^{\mathsmaller{\vartheta},1/n}\big\}_{n\in \mathbb{N}}$.  Given $\varphi\in \widetilde{C}_m$, the sequence
$\big\{\widetilde{\mathscr{Z}}_{P}^{\mathsmaller{\vartheta},1/n}(\varphi)\big\}_{n\in \mathbb{N}}$ is a  martingale with expectation $\mathbf{U}_P(\varphi)$ and uniformly bounded  second moments  since
\begin{align*}
\mathbb{E}\Big[\,\big(\widetilde{\mathscr{Z}}_{P}^{\mathsmaller{\vartheta},1/n}(\varphi)\big)^2\,\Big]\,=\, \mathbf{Q}_{P}^{\mathsmaller{\vartheta},1/n}(\varphi\otimes \varphi) \,\stackrel{(\ref{QIneq}) }{\leq} \,\mathbf{Q}_{P}^{\mathsmaller{\vartheta}}(\varphi\otimes \varphi) \,<\,\infty\,.
\end{align*}
 Hence the sequence of random variables $\big\{\widetilde{\mathscr{Z}}_{P}^{\mathsmaller{\vartheta},1/n}(\varphi)\big\}_{n\in \mathbb{N}}$ converges almost surely and in $L^2$ to a limit $L_{\varphi}$.  A standard argument, which we will review below, shows that there is a random measure
 $\mathscr{Z}_{P}^{\mathsmaller{\vartheta}}$ such that $ \mathscr{Z}_{P}^{\mathsmaller{\vartheta}}(\varphi)= L_{\varphi }$ holds almost surely for each $\varphi\in \widetilde{C}_m$.  The backwards martingale property for the family  $\{\widetilde{\mathscr{Z}}_{P}^{\mathsmaller{\vartheta},\varsigma}\}_{\varsigma\in (0,\Delta] }$ implies that $\mathbb{E}\big[\mathscr{Z}_{P}^{\mathsmaller{\vartheta}}\, | \,\textup{F}_{\varsigma}\big]=\widetilde{\mathscr{Z}}_{P}^{\mathsmaller{\vartheta,\varsigma}} $ for all $\varsigma\in (0,\Delta]$.  The second moment of $\mathscr{Z}_{P}^{\mathsmaller{\vartheta}}$ is the vague limit of $\mathbf{Q}_{P}^{\mathsmaller{\vartheta},\varsigma}$ as $\varsigma\rightarrow 0$, which is $\mathbf{Q}_{P}^{\mathsmaller{\vartheta}}$.\vspace{.2cm}

 Clearly, for a given $\alpha\in \R$ and $\varphi'\in C_c\big( (\R^2)^{m+1} \big)$, we have $L_{\alpha\varphi+\varphi'}=\alpha L_{\varphi}+L_{\varphi'}$ almost surely, and  $L_{\varphi}\geq 0$ holds almost surely when $\varphi\geq 0$. 
Let $D^+$ be a countable dense subset of  $C_c\big( (\R^2)^{m+1},[0,\infty) \big)$ such that each $\varphi\in C_c\big( (\R^2)^{m+1},[0,\infty) \big)$ is a uniform limit of a sequence  $\{ \varphi_j \}_1^{\infty}$ in $ D^+$ for which $\bigcup_1^{\infty}\textup{supp}(\varphi_j)$ is bounded.  We can also stipulate that $D^+$  is closed under nonnegative rational linear combinations. Put $D=D^++(-D^+)$.  There exists a full measure subset $\Omega'$ of $\Omega$ such that for every $\omega\in \Omega'$
\begin{itemize}
\item $L_{\varphi}(\omega)\geq 0$ for all nonnegative $\varphi\in D$.

\item $L_{\alpha\varphi+\varphi'}(\omega)= \alpha L_{\varphi}(\omega)+L_{\varphi'}$ for all $\alpha\in \mathbb{Q}$ and $\varphi,\varphi'\in D$.
  
\end{itemize}
Thus, for any fixed $\omega\in \Omega'$, the map $\varphi\rightarrow L_{\varphi}(\omega)$ extends uniquely to a nonnegative linear functional on $C_c\big( (\R^2)^{m+1} \big)$. By the Riesz representation theorem, there is a unique  measure $\mathscr{Z}_{P}^{\mathsmaller{\vartheta}}(\omega,\cdot)$ on $(\R^2)^{m+1}$ satisfying $\mathscr{Z}_{P}^{\mathsmaller{\vartheta}}(\omega,\varphi)=L_{\varphi}(\omega)$ for all $\varphi\in C_c\big( (\R^2)^{m+1} \big)$, and  we put $\mathscr{Z}_{P}^{\mathsmaller{\vartheta}}(\omega)=0$ for $\omega\in \Omega\backslash \Omega'$.
\end{proof}

\subsection{Proof of Theorem~\ref{ThmSHFExtension}}\label{SubsectThmSHF}

\begin{proof} Due to Lemma~\ref{LemmaMartingale}, it suffices for us to show that the  difference between $\mathscr{Z}_{P}^{\mathsmaller{\vartheta},\varsigma}$ and $\widetilde{\mathscr{Z}}_{P}^{\mathsmaller{\vartheta},\varsigma}$ vanishes vaguely in $L^2$ as $\varsigma\rightarrow 0$.  We will do this by showing that the second moments of the differences $\mathscr{Z}_{P}^{\mathsmaller{\vartheta},\varsigma}-\overline{\mathscr{Z}}_{P}^{\mathsmaller{\vartheta},\varsigma}$ and $\overline{\mathscr{Z}}_{P}^{\mathsmaller{\vartheta},\varsigma}-\widetilde{\mathscr{Z}}_{P}^{\mathsmaller{\vartheta},\varsigma}$ vanish under the vague topology, where $\overline{\mathscr{Z}}^{\mathsmaller{\vartheta},\varsigma}_P:= \mathbb{E}\big[ \mathscr{Z}^{\mathsmaller{\vartheta}}_P \,\big| \,\textup{F}_{\varsigma} \big]$. Using Proposition~\ref{PropCond}, we can compute
\begin{align*}
\overline{\mathscr{Z}}^{\mathsmaller{\vartheta},\varsigma}_P \,=\,& \,
\mathbb{E}\big[\,\mathscr{Z}^{\mathsmaller{\vartheta}}_{t_0,t_1}  \,\big| \,\textup{F}_{\varsigma} \,\big]  \,\mathlarger{\mathlarger{\circ}}_{\varsigma}\,\mathbb{E}\big[\,\mathscr{Z}^{\mathsmaller{\vartheta}}_{t_1,t_2}  \,\big| \,\textup{F}_{\varsigma} \,\big]  \,\mathlarger{\mathlarger{\circ}}_{\varsigma}\,\cdots \,\mathlarger{\mathlarger{\circ}}_{\varsigma}\,\mathbb{E}\big[\,\mathscr{Z}^{\mathsmaller{\vartheta}}_{t_{m-1},t_m} \,\big| \,\textup{F}_{\varsigma}\, \big]  \\ \,=\,& \,
\mathscr{Z}^{\mathsmaller{\vartheta}}_{t_0,t_1}    \,\mathlarger{\mathlarger{\circ}}_{\varsigma}\,\big(\mathlarger{\mathlarger{\bullet}}_{\varsigma}\,\mathscr{Z}^{\mathsmaller{\vartheta}}_{t_{1}+\varsigma,t_2} \big)   \,\mathlarger{\mathlarger{\circ}}_{\varsigma}\,\cdots\,\mathlarger{\mathlarger{\circ}}_{\varsigma}\,\big(\mathlarger{\mathlarger{\bullet}}_{\varsigma}\,\mathscr{Z}^{\mathsmaller{\vartheta}}_{t_{m-1}+\varsigma,t_m} \big)  \\ \,=\,& \,
\mathscr{Z}^{\mathsmaller{\vartheta}}_{t_0,t_1}    \,\mathlarger{\mathlarger{\circ}}_{2\varsigma}\,\mathscr{Z}^{\mathsmaller{\vartheta}}_{t_{1}+\varsigma,t_2}    \,\mathlarger{\mathlarger{\circ}}_{2\varsigma}\,\cdots\,\mathlarger{\mathlarger{\circ}}_{2\varsigma}\,\mathscr{Z}^{\mathsmaller{\vartheta}}_{t_{m-1}+\varsigma,t_m}  \,.
\end{align*}
Next, we can express the second moment of the difference between $\mathscr{Z}_{P}^{\mathsmaller{\vartheta},\varsigma} $ and $\overline{\mathscr{Z}}_{P}^{\mathsmaller{\vartheta},\varsigma}$ in terms of the telescoping sum below
\begin{align}\label{Telescoping}
\mathbb{E}\Big[\, \big(\mathscr{Z}_{P}^{\mathsmaller{\vartheta},\varsigma}-\overline{\mathscr{Z}}_{P}^{\mathsmaller{\vartheta},\varsigma}\big)^2 \, \Big] \,=\,
\mathbb{E}\Big[\, \big(\mathscr{Z}_{P}^{\mathsmaller{\vartheta},\varsigma} \big)^2\,\Big] \,-\,\mathbb{E}\Big[\,\big(\overline{\mathscr{Z}}_{P}^{\mathsmaller{\vartheta},\varsigma}\big)^2\,\Big] 
 \,=\,\sum_{k=2}^m\,
\mathbb{E}\Big[ \,\big(\overline{\mathscr{Z}}_{P}^{\mathsmaller{\vartheta},\varsigma,k}\big)^2 \,\Big] \,-\,\mathbb{E}\Big[ \,\big(\overline{\mathscr{Z}}_{P}^{\mathsmaller{\vartheta},\varsigma,k-1}\big)^2\,\Big]\,,
\end{align}
wherein $\overline{\mathscr{Z}}^{\mathsmaller{\vartheta},\varsigma,m}_P :=\mathscr{Z}_{P}^{\mathsmaller{\vartheta},\varsigma}$ and for $k\in \{1,\ldots, m-1\}$
\begin{align*}
\overline{\mathscr{Z}}^{\mathsmaller{\vartheta},\varsigma,k}_P\,:=\,\mathscr{Z}^{\mathsmaller{\vartheta}}_{t_0,t_1} \mathlarger{\mathlarger{\circ}}_{\varsigma}\,\cdots\,\mathlarger{\mathlarger{\circ}}_{\varsigma}\,\mathscr{Z}^{\mathsmaller{\vartheta}}_{t_{k-1},t_k}\,\mathlarger{\mathlarger{\circ}}_{2\varsigma}\,\mathscr{Z}^{\mathsmaller{\vartheta}}_{t_k+\varsigma,t_{k+1}} \,\mathlarger{\mathlarger{\circ}}_{2\varsigma}\,\cdots \,\mathlarger{\mathlarger{\circ}}_{2\varsigma}\,\mathscr{Z}^{\mathsmaller{\vartheta}}_{t_{m-1}+\varsigma,t_m}\,.
\end{align*}
We can express the second moment of $\overline{\mathscr{Z}}_{P}^{\mathsmaller{\vartheta},\varsigma,k}$ as
\begin{align}\label{Pill}
\mathbb{E}\Big[ \,\big(\overline{\mathscr{Z}}_{P}^{\mathsmaller{\vartheta},\varsigma,k}\big)^2 \,\Big] \,=\,&\,\mathbb{E}\Big[\,\big(\mathscr{Z}^{\mathsmaller{\vartheta}}_{t_0,t_1}\big)^2\,\Big]\,\mathlarger{\mathlarger{\circ}}_{\varsigma}\, \cdots \,\mathlarger{\mathlarger{\circ}}_{\varsigma}\,\mathbb{E}\Big[\,\big(\mathscr{Z}^{\mathsmaller{\vartheta}}_{t_{k-1},t_k}\big)^2\,\Big]\nonumber  \\ &\, \,\mathlarger{\mathlarger{\circ}}_{2\varsigma}\,\mathbb{E}\Big[\,\big(\mathscr{Z}^{\mathsmaller{\vartheta}}_{t_k+\varsigma,t_{k+1}}\big)^2\,\Big]\,\mathlarger{\mathlarger{\circ}}_{2\varsigma}\, \cdots \,\mathlarger{\mathlarger{\circ}}_{2\varsigma}\,\mathbb{E}\Big[\,\big(\mathscr{Z}^{\mathsmaller{\vartheta},\varsigma}_{t_{m-1}+\varsigma,t_m}\big)^2\,\Big]\nonumber \\
\,=\,&\,\mathbf{Q}^{\mathsmaller{\vartheta}}_{t_{1}-t_{0}}\,\mathlarger{\mathlarger{\circ}}_{\varsigma}\, \cdots \,\mathlarger{\mathlarger{\circ}}_{\varsigma}\,\mathbf{Q}^{\mathsmaller{\vartheta}}_{t_{k}-t_{k-1}}\,\mathlarger{\mathlarger{\circ}}_{2\varsigma}\,\mathbf{Q}^{\mathsmaller{\vartheta}}_{t_{k+1}-t_{k}-\varsigma}\,\mathlarger{\mathlarger{\circ}}_{2\varsigma}\, \cdots \,\mathlarger{\mathlarger{\circ}}_{2\varsigma}\,\mathbf{Q}^{\mathsmaller{\vartheta}}_{t_m-t_{m-1}-\varsigma}\,.
\end{align}
Thus a single term from the telescoping sum~(\ref{Telescoping}) can be bounded as
\begin{align}\label{Pill2}
\mathbb{E}\Big[\, \big(\overline{\mathscr{Z}}_{P}^{\mathsmaller{\vartheta},\varsigma,k}\big)^2 \,\Big]   \,-\, \mathbb{E}\Big[\, \big(\overline{\mathscr{Z}}_{P}^{\mathsmaller{\vartheta},\varsigma,k-1}\big)^2\, \Big]   =\,&\, \mathbf{Q}^{\mathsmaller{\vartheta}}_{t_{1}-t_{0}}\,\mathlarger{\mathlarger{\circ}}_{\varsigma}\, \cdots \,\mathlarger{\mathlarger{\circ}}_{\varsigma}\,\mathbf{Q}^{\mathsmaller{\vartheta}}_{t_{k-1}-t_{k-2}} \,\mathlarger{\mathlarger{\circ}}_{\varsigma}\,\big(\mathbf{Q}^{\mathsmaller{\vartheta}}_{t_{k}-t_{k-1}}-\mathlarger{\mathlarger{\bullet}}_{\varsigma}\mathbf{Q}^{\mathsmaller{\vartheta}}_{t_{k}-t_{k-1}-\varsigma}\big)\, \nonumber  \\ &\,\,\mathlarger{\mathlarger{\circ}}_{2\varsigma}\,\mathbf{Q}^{\mathsmaller{\vartheta}}_{t_{k+1}-t_{k}-\varsigma}\,\mathlarger{\mathlarger{\circ}}_{2\varsigma}\, \cdots \,\mathlarger{\mathlarger{\circ}}_{2\varsigma}\,\mathbf{Q}^{\mathsmaller{\vartheta}}_{t_m-t_{m-1}-\varsigma} \nonumber  \\  \leq\,&\, \mathbf{Q}^{\mathsmaller{\vartheta}}_{t_{1}-t_{0}}\,\mathlarger{\mathlarger{\circ}}\,\mathbf{Q}^{\mathsmaller{\vartheta}}_{t_{2}-t_{1}+\varsigma}\,\mathlarger{\mathlarger{\circ}}\, \cdots \,\mathlarger{\mathlarger{\circ}}\,\mathbf{Q}^{\mathsmaller{\vartheta}}_{t_{k-1}-t_{k-2}+\varsigma} \,\mathlarger{\mathlarger{\circ}}\,\big(\mathbf{Q}^{\mathsmaller{\vartheta}}_{\varsigma} \,\mathlarger{\mathlarger{\bullet}}\,\mathbf{K}^{\mathsmaller{\vartheta}}_{\varsigma}\,\mathlarger{\mathlarger{\bullet}}\,\mathbf{Q}^{\mathsmaller{\vartheta}}_{t_{k}-t_{k-1}-\varsigma}\big)\, \nonumber  \\ &\,\,\mathlarger{\mathlarger{\circ}}\,\mathbf{Q}^{\mathsmaller{\vartheta}}_{t_{k+1}-t_{k}+\varsigma}\,\mathlarger{\mathlarger{\circ}}\, \cdots \,\mathlarger{\mathlarger{\circ}}\,\mathbf{Q}^{\mathsmaller{\vartheta}}_{t_m-t_{m-1}+\varsigma} \,,
\end{align}
where the inequality uses that $\mathlarger{\mathlarger{\circ}}_{r}=\mathlarger{\mathlarger{\circ}}\,\mathbf{U}_{r}^2\,\mathlarger{\mathlarger{\bullet}}$ and $\mathbf{U}_{r}^2\leq \mathbf{Q}^{\mathsmaller{\vartheta}}_{r}$, along with the routine computation
\[
\mathbf{Q}^{\mathsmaller{\vartheta}}_{r}-\mathlarger{\mathlarger{\bullet}}_{\varsigma}\mathbf{Q}^{\mathsmaller{\vartheta}}_{r-\varsigma} \,=\,\mathbf{Q}^{\mathsmaller{\vartheta}}_{\varsigma}\,\mathlarger{\mathlarger{\bullet}}\,\mathbf{Q}^{\mathsmaller{\vartheta}}_{r-\varsigma}-\mathbf{U}_{\varsigma}^2\,\mathlarger{\mathlarger{\bullet}}\,\mathbf{Q}^{\mathsmaller{\vartheta}}_{r-\varsigma} \,=\,\big(\mathbf{Q}^{\mathsmaller{\vartheta}}_{\varsigma}\,-\,\mathbf{U}_{\varsigma}^2\big)\,\mathlarger{\mathlarger{\bullet}}\,\mathbf{Q}^{\mathsmaller{\vartheta}}_{r-\varsigma} \,=\, 
 \mathbf{K}^{\mathsmaller{\vartheta}}_{\varsigma}\, \mathlarger{\mathlarger{\bullet}}
\,\mathbf{Q}^{\mathsmaller{\vartheta}}_{r-\varsigma}\,.
\]
The projection of the measure~(\ref{Pill2}) that results from integrating out  the middle $m-1$ components of the $(m+1)$-fold Cartesian product 
$(\R^2\times \R^2)^{m+1}$ is given by
\begin{align}
\mathbf{Q}^{\mathsmaller{\vartheta}}_{t_{1}-t_{0}}\,&\,\mathlarger{\mathlarger{\bullet}}\,\mathbf{Q}^{\mathsmaller{\vartheta}}_{t_{2}-t_{1}+\varsigma}\,\mathlarger{\mathlarger{\bullet}}\, \cdots \,\mathlarger{\mathlarger{\bullet}}\,\mathbf{Q}^{\mathsmaller{\vartheta}}_{t_{k-1}-t_{k-2}+\varsigma} \,\mathlarger{\mathlarger{\bullet}}\,\big(\mathbf{Q}^{\mathsmaller{\vartheta}}_{\varsigma} \,\mathlarger{\mathlarger{\bullet}}\,\mathbf{K}^{\mathsmaller{\vartheta}}_{\varsigma}\,\mathlarger{\mathlarger{\bullet}}\,\mathbf{Q}^{\mathsmaller{\vartheta}}_{t_{k}-t_{k-1}-\varsigma}\big)\,\mathlarger{\mathlarger{\bullet}}\,\mathbf{Q}^{\mathsmaller{\vartheta}}_{t_{k+1}-t_{k}+\varsigma}\,\mathlarger{\mathlarger{\bullet}}\, \cdots \,\mathlarger{\mathlarger{\bullet}}\,\mathbf{Q}^{\mathsmaller{\vartheta}}_{t_m-t_{m-1}+\varsigma}\nonumber \\
&\,=\,\mathbf{Q}^{\mathsmaller{\vartheta}}_{t_{k-1}-t_{0}+(k-1)\varsigma }\,\mathlarger{\mathlarger{\bullet}}\,\mathbf{K}^{\mathsmaller{\vartheta}}_{\varsigma}\,\mathlarger{\mathlarger{\bullet}}\,\mathbf{Q}^{\mathsmaller{\vartheta}}_{t_m-t_{k-1}+(m-k-1)\varsigma}\,.\label{Fipel}
\end{align}
Since the measure~(\ref{Fipel}) vanishes vaguely as $\varsigma\rightarrow 0$ by Lemma~\ref{LemmaSqueeze}, so does~(\ref{Pill2}), and hence the second moment of the difference between $\mathscr{Z}_{P}^{\mathsmaller{\vartheta}} $ and $\overline{\mathscr{Z}}_{P}^{\mathsmaller{\vartheta},\varsigma}$ vanishes vaguely.\vspace{.2cm}

Now we bound the difference between $\overline{\mathscr{Z}}_{P}^{\mathsmaller{\vartheta},\varsigma}$ and $\widetilde{\mathscr{Z}}_{P}^{\mathsmaller{\vartheta},\varsigma} $, which we can write as the telescoping sum
\[
\overline{\mathscr{Z}}_{P}^{\mathsmaller{\vartheta},\varsigma} -\widetilde{\mathscr{Z}}_{P}^{\mathsmaller{\vartheta},\varsigma}\,=\, \sum_{k=2}^m\,\widetilde{\mathscr{Z}}_{P}^{\mathsmaller{\vartheta},\varsigma,k}\,-\,\widetilde{\mathscr{Z}}_{P}^{\mathsmaller{\vartheta},\varsigma,k-1}\,,
\]
where $\widetilde{\mathscr{Z}}^{\mathsmaller{\vartheta},\varsigma,m}_P:=\overline{\mathscr{Z}}_{P}^{\mathsmaller{\vartheta},\varsigma}$,  $\widetilde{\mathscr{Z}}^{\mathsmaller{\vartheta},\varsigma,1}_P:=\widetilde{\mathscr{Z}}_{P}^{\mathsmaller{\vartheta},\varsigma}$ , and for 
$k\in \{2,\ldots, m-1\}$ we define
\begin{align*}
\widetilde{\mathscr{Z}}^{\mathsmaller{\vartheta},\varsigma,k}_P\,:=\,\mathscr{Z}^{\mathsmaller{\vartheta}}_{t_0,t_1}\,\mathlarger{\mathlarger{\circ}}_{2\varsigma}\,\mathscr{Z}^{\mathsmaller{\vartheta}}_{t_{1}+\varsigma,t_2}\,\mathlarger{\mathlarger{\circ}}_{2\varsigma}\,\cdots \,\mathlarger{\mathlarger{\circ}}_{2\varsigma}\,\mathscr{Z}^{\mathsmaller{\vartheta}}_{t_{k-1}+\varsigma,t_k}\,\mathlarger{\mathlarger{\circ}}_{\varsigma}\,\mathscr{Z}^{\mathsmaller{\vartheta}}_{t_{k}+\varsigma ,t_{k+1}}\,\mathlarger{\mathlarger{\circ}}_{\varsigma}\,\cdots\,\mathlarger{\mathlarger{\circ}}_{\varsigma}\,\mathscr{Z}^{\mathsmaller{\vartheta}}_{t_{m-1}+\varsigma,t_m}\,.
\end{align*}
 The second moment of the difference between $\widetilde{\mathscr{Z}}_{P}^{\mathsmaller{\vartheta},\varsigma,k}$ and $\widetilde{\mathscr{Z}}_{P}^{\mathsmaller{\vartheta},\varsigma,k-1}$ is
\begin{align}\label{Tlbit}
\mathbb{E}\Big[\, \big(\widetilde{\mathscr{Z}}_{P}^{\mathsmaller{\vartheta},\varsigma,k}-\widetilde{\mathscr{Z}}_{P}^{\mathsmaller{\vartheta},\varsigma,k-1}\big)^2\, \Big] \,=\,\sum_{a,b\in \{0,1\}}\,(-1)^{a+b}\,\lambda^{a,b}_{\varsigma,k} 
\end{align}
for the measures
$
\lambda^{a,b}_{\varsigma,k} := \mathbb{E}\big[\widetilde{\mathscr{Z}}_{P}^{\mathsmaller{\vartheta},\varsigma,k-a}\times \widetilde{\mathscr{Z}}_{P}^{\mathsmaller{\vartheta},\varsigma,k-b} \big]
$, which we can express as 
\begin{align*}
&\lambda^{a,b}_{\varsigma,k}\big(dx_0,dx_0';\ldots;dx_m,dx_m'\big)\\
&\,=\,\int_{y,y'\in \R^2}\mathbf{Q}^{\mathsmaller{\vartheta}}_{t_{1}-t_{0}}\,\mathlarger{\mathlarger{\circ}}_{2\varsigma}\,\mathbf{Q}^{\mathsmaller{\vartheta}}_{t_{2}-t_{1}-\varsigma}\,\mathlarger{\mathlarger{\circ}}_{2\varsigma}\, \cdots \,\mathlarger{\mathlarger{\circ}}_{2\varsigma}\,\mathbf{Q}^{\mathsmaller{\vartheta}}_{t_{k}-t_{k-1}-\varsigma}\big(dx_0,dx_0'; dx_1,dx_1';\ldots;dx_k,dx_k'\big)\,\\
&\text{}\hspace{.2cm}\,g_{\varsigma+a\varsigma}(x_k-y)\,g_{\varsigma+b\varsigma}(x_k'-y')\,\mathbf{Q}^{\mathsmaller{\vartheta}}_{t_{k+1}-t_{k}-\varsigma}\,\mathlarger{\mathlarger{\circ}}_{\varsigma}\, \cdots \,\mathlarger{\mathlarger{\circ}}_{\varsigma}\,\mathbf{Q}^{\mathsmaller{\vartheta}}_{t_m-t_{m-1}-\varsigma}\big(dy,dy';dx_{k+1},dx_{k+1}';\ldots;dx_m,dx_m'\big)\,.
\end{align*}
Each $\lambda^{a,b}_{\varsigma,k}$ converges vaguely to $\mathbf{Q}^{\mathsmaller{\vartheta}}_{P}$, implying that second moment of the difference between $\widetilde{\mathscr{Z}}_{P}^{\mathsmaller{\vartheta},\varsigma,k}$ and $\widetilde{\mathscr{Z}}_{P}^{\mathsmaller{\vartheta},\varsigma,k-1}$ converges to zero.  Thus the difference between $\overline{\mathscr{Z}}_{P}^{\mathsmaller{\vartheta},\varsigma}$ and $\widetilde{\mathscr{Z}}_{P}^{\mathsmaller{\vartheta},\varsigma} $ vanishes vaguely, completing the proof.
\end{proof}

\subsection{Proofs of Propositions~\ref{PropSHFExtensionProp} and~\ref{PropCondGen}}\label{SubsectPropCondGenProof}

\begin{proof}[Proof of Proposition~\ref{PropCondGen}] Assume that $s,t\notin  P$.  For $P=\{t_0,\ldots, t_m\}$ let $k,\ell$ be the largest elements in $\{0,\ldots, m\}$ such that  $t_k<s$ and $t_{\ell}<t$.  We will assume that  $0<k<\ell<m$, as other cases  differ only notationally.  By Lemma~\ref{LemmaMartingale} $\widetilde{\mathscr{Z}}_{P}^{\mathsmaller{\vartheta},\varsigma} $ converges vaguely in $L^2$ to $\mathscr{Z}_{P}^{\mathsmaller{\vartheta}}  $ as $\varsigma\rightarrow 0$.  It follows that $\mathbb{E}\big[\widetilde{\mathscr{Z}}_{P}^{\mathsmaller{\vartheta},\varsigma}\big|\mathscr{F}_{[0,s]\cup [t,\infty)  }\big] $ converges vaguely in $L^2$ to $\mathbb{E}\big[\mathscr{Z}_{P}^{\mathsmaller{\vartheta}}\big|\mathscr{F}_{[0,s]\cup [t,\infty)  }\big] $.  Taking the conditional expectation of 
\begin{align*}
\widetilde{\mathscr{Z}}_{P}^{\mathsmaller{\vartheta},\varsigma}\,=\,&\,\mathscr{Z}^{\mathsmaller{\vartheta}}_{t_0,t_1}\,\mathlarger{\mathlarger{\circ}}_{\varsigma}\,\mathscr{Z}^{\mathsmaller{\vartheta}}_{t_1+\varsigma,t_2} \,\mathlarger{\mathlarger{\circ}}_{\varsigma} \, \cdots   \,\mathlarger{\mathlarger{\circ}}_{\varsigma}\, \mathscr{Z}^{\mathsmaller{\vartheta}}_{t_{k-1}+\varsigma,t_k}\,\mathlarger{\mathlarger{\circ}}_{\varsigma}\, \mathscr{Z}^{\mathsmaller{\vartheta}}_{t_{k}+\varsigma,t_{k+1}}\,\mathlarger{\mathlarger{\circ}}_{\varsigma}\, \mathscr{Z}^{\mathsmaller{\vartheta}}_{t_{k+1}+\varsigma,t_{k+2}}\,\\ 
 &\,\mathlarger{\mathlarger{\circ}}_{\varsigma} \, \cdots   \,\mathlarger{\mathlarger{\circ}}_{\varsigma}\, \mathscr{Z}^{\mathsmaller{\vartheta}}_{t_{\ell-1}+\varsigma,t_{\ell}}\,\mathlarger{\mathlarger{\circ}}_{\varsigma}\, \mathscr{Z}^{\mathsmaller{\vartheta}}_{t_{\ell}+\varsigma,t_{\ell+1}}\,\mathlarger{\mathlarger{\circ}}_{\varsigma}\, \mathscr{Z}^{\mathsmaller{\vartheta}}_{t_{\ell+1}+\varsigma,t_{\ell+2}}\,\mathlarger{\mathlarger{\circ}}_{\varsigma} \, \cdots   \,\mathlarger{\mathlarger{\circ}}_{\varsigma}\, \mathscr{Z}^{\mathsmaller{\vartheta}}_{t_{m-1}+\varsigma,t_{m}}
\end{align*}
with respect to $ \mathscr{F}_{[0,s]\cup [t,\infty)  } $ and applying Proposition~\ref{PropCond} to the $\mathscr{Z}^{\mathsmaller{\vartheta}}_{t_{k}+\varsigma,t_{k+1}}$ and $\mathscr{Z}^{\mathsmaller{\vartheta}}_{t_{\ell}+\varsigma,t_{\ell+1}}$ terms yields
\begin{align*}
\mathbb{E}\Big[\,\widetilde{\mathscr{Z}}_{P}^{\mathsmaller{\vartheta},\varsigma}\,\Big|\,\mathscr{F}_{[0,s]\cup [t,\infty)  }\,\Big] \,=\,&\,\mathscr{Z}^{\mathsmaller{\vartheta}}_{t_0,t_1}\,\mathlarger{\mathlarger{\circ}}_{\varsigma}\,\mathscr{Z}^{\mathsmaller{\vartheta}}_{t_1+\varsigma,t_2} \,\mathlarger{\mathlarger{\circ}}_{\varsigma} \, \cdots   \,\mathlarger{\mathlarger{\circ}}_{\varsigma}\, \mathscr{Z}^{\mathsmaller{\vartheta}}_{t_{k-1}+\varsigma,t_k}\,\mathlarger{\mathlarger{\circ}}_{\varsigma}\, \big(\mathscr{Z}^{\mathsmaller{\vartheta}}_{t_{k}+\varsigma,s}\,\mathlarger{\mathlarger{\bullet}}_{t_{k+1}-s }  \big)
\,\mathlarger{\mathlarger{\circ}}_{\varsigma}\,\mathbf{U}_{t_{k+2}-t_{k+1}-\varsigma} \\ & \,\mathlarger{\mathlarger{\circ}}_{\varsigma}\, \cdots   \,\mathlarger{\mathlarger{\circ}}_{\varsigma}\, \mathbf{U}_{t_{\ell}-t_{\ell-1}-\varsigma}\,\mathlarger{\mathlarger{\circ}}_{\varsigma}\,\big( \mathlarger{\mathlarger{\bullet}}_{t-t_{\ell}-\varsigma} \,  \mathscr{Z}^{\mathsmaller{\vartheta}}_{t,t_{\ell+1}}\big)\,\mathlarger{\mathlarger{\circ}}_{\varsigma} \,\mathscr{Z}^{\mathsmaller{\vartheta}}_{t_{\ell+1}+\varsigma,t_{\ell+2}}\, \mathlarger{\mathlarger{\circ}}_{\varsigma} \,\cdots   \,\mathlarger{\mathlarger{\circ}}_{\varsigma}\, \mathscr{Z}^{\mathsmaller{\vartheta}}_{t_{m-1}+\varsigma,t_{m}}\\
\,=\,&\,\big(\mathscr{Z}^{\mathsmaller{\vartheta}}_{t_0,t_1}\,\mathlarger{\mathlarger{\circ}}_{\varsigma}\,\mathscr{Z}^{\mathsmaller{\vartheta}}_{t_1+\varsigma,t_2} \,\mathlarger{\mathlarger{\circ}}_{\varsigma} \, \cdots   \,\mathlarger{\mathlarger{\circ}}_{\varsigma}\, \mathscr{Z}^{\mathsmaller{\vartheta}}_{t_{k-1}+\varsigma,t_k}\,\mathlarger{\mathlarger{\circ}}_{\varsigma}\,\mathscr{Z}^{\mathsmaller{\vartheta}}_{t_{k}+\varsigma,s}\big)\,\mathlarger{\mathlarger{\bullet}}\, \big(\mathbf{U}_{t_{k+1}-s }  
\,\mathlarger{\mathlarger{\circ}}\,\mathbf{U}_{t_{k+2}-t_{k+1}} \\ & \,\mathlarger{\mathlarger{\circ}}\, \cdots   \,\mathlarger{\mathlarger{\circ}}\, \mathbf{U}_{t_{\ell}-t_{\ell-1}}\,\mathlarger{\mathlarger{\circ}}\,\mathbf{U}_{t-t_{\ell}}\big)\,\mathlarger{\mathlarger{\bullet}}\,  \big(\mathscr{Z}^{\mathsmaller{\vartheta}}_{t,t_{\ell+1}}\,\mathlarger{\mathlarger{\circ}}_{\varsigma} \,\mathscr{Z}^{\mathsmaller{\vartheta}}_{t_{\ell+1}+\varsigma,t_{\ell+2}}\, \mathlarger{\mathlarger{\circ}}_{\varsigma} \,\cdots   \,\mathlarger{\mathlarger{\circ}}_{\varsigma}\, \mathscr{Z}^{\mathsmaller{\vartheta}}_{t_{m-1}+\varsigma,t_{m}}\big)\\
\,=\,&\, \widetilde{\mathscr{Z}}_{P_{\leq s}}^{\mathsmaller{\vartheta},\varsigma}\, \mathlarger{\mathlarger{\bullet}} \, \mathbf{U}_{P_{[s,t]}}  \, \mathlarger{\mathlarger{\bullet}}\,\widetilde{\mathscr{Z}}_{P_{\geq t}}^{\mathsmaller{\vartheta},\varsigma}\,.
\end{align*}
Again by Lemma~\ref{LemmaMartingale}, we have $\widetilde{\mathscr{Z}}_{P_{\leq s}}^{\mathsmaller{\vartheta},\varsigma}\rightarrow \mathscr{Z}_{P_{\leq s}}^{\mathsmaller{\vartheta}} $ $\tilde{C}_k$-weakly in $L^2$ and $\widetilde{\mathscr{Z}}_{P_{\geq t}}^{\mathsmaller{\vartheta},\varsigma}\rightarrow \mathscr{Z}_{P_{\geq t}}^{\mathsmaller{\vartheta}} $ $\tilde{C}_{m-\ell}$-weakly in $L^2$, which implies that $\widetilde{\mathscr{Z}}_{P_{\leq s}}^{\mathsmaller{\vartheta},\varsigma}\, \mathlarger{\mathlarger{\bullet}}\, \mathbf{U}_{P_{[s,t]}}  \, \mathlarger{\mathlarger{\bullet}} \,\widetilde{\mathscr{Z}}_{P_{\geq t}}^{\mathsmaller{\vartheta},\varsigma}$ converges $\tilde{C}_{m}$-weakly in $L^2$ to $\mathscr{Z}_{P_{\leq s}}^{\mathsmaller{\vartheta}}\, \mathlarger{\mathlarger{\bullet}} \, \mathbf{U}_{P_{[s,t]}}  \, \mathlarger{\mathlarger{\bullet}}\,\mathscr{Z}_{P_{\geq t}}^{\mathsmaller{\vartheta}}$. Hence $\mathbb{E}\big[\mathscr{Z}_{P}^{\mathsmaller{\vartheta}}\big|\mathscr{F}_{[0,s]\cup [t,\infty)  }\big]$ is equal to  $\mathscr{Z}_{P_{\leq s}}^{\mathsmaller{\vartheta}}\, \mathlarger{\mathlarger{\bullet}} \, \mathbf{U}_{P_{[s,t]}}  \, \mathlarger{\mathlarger{\bullet}}\,\mathscr{Z}_{P_{\geq t}}^{\mathsmaller{\vartheta}}$.  Now we can compute
\begin{align*}
\mathbb{E}\Big[\, \big(  \mathscr{Z}_{P}^{\mathsmaller{\vartheta}}-\mathscr{Z}_{P_{\leq s}}^{\mathsmaller{\vartheta}}\, \mathlarger{\mathlarger{\bullet}} \, \mathbf{U}_{P_{[s,t]}}  \, \mathlarger{\mathlarger{\bullet}}\,\mathscr{Z}_{P_{\geq t}}^{\mathsmaller{\vartheta}} \big)^2  \,\Big]\,=\,&\,\mathbb{E}\Big[ \,\big(  \mathscr{Z}_{P}^{\mathsmaller{\vartheta}}\big)^2  \,\Big]\,-\,\mathbb{E}\Big[ \,\big(  \mathscr{Z}_{P_{\leq s}}^{\mathsmaller{\vartheta}}\, \mathlarger{\mathlarger{\bullet}} \, \mathbf{U}_{P_{[s,t]}}  \, \mathlarger{\mathlarger{\bullet}}\,\mathscr{Z}_{P_{\geq t}}^{\mathsmaller{\vartheta}} \big)^2 \, \Big]\\
\,=\,&\,\mathbb{E}\Big[ \,\big(  \mathscr{Z}_{P}^{\mathsmaller{\vartheta}}\big)^2 \, \Big]\,-\,\mathbb{E}\Big[\,\big(\mathscr{Z}_{P_{\leq s}}^{\mathsmaller{\vartheta}}\big)^2\,\Big] \, \mathlarger{\mathlarger{\bullet}} \,\mathbf{U}_{P_{[s,t]}}^2\,\mathlarger{\mathlarger{\bullet}} \,\mathbb{E}\Big[\,\big(\mathscr{Z}_{P_{\geq t}}^{\mathsmaller{\vartheta}}\big)^2\,\Big]
\\
\,=\,&\ \mathbf{Q}_{P}^{\mathsmaller{\vartheta}} \,-\,\mathbf{Q}_{P_{\leq s}}^{\mathsmaller{\vartheta}} \, \mathlarger{\mathlarger{\bullet}} \,\mathbf{U}_{P_{[s,t]}}^2\,\mathlarger{\mathlarger{\bullet}} \,\mathbf{Q}_{P_{\geq t}}^{\mathsmaller{\vartheta}}\\
\,=\,&\, \mathbf{Q}_{P_{\leq s}}^{\mathsmaller{\vartheta}} \, \mathlarger{\mathlarger{\bullet}} \,\mathbf{Q}_{P_{[s,t]}}^{\mathsmaller{\vartheta}}\,\mathlarger{\mathlarger{\bullet}} \,\mathbf{Q}_{P_{\geq t}}^{\mathsmaller{\vartheta}}\,-\,\mathbf{Q}_{P_{\leq s}}^{\mathsmaller{\vartheta}} \, \mathlarger{\mathlarger{\bullet}} \,\mathbf{U}_{P_{[s,t]}}^2\,\mathlarger{\mathlarger{\bullet}} \,\mathbf{Q}_{P_{\geq t}}^{\mathsmaller{\vartheta}}\\
\,=\,&\, \mathbf{Q}_{P_{\leq s}}^{\mathsmaller{\vartheta}} \, \mathlarger{\mathlarger{\bullet}} \,\big(\mathbf{Q}_{P_{[s,t]}}^{\mathsmaller{\vartheta}}\,-\,\mathbf{U}_{P_{[s,t]}}^2\big)\,\mathlarger{\mathlarger{\bullet}} \,\mathbf{Q}_{P_{\geq t}}^{\mathsmaller{\vartheta}}\\
\,=\,&\, \mathbf{Q}_{P_{\leq s}}^{\mathsmaller{\vartheta}} \, \mathlarger{\mathlarger{\bullet}} \,\mathbf{K}_{P_{[s,t]}}^{\mathsmaller{\vartheta}}\,\mathlarger{\mathlarger{\bullet}} \,\mathbf{Q}_{P_{\geq t}}^{\mathsmaller{\vartheta}}\,.
\end{align*}
Thus we have the desired results when $s,t\notin P$, and the cases when $s \in P$ or $t \in P$ require only small modifications of the above. 
\end{proof}

\begin{proof}[Proof of Proposition~\ref{PropSHFExtensionProp}] 
Statement (i) holds by Lemma~\ref{LemmaMartingale}, and (ii)--(iv) follow closely from the definition of $\mathscr{Z}^{\mathsmaller{\vartheta}}_{P}$ as the vague limit in $L^2$ of $\mathscr{Z}^{\mathsmaller{\vartheta},\varsigma}_{P}$, defined in~(\ref{MartMulti}), and Proposition~\ref{PropBasicProperties}.\vspace{.2cm}

\noindent (v) We need only focus on the case that the refining partition $P'$ contains exactly one more point, $a$, than $P$.  Then there exists $k\in \{0,\ldots,m-1\}$ such that $a\in (t_k,t_{k+1})$.  Since $\widetilde{\mathscr{Z}}^{\mathsmaller{\vartheta},\varepsilon}_{P'} $ converges $\widetilde{C}_{\ell}$-weakly in $L^2$ to $\mathscr{Z}^{\mathsmaller{\vartheta}}_{P'} $ by Lemma~\ref{LemmaMartingale}, the projection $\widetilde{\mathscr{Z}}^{\mathsmaller{\vartheta},\varepsilon}_{P'}\circ \Pi_{P',P }^{-1} $ converges $\widetilde{C}_m$-weakly in $L^2$ to $\mathscr{Z}^{\mathsmaller{\vartheta}}_{P'}\circ \Pi_{P',P }^{-1} $.  Thus we only need to show that $\widetilde{\mathscr{Z}}^{\mathsmaller{\vartheta},\varepsilon}_{P'}\circ \Pi_{P',P }^{-1} $ converges vaguely in $L^2$ to $\mathscr{Z}^{\mathsmaller{\vartheta}}_{P} $.  Since $\widetilde{\mathscr{Z}}^{\mathsmaller{\vartheta},\varepsilon}_{P} $ converges $\widetilde{C}_m$-weakly in $L^2$ to $\mathscr{Z}^{\mathsmaller{\vartheta}}_{P} $, it suffices to show that the difference between $\widetilde{\mathscr{Z}}^{\mathsmaller{\vartheta},\varepsilon}_{P'}\circ \Pi_{P',P }^{-1} $ and $\widetilde{\mathscr{Z}}^{\mathsmaller{\vartheta},\varepsilon}_{P} $ vanishes vaguely in $L^2$.  The second moment of the difference
\begin{align*}
\widetilde{\mathscr{Z}}^{\mathsmaller{\vartheta},\varepsilon}_{P'}\circ \Pi_{P',P }^{-1}\,-\,\widetilde{\mathscr{Z}}^{\mathsmaller{\vartheta},\varepsilon}_{P}\,=\,&\,\mathscr{Z}^{\mathsmaller{\vartheta}}_{t_0,t_1}\,\mathlarger{\mathlarger{\circ}}_{\varsigma}\,\mathscr{Z}^{\mathsmaller{\vartheta}}_{t_1+\varsigma,t_2} \,\mathlarger{\mathlarger{\circ}}_{\varsigma} \, \cdots   \,\mathlarger{\mathlarger{\circ}}_{\varsigma}\,\mathscr{Z}^{\mathsmaller{\vartheta}}_{t_{k-1}+\varsigma,t_k} \,\mathlarger{\mathlarger{\circ}}_{\varsigma} \,\big(\mathscr{Z}^{\mathsmaller{\vartheta}}_{t_k+\varsigma,a} \,\mathlarger{\mathlarger{\bullet}}_{\varsigma} \,\mathscr{Z}^{\mathsmaller{\vartheta}}_{a+\varsigma,t_{k+1}} \,-\,\mathscr{Z}^{\mathsmaller{\vartheta}}_{t_k+\varsigma,t_{k+1}} \big)\\ \,& \,\mathlarger{\mathlarger{\circ}}_{\varsigma}\,\mathscr{Z}^{\mathsmaller{\vartheta}}_{t_{k+1}+\varsigma,t_{k+2}} \,\mathlarger{\mathlarger{\circ}}_{\varsigma} \, \cdots   \,\mathlarger{\mathlarger{\circ}}_{\varsigma}\,\mathscr{Z}^{\mathsmaller{\vartheta}}_{t_{m-1}+\varsigma,t_m} \,.
\end{align*}
can be expressed as
\begin{align}\label{Pinel}
\mathbf{Q}^{\mathsmaller{\vartheta}}_{t_1-t_0}\,&\,\mathlarger{\mathlarger{\circ}}_{\varsigma}\,\mathbf{Q}^{\mathsmaller{\vartheta}}_{t_2-t_1-\varsigma} \,\mathlarger{\mathlarger{\circ}}_{\varsigma} \, \cdots   \,\mathlarger{\mathlarger{\circ}}_{\varsigma}\,\mathbf{Q}^{\mathsmaller{\vartheta}}_{t_k-t_{k-1}-\varsigma} \,\mathlarger{\mathlarger{\circ}}_{\varsigma}\,\mathbb{E}\Big[\,\big(\mathscr{Z}^{\mathsmaller{\vartheta}}_{t_k+\varsigma,a} \,\mathlarger{\mathlarger{\bullet}}_{\varsigma} \,\mathscr{Z}^{\mathsmaller{\vartheta}}_{a+\varsigma,t_{k+1}} \,-\,\mathscr{Z}^{\mathsmaller{\vartheta}}_{t_k+\varsigma,t_{k+1}} \big)^2\,\Big]\nonumber \\ 
&\,\mathlarger{\mathlarger{\circ}}_{\varsigma}\,\mathbf{Q}^{\mathsmaller{\vartheta}}_{t_{k+2}-t_{k+1}-\varsigma} \,\mathlarger{\mathlarger{\circ}}_{\varsigma} \, \cdots   \,\mathlarger{\mathlarger{\circ}}_{\varsigma}\,\mathbf{Q}^{\mathsmaller{\vartheta}}_{t_m- t_{m-1}-\varsigma} \,.
\end{align}
By Proposition~\ref{PropCond2} we have the equality
\[
\mathbb{E}\Big[\,\big(\mathscr{Z}^{\mathsmaller{\vartheta}}_{t_k+\varsigma,a} \,\mathlarger{\mathlarger{\bullet}}_{\varsigma} \,\mathscr{Z}^{\mathsmaller{\vartheta}}_{a+\varsigma,t_{k+1}} \,-\,\mathscr{Z}^{\mathsmaller{\vartheta}}_{t_k+\varsigma,t_{k+1}} \,\big)^2\,\Big]\,=\,\mathbf{Q}^{\mathsmaller{\vartheta}}_{a-t_{k}-\varsigma} \,\mathlarger{\mathlarger{\bullet}}\,\mathbf{K}^{\mathsmaller{\vartheta}}_{\varsigma} \,\mathlarger{\mathlarger{\bullet}}\,\mathbf{Q}^{\mathsmaller{\vartheta}}_{t_{k+1}-a-\varsigma} \,.
\]
Hence, by using the usual relations $\mathlarger{\mathlarger{\circ}}_{\varsigma}=\mathlarger{\mathlarger{\circ}}\, \mathbf{U}_{\varsigma}^2\,\mathlarger{\mathlarger{\bullet}}$ and $\mathbf{U}_{\varsigma}^2\leq \mathbf{Q}^{\mathsmaller{\vartheta}}_{\varsigma}$ and  $ \mathbf{Q}^{\mathsmaller{\vartheta}}_{r}=  \mathbf{Q}^{\mathsmaller{\vartheta}}_{\varsigma}\,\mathlarger{\mathlarger{\bullet}} \,\mathbf{Q}^{\mathsmaller{\vartheta}}_{r-\varsigma}$, we find that~(\ref{Pinel}) is bounded by
\begin{align*}
\mathbf{Q}^{\mathsmaller{\vartheta}}_{t_1-t_0}\,\mathlarger{\mathlarger{\circ}}\,&\,\mathbf{Q}^{\mathsmaller{\vartheta}}_{t_2-t_1} \,\mathlarger{\mathlarger{\circ}} \, \cdots   \,\mathlarger{\mathlarger{\circ}}\,\mathbf{Q}^{\mathsmaller{\vartheta}}_{t_k-t_{k-1}} \,\mathlarger{\mathlarger{\circ}}\,\mathbf{Q}^{\mathsmaller{\vartheta}}_{a-t_{k}} \,\mathlarger{\mathlarger{\bullet}}\,\mathbf{K}^{\mathsmaller{\vartheta}}_{\varsigma} \,\mathlarger{\mathlarger{\bullet}}\,\mathbf{Q}^{\mathsmaller{\vartheta}}_{t_{k+1}-a-\varsigma}\,\mathlarger{\mathlarger{\circ}}\,\mathbf{Q}^{\mathsmaller{\vartheta}}_{t_{k+2}-t_{k+1}} \,\mathlarger{\mathlarger{\circ}} \, \cdots   \,\mathlarger{\mathlarger{\circ}}\,\mathbf{Q}^{\mathsmaller{\vartheta}}_{t_m-t_{m-1}} \,.
\end{align*}
The above measure has $\mathbf{Q}^{\mathsmaller{\vartheta}}_{a-s} \,\mathlarger{\mathlarger{\bullet}}\,\mathbf{K}^{\mathsmaller{\vartheta}}_{\varsigma} \,\mathlarger{\mathlarger{\bullet}}\,\mathbf{Q}^{\mathsmaller{\vartheta}}_{t-a-\varsigma}$ as one of its projections, and it thus vanishes vaguely in consequence of Lemma~\ref{LemmaScheme}.  Therefore~(\ref{Pinel}) vanishes  vaguely.\vspace{.2cm}

\noindent (vi)  It suffices to show that $\mathscr{Z}^{\mathsmaller{\vartheta}}_{P} $ depends continuously on the partition  $P=\{t_0,t_1,\ldots ,t_m\}$ as one point $t_k$  is varied with the others fixed. For  $r\in (t_k,t_{k+1})$ define $ P'= \big(P\backslash\{t_k\}\big)  \cup \{ r \}  $.  By Proposition~\ref{PropCondGen}, we have
\begin{align*}
\mathbb{E}\Big[\,\big( \mathscr{Z}^{\mathsmaller{\vartheta}}_{P}  \,-\, \mathscr{Z}^{\mathsmaller{\vartheta}}_{P_{\leq t_k}}  \,\mathlarger{\mathlarger{\circ}}\,\mathbf{U}_{r-t_k }\,\mathlarger{\mathlarger{\bullet}}\,\mathscr{Z}^{\mathsmaller{\vartheta}}_{P_{\geq r} }\big)^2\,\Big]\,=\,&\,\mathbf{Q}_{P_{\leq t_k}}^{\mathsmaller{\vartheta}}\,\mathlarger{\mathlarger{\circ}}\,\mathbf{K}_{r-t_k}^{\mathsmaller{\vartheta}}\,\mathlarger{\mathlarger{\bullet}}\,\mathbf{Q}^{\mathsmaller{\vartheta}}_{P_{\geq r } }\,,\\ \mathbb{E}\Big[\,\big( \mathscr{Z}^{\mathsmaller{\vartheta}}_{P'}  \,-\, \mathscr{Z}^{\mathsmaller{\vartheta}}_{P_{\leq t_k}}  \,\mathlarger{\mathlarger{\bullet}}\,\mathbf{U}_{ r-t_{k} }\,\mathlarger{\mathlarger{\circ}}\,\mathscr{Z}^{\mathsmaller{\vartheta}}_{P_{\geq r }}\big)^2\,\Big]\,=\,&\,\mathbf{Q}_{P_{\leq t_k}}^{\mathsmaller{\vartheta}}\,\mathlarger{\mathlarger{\bullet}}\,\mathbf{K}_{r-t_{k}}^{\mathsmaller{\vartheta}}\,\mathlarger{\mathlarger{\circ}}\,\mathbf{Q}^{\mathsmaller{\vartheta}}_{P_{\geq r } }\,.
\end{align*}
The above  measures vanish vaguely as $k\rightarrow \infty$ by Lemma~\ref{LemmaSqueeze}, so we can focus on bounding the difference between $\mathscr{Z}^{\mathsmaller{\vartheta}}_{P_{\leq t_k}}  \,\mathlarger{\mathlarger{\circ}}\,\mathbf{U}_{r-t_k }\,\mathlarger{\mathlarger{\bullet}}\,\mathscr{Z}^{\mathsmaller{\vartheta}}_{P_{\geq r} }$ and $\mathscr{Z}^{\mathsmaller{\vartheta}}_{P_{\leq t_k}}  \,\mathlarger{\mathlarger{\bullet}}\,\mathbf{U}_{ r-t_{k} }\,\mathlarger{\mathlarger{\circ}}\,\mathscr{Z}^{\mathsmaller{\vartheta}}_{P_{\geq r }}$, which  are both projections of 
$\mathscr{Z}^{\mathsmaller{\vartheta}}_{P_{\leq t_k}}  \,\mathlarger{\mathlarger{\circ}}\,\mathbf{U}_{ r-t_{k} }\,\mathlarger{\mathlarger{\circ}}\,\mathscr{Z}^{\mathsmaller{\vartheta}}_{P_{\geq r }}$. For $\varphi\in C_{c}^1\big( (\R^2)^{m+1} \big)$,  we can write
\begin{align}
\mathbb{E}\bigg[\,\Big| \mathscr{Z}^{\mathsmaller{\vartheta}}_{P_{\leq t_k}}  \,&\,\mathlarger{\mathlarger{\circ}}\,\mathbf{U}_{r-t_k }\,\mathlarger{\mathlarger{\bullet}}\,\mathscr{Z}^{\mathsmaller{\vartheta}}_{P_{\geq r} }(\varphi)\,-\,  \mathscr{Z}^{\mathsmaller{\vartheta}}_{P_{\leq t_k}}  \,\mathlarger{\mathlarger{\bullet}}\,\mathbf{U}_{ r-t_{k} }\,\mathlarger{\mathlarger{\circ}}\,\mathscr{Z}^{\mathsmaller{\vartheta}}_{P_{\geq r }}(\varphi)\Big|\,\bigg]\nonumber  \\ 
  =\, &\,\mathbb{E}\Bigg[\,\bigg|\int_{ (\R^2)^{m+2} } \Big(\varphi\big(x_0,\ldots, dx_m\big) 
\,-\, \varphi\big(x_0,\ldots, dx_{k-1}, y,x_{k+1},\ldots, x_m\big)   \Big)\nonumber \\  &\hspace{.5cm}\mathscr{Z}^{\mathsmaller{\vartheta}}_{P_{\leq t_k}}  \,\mathlarger{\mathlarger{\circ}}\,\mathbf{U}_{ r-t_{k} }\,\mathlarger{\mathlarger{\circ}}\,\mathscr{Z}^{\mathsmaller{\vartheta}}_{P_{\geq r }}\big(dx_0,\ldots, dx_k,  dy,dx_{k+1},\ldots ,dx_m  \big)\bigg|\,\Bigg] \nonumber  \\
 \leq \, &\,\int_{ (\R^2)^{m+2} } \Big|\varphi\big(x_0,\ldots, dx_m\big) 
\,-\, \varphi\big(x_0,\ldots, dx_{k-1}, y,x_{k+1},\ldots, x_m\big)   \Big|\nonumber \\  &\hspace{.5cm}\mathbf{U}_{P\cup \{r\} }\big(dx_0,\ldots, dx_k,  dy,dx_{k+1},\ldots ,dx_m  \big)\,,\label{UgDiff}
\end{align}
in which we have used that 
\[
\mathbb{E}\Big[ \, \mathscr{Z}^{\mathsmaller{\vartheta}}_{P_{\leq t_k}}  \,\mathlarger{\mathlarger{\circ}}\,\mathbf{U}_{ r-t_{k} }\,\mathlarger{\mathlarger{\circ}}\,\mathscr{Z}^{\mathsmaller{\vartheta}}_{P_{\geq r }}\,\Big]\,=\,\mathbb{E}\big[\,  \mathscr{Z}^{\mathsmaller{\vartheta}}_{P_{\leq t_k}}\,  \big]\,\mathlarger{\mathlarger{\circ}}\,\mathbf{U}_{ r-t_{k} }\,\mathlarger{\mathlarger{\circ}}\,\mathbb{E}\big[\,\mathscr{Z}^{\mathsmaller{\vartheta}}_{P_{\geq r }}\,\big]\,=\, \mathbf{U}^{\mathsmaller{\vartheta}}_{P_{\leq t_k}} \,\mathlarger{\mathlarger{\circ}}\,\mathbf{U}_{ r-t_{k} }\,\mathlarger{\mathlarger{\circ}}\,\mathbf{U}_{P_{\geq r }}\,=\,\mathbf{U}_{P\cup \{r\} }\,.
\]
Pick $R>0$ large enough so that $\textup{supp}(\varphi)$ is a subset of $\big\{(x_0,\ldots, x_{m})\in  (\R^2)^{m+1} : |x_0| < R  \big\}$.  Then~(\ref{UgDiff}) is bounded by
\begin{align*}
\bigg\| \frac{\partial}{\partial x_k}\varphi\bigg\|_{\infty} \,&\,\int_{ (\R^2)^{m+2} }\,1_{\{|x_0|\leq R\}}\,|x_k-y|\,
\mathbf{U}_{P\cup \{r\} }\big(dx_0,\ldots, dx_k,  dy,dx_{k+1},\ldots ,dx_m  \big) \\  \, & \,=  \,\pi\,R^2\,\bigg\| \frac{\partial}{\partial x_k}\varphi\bigg\|_{\infty} \,\int_{\R^2}\,|a|\, g_{r-t_k}(a)\,da\,,
\end{align*}
which vanishes as $r\searrow t_k$. \vspace{.2cm}

\noindent (vii) For $a\in [t,s_1]$, define $P'_{a}=\{a, s_1,\ldots, s_{\ell}  \}$. By Proposition~\ref{PropCondGen} the second moment of the difference between  $\mathscr{Z}^{\mathsmaller{\vartheta}}_{P\cup P'}$ and $\mathscr{Z}^{\mathsmaller{\vartheta}}_{P}\,\mathlarger{\mathlarger{\circ}}_{\varsigma}\,\mathscr{Z}^{\mathsmaller{\vartheta}}_{P'_{t+\varsigma} }$ for $\varsigma\in [0, s_1-t]$ can be expressed as
\[
\mathbb{E}\Big[\,\big( \mathscr{Z}^{\mathsmaller{\vartheta}}_{P\cup P'} -\mathscr{Z}^{\mathsmaller{\vartheta}}_{P}\,\mathlarger{\mathlarger{\circ}}_{\varsigma}\,\mathscr{Z}^{\mathsmaller{\vartheta}}_{P'_{t+\varsigma} }\big)^2\,\Big]\,=\,\mathbf{Q}_{P}^{\mathsmaller{\vartheta}}\,\mathlarger{\mathlarger{\circ}}\,\mathbf{K}_{\varsigma}^{\mathsmaller{\vartheta}}\,\mathlarger{\mathlarger{\bullet}}\,\mathbf{Q}_{P_{t+\varsigma}'}^{\mathsmaller{\vartheta}}\,,
\]
which vanishes vaguely  as $\varsigma\rightarrow 0$ by Lemma~\ref{LemmaSqueeze}. Since $\mathlarger{\mathlarger{\bullet}}_{\varsigma}\,\mathscr{Z}^{\mathsmaller{\vartheta}}_{P'_{t+\varsigma} }$ is the conditional expectation of $\mathscr{Z}^{\mathsmaller{\vartheta}}_{P' }$ with respect to $\mathscr{F}_{[0,t]\cup [t+\varsigma,\infty)} $, we have 
\begin{align*}
\mathbb{E}\Big[\,\big( \mathscr{Z}^{\mathsmaller{\vartheta}}_{P}\,\mathlarger{\mathlarger{\circ}}_{\varsigma}\,\mathscr{Z}^{\mathsmaller{\vartheta}}_{P' }\, -\,\mathscr{Z}^{\mathsmaller{\vartheta}}_{P' }\,\mathlarger{\mathlarger{\circ}}_{2\varsigma}\,\mathscr{Z}^{\mathsmaller{\vartheta}}_{P'_{t+\varsigma}} \big)^2\,\Big]\,=\,& \mathbb{E}\Big[\,\big( \mathscr{Z}^{\mathsmaller{\vartheta}}_{P}\,\mathlarger{\mathlarger{\circ}}_{\varsigma}\,\big(\mathscr{Z}^{\mathsmaller{\vartheta}}_{P' }\, -\,\mathlarger{\mathlarger{\bullet}}_{\varsigma}\,\mathscr{Z}^{\mathsmaller{\vartheta}}_{P'_{t+\varsigma}} \big)\big)^2\,\Big]\\
\,=\,&
\mathbb{E}\Big[ \,\big( \mathscr{Z}^{\mathsmaller{\vartheta}}_{P}\big)^2\,\Big]\,\mathlarger{\mathlarger{\circ}}_{\varsigma}\,\mathbb{E}\Big[\,\big(\mathscr{Z}^{\mathsmaller{\vartheta}}_{P' }\, -\,\mathlarger{\mathlarger{\bullet}}_{\varsigma}\,\mathscr{Z}^{\mathsmaller{\vartheta}}_{P'_{t+\varsigma}}\big)^2\,\Big]\\ \,=\,&\, \mathbf{Q}_{P}^{\mathsmaller{\vartheta}}\,\mathlarger{\mathlarger{\circ}}_{\varsigma}\,\mathbf{K}_{\varsigma}^{\mathsmaller{\vartheta}}\,\mathlarger{\mathlarger{\bullet}}\,\mathbf{Q}_{P_{t+\varsigma}'}^{\mathsmaller{\vartheta}}\,,
\end{align*}
where we have used Proposition~\ref{PropCondGen}.  Thus the second moment of the difference between $\mathscr{Z}^{\mathsmaller{\vartheta}}_{P}\,\mathlarger{\mathlarger{\circ}}_{\varsigma}\,\mathscr{Z}^{\mathsmaller{\vartheta}}_{P' }$ and $\mathscr{Z}^{\mathsmaller{\vartheta}}_{P' }\,\mathlarger{\mathlarger{\circ}}_{2\varsigma}\,\mathscr{Z}^{\mathsmaller{\vartheta}}_{P'_{t+\varsigma}} $ vanishes vaguely by Lemma~\ref{LemmaSqueeze}.  It remains to bound the difference between $\mathscr{Z}^{\mathsmaller{\vartheta}}_{P' }\,\mathlarger{\mathlarger{\circ}}_{\varsigma}\,\mathscr{Z}^{\mathsmaller{\vartheta}}_{P'_{t+\varsigma}} $ and $\mathscr{Z}^{\mathsmaller{\vartheta}}_{P' }\,\mathlarger{\mathlarger{\circ}}_{2\varsigma}\,\mathscr{Z}^{\mathsmaller{\vartheta}}_{P'_{t+\varsigma}} $, which has second moment
\begin{align*}
\mathbb{E}\Big[\,\big( \mathscr{Z}^{\mathsmaller{\vartheta}}_{P' }\,\mathlarger{\mathlarger{\circ}}_{\varsigma}\,\mathscr{Z}^{\mathsmaller{\vartheta}}_{P'_{t+\varsigma}}\, -\,\mathscr{Z}^{\mathsmaller{\vartheta}}_{P' }\,\mathlarger{\mathlarger{\circ}}_{2\varsigma}\,\mathscr{Z}^{\mathsmaller{\vartheta}}_{P'_{t+\varsigma}} \big)^2\,\Big]\,=\,\sum_{a,b\in \{1,2\} }\,(-1)^{a+b}\,\lambda^{a,b}_{\varsigma}\,,
\end{align*}
where
\[
\lambda^{a,b}_{\varsigma}\,:=\,\mathbb{E}\Big[\,\big( \mathscr{Z}^{\mathsmaller{\vartheta}}_{P' }\,\mathlarger{\mathlarger{\circ}}_{a\varsigma}\,\mathscr{Z}^{\mathsmaller{\vartheta}}_{P'_{t+\varsigma}}\big)\times \big(\mathscr{Z}^{\mathsmaller{\vartheta}}_{P' }\,\mathlarger{\mathlarger{\circ}}_{b\varsigma}\,\mathscr{Z}^{\mathsmaller{\vartheta}}_{P'_{t+\varsigma}} \big)\,\Big]\,.
\]
We can express $\lambda^{a,b}_{\varsigma}$ in the form
\begin{align*}
 \lambda^{a,b}_{\varsigma}\big(dx_0,dx_0';\ldots ; dx_{m+\ell },dx_{m+\ell}'\big)\,=\,\int_{y,y'\in \R^2 }\,&\,\mathbf{Q}_{ P }^{\mathsmaller{\vartheta}}\big(dx_0,dx_0';\ldots ;dx_m,dx_m'\big)\,g_{a\varsigma}(x_m-y)\,   \,g_{b\varsigma}(x_m'-y')\,\\  &\mathbf{Q}_{ P'_{t+\varsigma}}^{\mathsmaller{\vartheta}}\big(dy,dy';dx_{m+1},dx_{m+1}';\ldots ; dx_{m+\ell},dx_{m+\ell}'\big)\,.
\end{align*}
For each $a,b\in \{1,2\}$, the measure $\lambda^{a,b}_{\varsigma}$ 
 converges vaguely to $\mathbf{Q}_{ P\cup P' }^{\mathsmaller{\vartheta}}$, and thus the second moment of the difference between $\mathscr{Z}^{\mathsmaller{\vartheta}}_{P' }\,\mathlarger{\mathlarger{\circ}}_{\varsigma}\,\mathscr{Z}^{\mathsmaller{\vartheta}}_{P'_{t+\varsigma}} $ and $\mathscr{Z}^{\mathsmaller{\vartheta}}_{P' }\,\mathlarger{\mathlarger{\circ}}_{2\varsigma}\,\mathscr{Z}^{\mathsmaller{\vartheta}}_{P'_{t+\varsigma}} $ vanishes vaguely.
\end{proof}

\section{The continuum polymer measures: construction and  properties}\label{SecCPM}

The construction of the random path measure $\mathscr{Z}^{\mathsmaller{\vartheta}}$ in Section~\ref{SubsectPathMeasConst} follows a standard form involving the  Kolmogorov's extension theorem, where the needed  consistency property is provided by (v) of Proposition~\ref{PropSHFExtensionProp}.   The proofs of  Propositions~\ref{PropPathMeasProp} \&~\ref{PropCondExpLast} in Section~\ref{SubsectPathMeasProp} are straightforward consequences of the defining relation $\mathscr{Z}^{\mathsmaller{\vartheta}}\circ \Pi_P^{-1}=\mathscr{Z}^{\mathsmaller{\vartheta}}_P$ from~(\ref{ZProject}) and the properties of the multi-interval 2d SHF in Propositions~\ref{PropSHFExtensionProp} \& \ref{PropCondGen}. In Section~\ref{SubsectionSM} we explain how the theorems in Section~\ref{SubsectCPMSM} concerning the path measure $\mathbf{Q}^{\mathsmaller{\vartheta}}_{[s,t]}$ relate to results in~\cite{CM}.

\subsection{Proof of Theorem~\ref{ThmPathMeasConst}}\label{SubsectPathMeasConst}

For $n\in \mathbb{N}$ define the linear map
$\Xi_{n}:(\R^2)^{2^n +1}\rightarrow (\R^2)^{2^{n-1} +1}  $ by
\[
(x_0,x_1, x_2,\ldots, x_{2^n})  \hspace{.5cm} \mapsto \hspace{.5cm}  (x_0,x_{2}, x_4,\ldots, x_{2^{n}}) \,,
\]
that is, $\Xi_{n}$ acts by deleting the odd-indexed components $x_j$ of the preimage tuple.  Define $\widehat{\Upsilon}_{[0,1]}:= (\R^2)^{D  }$, where  $D $ denotes the set of dyadic numbers in $[0,1]$, i.e., of the form $x=\frac{k}{2^n}$ for $n\in \mathbb{N}$ and $k\in \{0,1,\ldots,2^n\}$.

\begin{proof} For convenience, let us  put $[s,t]=[0,1]$. Let $\{P_n\}_0^{\infty}$ be the sequence of dyadic partitions of $[0,1]$, that is, with $P_n=\big\{ \frac{k}{2^n} :   k=0,1,\ldots, 2^n    \big\}  $. Then we have  $D=\bigcup_1^{\infty}P_n$.  By (v) of Proposition~\ref{PropSHFExtensionProp}, the random measure $\mathscr{Z}^{\mathsmaller{\vartheta}}_{P_{n-1}}$ is almost surely a projection (marginal) of $\mathscr{Z}^{\mathsmaller{\vartheta}}_{P_{n}}$, in the sense that $\mathscr{Z}^{\mathsmaller{\vartheta}}_{P_{n-1}}=\mathscr{Z}^{\mathsmaller{\vartheta}}_{P_{n}}  \circ \Xi_{n}^{-1} $.
Thus there exists a full measure set  $\Omega'\in \mathcal{F}$ such that $\mathscr{Z}^{\mathsmaller{\vartheta}}_{P_{n-1}}(\omega)=\mathscr{Z}^{\mathsmaller{\vartheta}}_{P_{n}}(\omega)\circ \Xi_{n}^{-1} $ holds for all $n\in \mathbb{N}$ and $\omega\in \Omega'$.  For any $\omega\in \Omega'$  the family of $\sigma$-finite measures $\big\{ \mathscr{Z}^{\mathsmaller{\vartheta}}_{P_{n}}(\omega)  \big\}_{n=1}^{\infty} $ is  consistent, and so by  Kolmogorov's extension theorem there exists a  measure 
$\widehat{\mathscr{Z}}^{\mathsmaller{\vartheta}}(\omega)$ on $\widehat{\Upsilon}_{[0,1]}$ such that
$\mathscr{Z}^{\mathsmaller{\vartheta}}_{P_{n}}(\omega)=\widehat{\mathscr{Z}}^{\mathsmaller{\vartheta}}(\omega)\circ \Pi_{P_n}^{-1}$ holds for each $n\in \mathbb{N}$.  In the case $\omega\in \Omega\backslash \Omega'$, we can simply set $\widehat{\mathscr{Z}}^{\mathsmaller{\vartheta}}(\omega)$ equal to the null measure on $\widehat{\Upsilon}_{[0,1]}$. \vspace{.2cm}

Let $\digamma:\boldsymbol{\Upsilon}_{[0,1]}\rightarrow \boldsymbol{\widehat{\Upsilon}}_{[0,1]} $ denote the map assigning $q\in \boldsymbol{\Upsilon}_{[0,1]}$ its restriction, $q|_{D} $, to $D$.  Then $\digamma$ is injective and maps onto the set $\boldsymbol{\Upsilon'}$  of uniformly continuous functions in $\boldsymbol{\widehat{\Upsilon}}_{[0,1]}$ since  each uniformly continuous function in $\boldsymbol{\widehat{\Upsilon}}_{[0,1]}$ extends uniquely to a continuous function on $[0,1]$. Let
$\Theta_{\digamma}:\mathcal{B}\big(\boldsymbol{\Upsilon}_{[0,1]}\big)\rightarrow \mathcal{B}\big(\boldsymbol{\widehat{\Upsilon}}_{[0,1]}\big) $ be the Borel set map induced by $\digamma$:
 \[
 \Theta_{\digamma}(E)=\{\digamma(p):p\in E     \}\,,\hspace{1cm}E\in \mathcal{B}\big(\boldsymbol{\Upsilon}_{[0,1]}\big)\,.
 \]
   For $\omega\in \Omega'$ we define $ \mathscr{Z}^{\mathsmaller{\vartheta}}(\omega):=\widehat{\mathscr{Z}}^{\mathsmaller{\vartheta}}(\omega)\circ \Theta_{\digamma}$. In the analysis below, we will show that  $\widehat{\mathscr{Z}}^{\mathsmaller{\vartheta}}(\omega)$ takes full measure on $\boldsymbol{\Upsilon'}$ for almost every $\omega\in \Omega'$.\vspace{.2cm}

  For $N\in \mathbb{N}$ define $B_N\subset \R^{2  }$ by  $B_N=\{x: |x|\leq N   \} $ and $\widehat{B}_N\subset \boldsymbol{\widehat{\Upsilon}}_{[0,1]}$ by $\widehat{B}_N=\widehat{\Pi}_0^{-1}[B_N] $, where $\widehat{\Pi}_t: \boldsymbol{\widehat{\Upsilon}}_{[0,1]}\rightarrow \R^2$ denotes the evaluation map at $t\in D$.  Put $\widehat{\mathscr{Z}}^{\mathsmaller{\vartheta},\mathsmaller{N}}:=1_{\overline{B}_N}\widehat{\mathscr{Z}}^{\mathsmaller{\vartheta}} $.  Realizations of $\widehat{\mathscr{Z}}^{\mathsmaller{\vartheta},\mathsmaller{N}}$ are almost surely finite since
\[
\mathbb{E}\Big[ \, \widehat{\mathscr{Z}}^{\mathsmaller{\vartheta},\mathsmaller{N}}\big( \boldsymbol{\widehat{\Upsilon}}_{[0,1]}\big)\,\Big]\,=\,\mathbb{E}\Big[ \, \boldsymbol{\mathscr{Z}}^{\mathsmaller{\vartheta}}_{0,1}\big(B_N\times \R^2\big)\,\Big]\,=\,\int_{B_N\times \R^2}\,g_1(x-y) \,dy\,dx\,= \, \frac{\pi}{2} N^2\,.  
\]
For any $0\leq \mathbf{s}<\mathbf{t}\leq 1$ and $m\in\mathbb{N}$, we have that
\begin{align*}
\mathbb{E}\bigg[ \, \int_{ \boldsymbol{\widehat{\Upsilon}}_{[0,1]}}\,\big|p(\mathbf{t})-p(\mathbf{s})\big|^{2m}   \,  \widehat{\mathscr{Z}}^{\mathsmaller{\vartheta},N}(dp)\,\bigg]\,=\,&\, \mathbb{E}\bigg[\,  \int_{B_N\times \R^2\times \R^2 }\,|y-z|^{2m}   \,  \boldsymbol{\mathscr{Z}}^{\mathsmaller{\vartheta}}_{ \{0,\mathbf{s},\mathbf{t}\} }(dx,dy,dz)\,\bigg]\\ \,=\,&\,   \int_{B_N\times \R^2\times \R^2 }\,|y-z|^{2m}   \, g_{\mathbf{s}}(x-y)\,g_{\mathbf{t}-\mathbf{s}}(y-z)\,dx\,dy\,dz  \\ \,=\,&\, C_{N,m}\, (\mathbf{t}-\mathbf{s})^m \,,
\end{align*}
where $C_{N,m}:= \pi N^2\frac{ (2m)!  }{ 2^m m! } $.   Chebyshev's inequality yields that 
\[
\mathbb{P}\bigg[ \, \int_{ \boldsymbol{\widehat{\Upsilon}}_{[0,1]} }\,\big|p(\mathbf{t})-p(\mathbf{s})\big|^{2m}   \,  \widehat{\mathscr{Z}}^{\mathsmaller{\vartheta},N}(dp)\,>\,(\mathbf{t}- \mathbf{s})^{m-2} \,\bigg] \,\leq \,  C_{N,m} \,(\mathbf{t}- \mathbf{s})^2 \, .
\]
Hence, the random variable
\[
Y^{N,m}_n\,:=\,\max_{ 1\leq k\leq 2^n}\, \int_{\boldsymbol{\widehat{\Upsilon}}_{[0,1]} }\,\bigg|p\Big(\frac{k}{2^n}\Big)-p\Big(\frac{k-1}{2^n}\Big)\bigg|^{2m}   \,  \widehat{\mathscr{Z}}^{\mathsmaller{\vartheta},N}(dp)\,
\]
satisfies
\[
\mathbb{P}\Big[\, Y^{N,m}_n\,>\, 2^{n(2-m)} \,\Big] \,\leq  \,\frac{ C_{N,m}}{ 2^{n} } \, .
\]
Since $\sum_1^{\infty}\frac{ 1 }{ 2^{n} } $ is summable, 
the first Borel-Cantelli lemma implies that with probability one  $Y^{N,m}_n \leq  2^{n(2-m)}$ holds for all but finitely $n\in \mathbb{N}$.  Thus there exists  an $\mathbb{N}$-valued random variable $\mathbf{n}$ and  a full measure subset $\Omega''_N$  of $ \Omega'$ such that $Y^{N,m}_n(\omega) \leq  2^{n(2-m)}$ holds for all  $n\geq \mathbf{n}(\omega)$, $\omega\in \Omega''_N$.  We put $\Omega''=\bigcap_{1}^{\infty} \Omega''_N$. \vspace{.2cm}

Now we assume that $m\geq 5$.
Applying Chebyshev's inequality to the measure $\widehat{\mathscr{Z}}^{\mathsmaller{\vartheta},N}(\omega)$ yields
\begin{align*}
\max_{1\leq k\leq 2^n }\,\widehat{\mathscr{Z}}^{\mathsmaller{\vartheta},N}\Bigg(\omega;\,\bigg\{\,p\in \boldsymbol{\widehat{\Upsilon}}_{[0,1]} \,: \bigg|p\Big(\frac{k}{2^n}\Big)-p\Big(\frac{k-1}{2^n}\Big)\bigg|\,>\,2^{-\frac{n}{2m}}\,\bigg\}\Bigg) \,\leq \,2^{n} \, Y^{N,m}_n(\omega)   \, ,
\end{align*}
and thus 
\begin{align*}
\,\widehat{\mathscr{Z}}^{\mathsmaller{\vartheta},N}\Bigg(\omega;\,\bigg\{\,p\in \boldsymbol{\widehat{\Upsilon}}_{[0,1]} \,:\,\max_{1\leq k\leq 2^n } \bigg|p\Big(\frac{k}{2^n}\Big)-p\Big(\frac{k-1}{2^n}\Big)\bigg|\,>\,2^{-\frac{n}{2m}}\,\bigg\}\Bigg) \,\leq \,2^{2n} \, Y^{N,m}_n(\omega)
   \, .
\end{align*}
The series $\sum_{n=1}^{\infty} 2^{2n} Y^{N,m}_n(\omega)$ is summable for every $\omega\in \Omega''$ since $ Y^{N,m}_n(\omega)\leq 2^{n(2- m)}$ holds for every $n\geq \mathbf{n}(\omega)$,  Hence, by the first Borel-Cantelli lemma, $\big\{p\in  \boldsymbol{\widehat{\Upsilon}}_{[0,1]}: \mathbf{N}(p)<\infty\big\}  $ is a full measure set for 
$\widehat{\mathscr{Z}}^{\mathsmaller{\vartheta},N}\big(\omega\big)$, where
\[
\mathbf{N}(p)\,=\,  \sup \bigg\{\, n\in \mathbb{N}\,:\, \max_{1\leq k\leq 2^n } \bigg|p\Big(\frac{k}{2^n}\Big)-p\Big(\frac{k-1}{2^n}\Big)\bigg|\,>\,2^{-\frac{n}{2m}} \,\bigg\}\,.
\] 
The set  $\boldsymbol{\Upsilon'}$  is contained in $\big\{p\in  \boldsymbol{\widehat{\Upsilon}}_{[0,1]}: \mathbf{N}(p)<\infty\big\}  $; see, for example, the standard argument in the proof of Kolmogorov's criterion in \cite[pp.\ 54-55]{Karatzas}.  Hence  $\boldsymbol{\Upsilon'}$ is a  $ \widehat{\mathscr{Z}}^{\mathsmaller{\vartheta},N}(\omega)$ full measure set for every $\omega\in \Omega''$.  Since $\widehat{\mathscr{Z}}^{\mathsmaller{\vartheta},N}(\omega)\nearrow \widehat{\mathscr{Z}}^{\mathsmaller{\vartheta}}(\omega)$ as $N\rightarrow \infty$, it follows that $\boldsymbol{\Upsilon'}$ is a full measure set for $\widehat{\mathscr{Z}}^{\mathsmaller{\vartheta}}(\omega)$. \vspace{,2cm}

Let $Q=\{t_0,t_1,\ldots, t_{\ell}\}$ be a partition of a subinterval of $ [0,1]$. We wish to verify that $ \mathscr{Z}^{\mathsmaller{\vartheta}}\circ \Pi_{Q}^{-1}=\mathscr{Z}^{\mathsmaller{\vartheta}}_{Q} $ holds almost surely. Suppose that  $Q\subset D$. Then there exists $n\in \mathbb{N}$ such that $Q\subset P_n$ and $j_0,\ldots, j_{\ell}\in \{0,1,\ldots, 2^n\}$ with $t_i=\frac{ j_i }{ 2^n} $.  For $E\in \mathcal{B}\big( (\R^2)^{\ell+1} \big)$ we have
\begin{align*}
 \mathscr{Z}^{\mathsmaller{\vartheta}}\circ \Pi_{Q}^{-1} (E)&\,\,=\,\,\mathscr{Z}^{\mathsmaller{\vartheta}}\bigg(\bigg\{\,p \in \boldsymbol{\Upsilon}_{[0,1]}: \big(p(t_0),p(t_1),\ldots, p(t_{\ell})\big)\in E    \,\bigg\}   \bigg)\\
&\,\stackrel{\text{a.s.} 
 }{=}\,\widehat{\mathscr{Z}}^{\mathsmaller{\vartheta}}\bigg(\bigg\{\,p \in \boldsymbol{\widehat{\Upsilon}}_{[0,1]}: \big(p(t_0),p(t_1),\ldots, p(t_{\ell})\big)\in E \,   \bigg\}   \bigg)\\
&\,\,=\,\,\widehat{\mathscr{Z}}^{\mathsmaller{\vartheta}}\circ \Pi_{P_n}^{-1} \bigg(\bigg\{\,(x_0,\ldots,x_{2^n}  )\in (\R^2)^{2^n+1} : \big(x_{j_0},x_{j_1},\ldots, x_{j_{\ell}}\big)\in E \,   \bigg\}   \bigg)
\\
&\,\stackrel{\text{a.s.} 
 }{=}\,\mathscr{Z}^{\mathsmaller{\vartheta}}_{P_n} \bigg(\bigg\{\,(x_0,\ldots,x_{2^n}  )\in (\R^2)^{2^n+1} : \big(x_{j_0},x_{j_1},\ldots, x_{j_{\ell}}\big)\in E   \, \bigg\}   \bigg)
\\
&\,\,=\,\mathscr{Z}^{\mathsmaller{\vartheta}}_{P_n}\circ \Pi_{P_n,Q }^{-1}(E)
\\
&\,\stackrel{\text{a.s.} 
 }{=}\,\mathscr{Z}^{\mathsmaller{\vartheta}}_{Q} ( E)\,.
\end{align*}
The second equality above holds since $\mathscr{Z}^{\mathsmaller{\vartheta}}=\widehat{\mathscr{Z}}^{\mathsmaller{\vartheta}}\circ \Theta_{\digamma}$  almost surely, and $\widehat{\mathscr{Z}}^{\mathsmaller{\vartheta}}$ is almost surely supported on the range of $\digamma$.  The fourth and sixth equalities follow from  our construction of $\widehat{\mathscr{Z}}^{\mathsmaller{\vartheta}}$ and (v) of Proposition~\ref{PropSHFExtensionProp}, respectively.  Let $\mathcal{C}$ be some countable subcollection of $ \mathcal{B}\big( (\R^2)^{\ell+1} \big) $ that uniquely determines  measures on $(\R^2)^{\ell+1}$, such as  the collection of $2(\ell+1)$ fold products of $h$-intervals $(\alpha,\beta]$ with rational endpoints. For almost every $\omega\in \Omega$ we have $\mathscr{Z}^{\mathsmaller{\vartheta}}(\omega)\circ \Pi_{Q}^{-1}(E) = \mathscr{Z}^{\mathsmaller{\vartheta}}_{Q}(\omega;E) $ for all $E\in \mathcal{C}$.  Thus $\mathscr{Z}^{\mathsmaller{\vartheta}}\circ \Pi_{Q}^{-1}=\mathscr{Z}^{\mathsmaller{\vartheta}}_{Q}$ holds almost surely.  \vspace{.2cm}

If $Q$ is not a subset of $D$, then we can pick a sequence of partitions $\{Q^{(m)}\} _1^{\infty}$ such that  $Q^{(m)}=\big\{t_0^{(m)},t_1^{(m)},\ldots, t_{\ell}^{(m)}\big\}\subset D$ and $t_j^{(m)}\rightarrow t_j$. By the above,   $ \mathscr{Z}^{\mathsmaller{\vartheta}}\circ \Pi_{Q^{(m)}}^{-1}=\mathscr{Z}^{\mathsmaller{\vartheta}}_{Q^{(m)}} $ holds  almost surely for every $m$. Furthermore, there is vague convergence in $L^2$ $\mathscr{Z}^{\mathsmaller{\vartheta}}_{Q_{m}}\rightarrow \mathscr{Z}^{\mathsmaller{\vartheta}}_{Q} $ by (vi) of Proposition~\ref{PropSHFExtensionProp}.  The convergence $\mathscr{Z}^{\mathsmaller{\vartheta}}\circ \Pi_{Q^{(m)}}^{-1}\rightarrow \mathscr{Z}^{\mathsmaller{\vartheta}}\circ \Pi_{Q}^{-1} $ holds vaguely pointwise, and thus in probability, since  any $\varphi\in C_c\big( (\R^2)^{\ell+1} \big)$ and $\omega\in \Omega$
\begin{align*}
\mathscr{Z}^{\mathsmaller{\vartheta}}(\omega)\circ \Pi_{Q}^{-1}(\varphi)\,=\,\mathscr{Z}^{\mathsmaller{\vartheta}}\big(\omega; \varphi\circ \Pi_{Q}  \big)\,=\,\lim_{m}\mathscr{Z}^{\mathsmaller{\vartheta}}\big(\omega; \varphi\circ \Pi_{Q^{(m)}}  \big)\,=\,\lim_{m}\mathscr{Z}^{\mathsmaller{\vartheta}}(\omega)\circ \Pi_{Q^{(m)}}^{-1}(\varphi)\,,
\end{align*}
in which the second equality holds by the dominated convergence theorem. Therefore we can conclude that $  \mathscr{Z}^{\mathsmaller{\vartheta}}\circ \Pi_{Q}^{-1} = \mathscr{Z}^{\mathsmaller{\vartheta}}_{Q}  $.
\end{proof}

\subsection{Proofs of Propositions~\ref{PropPathMeasProp} \&~\ref{PropCondExpLast}}\label{SubsectPathMeasProp}

In the proof below, we will use that two $\sigma$-finite measures $\mu$, $\lambda$ on $\boldsymbol{\Upsilon}_{[s,t]}$ are equal provided that their pushforwards $\mu\circ \Pi_P^{-1}$ and  $\lambda\circ \Pi_P^{-1}$ by the evaluation map $\Pi_P$ are equal   for every partition $P\in \Lambda_{[s,t]}$ in a sequence of increasingly refined partitions  $P_1, P_2,\,\ldots $  whose mesh vanishes as $n\rightarrow \infty$.

\begin{proof}[Proof of Proposition~\ref{PropPathMeasProp}] 
(i) For any partition $P=\{t_0,\ldots, t_m\}$ of $[s,t]$, the relation $ \mathscr{Z}_{[s,t]}^{\mathsmaller{\vartheta}}\circ \Pi_{P}^{-1}  =  \mathscr{Z}^{\mathsmaller{\vartheta}}_{P} $ implies that 
\begin{align*}
\mathbb{E}\big[ \mathscr{Z}_{[s,t]}^{\mathsmaller{\vartheta}} \big]\,\circ \, \Pi_{P}^{-1} \,=\,\mathbb{E}\big[ \mathscr{Z}_{[s,t]}^{\mathsmaller{\vartheta}}\,\circ \,\Pi_{P}^{-1} \big]\,=\,\mathbb{E}\big[ \mathscr{Z}_{P}^{\mathsmaller{\vartheta}} \big]\,=\,\mathbf{U}_{P}\,,
\end{align*}
where the last equality holds by (i) of Proposition~\ref{PropSHFExtensionProp}.  Since $\mathbf{U}_{P}=\mathbf{U}_{[s,t]}\circ \Pi_{P}^{-1} $, the above implies that the measures $\mathbb{E}\big[ \mathscr{Z}_{[s,t]}^{\mathsmaller{\vartheta}} \big]$ and $\mathbf{U}_{[s,t]}$ are equal.
 We can apply similar reasoning to deduce that the second moment of  $\mathscr{Z}_{[s,t]}^{\mathsmaller{\vartheta}} $ is equal to $\mathbf{Q}^{\mathsmaller{\vartheta}}_{[s,t]} $.\vspace{.2cm}

\noindent (ii) In consequence of (ii) of Proposition~\ref{PropSHFExtensionProp}, it suffices to show that $\sigma\big(\mathscr{Z}_{[s,t]}^{\mathsmaller{\vartheta}} \big)$ is generated by the algebra $\mathcal{A} := \bigcup_{P\in \Lambda_{[s,t]}} \sigma(  \mathscr{Z}_{P}^{\mathsmaller{\vartheta}} )$.    By~\cite[Lemma 4.7]{Kallenberg}  the $\sigma$-algebra $\sigma\big(\mathscr{Z}_{[s,t]}^{\mathsmaller{\vartheta}} \big)$ is generated by $\big\{\mathscr{Z}_{[s,t]}^{\mathsmaller{\vartheta}}(E): E\in \mathcal{I}  \big\} $ when $\mathcal{I}\subset \mathcal{B}(\boldsymbol{\Upsilon}_{[s,t]} )$ is a dissecting semi-ring comprised of bounded subsets of $\boldsymbol{\Upsilon}_{[s,t]}$ generating
$\mathcal{B}(\boldsymbol{\Upsilon}_{[s,t]} )$.  We will construct a dissecting semi-ring $\mathcal{I}$ such that   $\mathscr{Z}_{[s,t]}^{\mathsmaller{\vartheta}}(E) $ is  $\mathcal{A}$-measurable for each $E\in \mathcal{I}$.\vspace{.2cm}

For $n\in \mathbb{N}$ define the  partition $P_n=\big\{t_0^n,\ldots,t_{2^n}^n  \big\}$ of the interval $[s,t]$ by putting $t_{j}^n=s+ \frac{j}{2^n} (t-s)$. For $\boldsymbol{\ell}=(\boldsymbol{\ell}_0,\boldsymbol{\ell}_0';\ldots ; \boldsymbol{\ell}_{2^n},\boldsymbol{\ell}_{2^n}' )\in (\Z\times \Z)^{2^n+1} $, define $C_{\boldsymbol{\ell}}^n \subset (\R^2)^{2^n+1}$ as the product set  
\[
C_{\boldsymbol{\ell}}^n \,:=\,\prod_{j=0}^{2^n}\,\bigg(\frac{\boldsymbol{\ell}_j-1}{2^n} , \frac{\boldsymbol{\ell}_j}{2^n}  \bigg]\times \bigg(\frac{\boldsymbol{\ell}_j'-1}{2^n} , \frac{\boldsymbol{\ell}_j'}{2^n}  \bigg]\,.
\]
Then the  subcollection of $\mathcal{B}(\boldsymbol{\Upsilon}_{[s,t]} )$ given by
\[
\mathcal{I}\,:=\,\bigcup_{n=1}^{\infty} \Big\{\,   \Pi_{P_n}^{-1}\big[C_{\boldsymbol{\ell}}^n\big]\, :\,  \boldsymbol{\ell}\in (\Z\times \Z)^{2^n+1}  \,  \Big\}  
\]
is a dissection semi-ring generating $\mathcal{B}(\boldsymbol{\Upsilon}_{[s,t]} )$ comprised of bounded subsets of $\boldsymbol{\Upsilon}_{[s,t]}$  and satisfying $\sigma(\mathcal{I})=\mathcal{B}(\boldsymbol{\Upsilon}_{[s,t]} )$.  Finally,  $\big\{\mathscr{Z}_{[s,t]}^{\mathsmaller{\vartheta}}(E): E\in \mathcal{I} \big\}=\big\{  \mathscr{Z}_{P_n}^{\mathsmaller{\vartheta}}(C_{\boldsymbol{\ell}}^n):  n\in \mathbb{N},\, \boldsymbol{\ell}\in (\Z\times \Z)^{2^n+1}  \big\} $ is a subset of $\mathcal{A}$.\vspace{.2cm}

\noindent (iii) Let the collection $\mathcal{I}\subset \mathcal{B}(\boldsymbol{\Upsilon}_{[s,t]} )$ be defined above. It suffices to show that $\mathscr{Z}_{[u+s,u+t]}^{\mathsmaller{\vartheta}}\,\circ \,\mathbf{s}_u^{-1} (E)$ and $\mathscr{Z}_{[s,t]}^{\mathsmaller{\vartheta}}(E)$ are equal in distribution for any $E\in \mathcal{I}$.  Given $E\in \mathcal{I}$ there exists $n\in \mathbb{N}$ and  $  \boldsymbol{\ell}\in (\Z\times \Z)^{2^n+1}$. such that 
 $ E=\Pi_{P_n}^{-1}\big[C_{\boldsymbol{\ell}}^n\big] $, and we have
\[
\mathscr{Z}_{[u+s,u+t]}^{\mathsmaller{\vartheta}}\,\circ \,\mathbf{s}_u^{-1} (E)\,=\,\mathscr{Z}_{[u+s,u+t]}^{\mathsmaller{\vartheta}}\circ\Pi_{u+P_n}^{-1}\big[C_{\boldsymbol{\ell}}^n\big]\,=\,\mathscr{Z}_{u+P_n}^{\mathsmaller{\vartheta}}(C_{\boldsymbol{\ell}}^n)\,\stackrel{\textup{d} }{=}\,\mathscr{Z}_{P_n}^{\mathsmaller{\vartheta}}(C_{\boldsymbol{\ell}}^n)\,=\,\mathscr{Z}_{[s,t]}^{\mathsmaller{\vartheta}}(E)\,,
\]
where the second and fourth equalities apply the defining relation~(\ref{ZProject}), and the third equality uses  (iii) of  Proposition~\ref{PropSHFExtensionProp}.\vspace{.2cm}

We can prove (iv) using a similar argument as for (iii).  
\end{proof}

\begin{proof}[Proof of Proposition~\ref{PropCondExpLast}]  
Let $P=\{t_0,\ldots, t_n\}$ be a partition of $[r,u]$ with $s,t\in P$. Since the equality $\mathscr{Z}^{\mathsmaller{\vartheta}}_{[r,u]}\circ \Pi_P^{-1}=\mathscr{Z}^{\mathsmaller{\vartheta}}_{P}$ holds almost surely, we have that
  \begin{align*}
\mathbb{E}\Big[\,\mathscr{Z}^{\mathsmaller{\vartheta}}_{[r,u]}\,\Big|\,\mathscr{F}_{[0,s]\cup [t,\infty) }  \,\Big]\circ \Pi_P^{-1}\,=\,&\,\mathbb{E}\Big[\,\mathscr{Z}^{\mathsmaller{\vartheta}}_{[r,u]}\circ \Pi_P^{-1}\,\Big|\,\mathscr{F}_{[0,s]\cup [t,\infty) } \, \Big]\\ \,=\,&\, \mathbb{E}\Big[\,\mathscr{Z}^{\mathsmaller{\vartheta}}_{P}\,\Big|\,\mathscr{F}_{[0,s]\cup [t,\infty) } \, \Big]\,=\,  \mathscr{Z}^{\mathsmaller{\vartheta}}_{P_{\leq s}}\,\mathlarger{\mathlarger{\circ}}\,\mathbf{U}_{P_{[s,t]}}\,\mathlarger{\mathlarger{\circ}}\,\mathscr{Z}^{\mathsmaller{\vartheta}}_{P_{\geq t} }\,,
  \end{align*}
where the second equality holds by Proposition~\ref{PropCondGen}.  However, since $\mathscr{Z}^{\mathsmaller{\vartheta}}_{P_{\leq s}}\,\mathlarger{\mathlarger{\circ}}\,\mathbf{U}_{P_{[s,t]}}\mathlarger{\mathlarger{\circ}}\mathscr{Z}^{\mathsmaller{\vartheta}}_{P_{\geq t} }$ is the pushforward of
$\mathscr{Z}^{\mathsmaller{\vartheta}}_{[r,s]}\odot   \mathscr{Z}^{\mathsmaller{\vartheta}}_{[t,u]}$ by $\Pi_P$, it follows that $\mathbb{E}\big[\mathscr{Z}^{\mathsmaller{\vartheta}}_{[r,u]}\,\big|\,\mathscr{F}_{[0,s]\cup [t,\infty) }  \big]$ is equal to $\mathscr{Z}^{\mathsmaller{\vartheta}}_{[r,s]}\odot   \mathscr{Z}^{\mathsmaller{\vartheta}}_{[t,u]}$. Likewise, the formula for the pushforward of the second moment of the difference between $\mathscr{Z}^{\mathsmaller{\vartheta}}_{[r,u]}$ and $\mathscr{Z}^{\mathsmaller{\vartheta}}_{[r,s]}\odot   \mathscr{Z}^{\mathsmaller{\vartheta}}_{[t,u]}$ follows easily from~(\ref{QKQ2}). 
\end{proof}

\begin{proof}[Proof of Proposition~\ref{PropConn}]  Let $P=\{t_0,\ldots, t_m\}$ be a partition of $[r,t]$ with $s\in P$.  Then $t_{\ell}=s$ for some $1\leq \ell <m $.  Define $\widetilde{\Pi}_{P}: \boldsymbol{\Upsilon}_{[r,s]}\times \boldsymbol{\Upsilon}_{[s,t]}\rightarrow (\R^2)^{m+1}$  by
\[
\widetilde{\Pi}_{P}(p,q)\,=\,\big(p(t_0),\ldots, p(t_\ell), q(t_{\ell+1}),\ldots, q(t_m)   \big)\,, \hspace{1cm}p\in \boldsymbol{\Upsilon}_{[r,s]}\,, \,\,q\in \boldsymbol{\Upsilon}_{[s,t]}\,.
\]
The pushforwards of  $\mathscr{Z}^{\mathsmaller{\vartheta}}_{[r,t]}\circ\mathlarger{\iota}_{r,s,t}^{-1} $ and $\big(\mathscr{Z}^{\mathsmaller{\vartheta}}_{[r,s]}\circ \Xi_{r,s,+}^{-1} \big)\,\mathlarger{\mathlarger{\bullet}}_{\varsigma}\,\big(\mathscr{Z}^{\mathsmaller{\vartheta}}_{[s,t]}\circ \Xi_{s,t,-}^{-1} \big)$  by $\widetilde{\Pi}_{P}$ are 
\begin{align*}
\big(\mathscr{Z}^{\mathsmaller{\vartheta}}_{[r,t]}\,\circ\,\mathlarger{\iota}_{r,s,t}^{-1}\big) \circ \widetilde{\Pi}_{P}^{-1}\,=\, &\,\mathscr{Z}^{\mathsmaller{\vartheta}}_{[r,t]}\circ \Pi_{P}^{-1}\,=\, \mathscr{Z}^{\mathsmaller{\vartheta}}_{P}    \,,    \\
\Big(\big(\mathscr{Z}^{\mathsmaller{\vartheta}}_{[r,s]}\circ \Xi_{r,s,+}^{-1} \big)\,\mathlarger{\mathlarger{\bullet}}_{\varsigma}\,\big(\mathscr{Z}^{\mathsmaller{\vartheta}}_{[s,t]}\circ \Xi_{s,t,-}^{-1} \big) \Big) \, \circ \,\widetilde{\Pi}_{P}^{-1}\,=\, &\,\big(\mathscr{Z}^{\mathsmaller{\vartheta}}_{[r,s]}\circ \Pi_{P_{\leq s}}^{-1}\big)  \,\mathlarger{\mathlarger{\circ}}_{\varsigma}\, \big(\mathscr{Z}^{\mathsmaller{\vartheta}}_{[s,t]}\circ \Pi_{P_{\geq s}}^{-1}\big)     \,=\,  \mathscr{Z}^{\mathsmaller{\vartheta}}_{P_{\leq s}} \,\mathlarger{\mathlarger{\circ}}_{\varsigma}\,  \mathscr{Z}^{\mathsmaller{\vartheta}}_{P_{\geq s}}    \,.
\end{align*}
However $\mathscr{Z}^{\mathsmaller{\vartheta}}_{P_{\leq s}} \,\mathlarger{\mathlarger{\circ}}_{\varsigma}\,  \mathscr{Z}^{\mathsmaller{\vartheta}}_{P_{\geq s}}$ converges vaguely in $L^2$ to $\mathscr{Z}^{\mathsmaller{\vartheta}}_{P}$ as $\varsigma\rightarrow 0 $ by (vii) of Proposition~\ref{PropSHFExtensionProp}.  This implies that $\big(\mathscr{Z}^{\mathsmaller{\vartheta}}_{[r,s]}\circ \Xi_{r,s,+}^{-1} \big)\,\mathlarger{\mathlarger{\bullet}}_{\varsigma}\,\big(\mathscr{Z}^{\mathsmaller{\vartheta}}_{[s,t]}\circ \Xi_{s,t,-}^{-1} \big) $ converges vaguely in probability to $\mathscr{Z}^{\mathsmaller{\vartheta}}_{[r,t]}\,\circ\,\mathlarger{\iota}_{r,s,t}^{-1}$.
\end{proof}

\subsection{Relating  Theorems~\ref{ThmRNDer} \&~\ref{ThmIntLogHaus} to~\cite{CM}} \label{SubsectionSM}

Recall from~(\ref{QUV}) that  $\mathbf{Q}^{\mathsmaller{\vartheta}}_{[s,t]}$  decomposes   in terms of the measure $\mathbf{V}^{\mathsmaller{\vartheta}}_{[s,t]}$  on $\boldsymbol{\Upsilon}_{[s,t]}$ as
\[
\mathbf{Q}^{\mathsmaller{\vartheta}}_{[s,t]}(dp,dq)\,=\,\mathbf{U}_{[s,t]}\Big(\, d\Big(\frac{p+q }{\sqrt{2}}\Big) \, \Big) \,\mathbf{V}^{\mathsmaller{\vartheta}}_{[s,t]}\Big(\, d\Big(\frac{p-q }{\sqrt{2}}\Big) \, \Big)\,, \hspace{1cm}p,q\in \boldsymbol{\Upsilon}_{[s,t]}\,.
\]
 The observations regarding $\mathbf{Q}^{\mathsmaller{\vartheta}}_{[s,t]}$ collected in Theorems~\ref{ThmRNDer} \& \ref{ThmIntLogHaus} follow directly from those for  $\mathbf{V}^{\mathsmaller{\vartheta}}_{[s,t]}$ in Proposition~\ref{PropLocTime} below, which are consequences of results in~\cite{CM}, as discussed in  this subsection.\vspace{.2cm}

Let $\mathcal{O}$ denote the set of $p\in \boldsymbol{\Upsilon}_{[s,t]}$ with $ 0\in \textup{range}(p)$, i.e., paths that visit the origin. Define $O_{p}:=\big\{a\in [s,t]: p(a)=0 \big\}$. Given $\varepsilon\in (0,1)$  define the map $\mathbf{L}_{[s,t]}^{\varepsilon}:\boldsymbol{\Upsilon}_{[s,t]} \rightarrow [0,\infty)$ by
$$
\mathbf{L}_{[s,t]}^{\varepsilon}(p) \,:=\,\frac{1}{2\varepsilon^2\log^2\frac{1}{\varepsilon}   } \,\textup{meas}\big(  \,\big\{\, a\in [s,t] \,:\,  |p(a)|\leq \varepsilon   \, \big\}\,\big)\,.$$
 In relation to the notations in Section~\ref{SubsectCPMSM} we have 
 $\mathbf{I}_{[s,t]}^{\varepsilon}(p,q)= \mathbf{L}_{[s,t]}^{\varepsilon}\big(  \frac{p-q}{\sqrt{2}} \big)$,  $ I_{p,q}=O_{\frac{p-q}{\sqrt{2}} }$, and a pair $(p,q)\in \boldsymbol{\Upsilon}_{[s,t]}^2$ is in $ \mathcal{I}$ if and only if $\frac{p-q}{\sqrt{2}}\in \mathcal{O}$.
\begin{proposition}\label{PropLocTime} Fix $\vartheta\in \R$ and $0\leq s<t$. 
\begin{enumerate}[(i)]
\item  The  Lebesgue decomposition of  $\mathbf{V}^{\mathsmaller{\vartheta}}_{[s,t]}$ with respect to   $\mathbf{U}_{[s,t]}$ has components $\mathbf{U}_{[s,t]}$ and  $\mathbf{V}^{\mathsmaller{\vartheta}}_{[s,t]}-\mathbf{U}_{[s,t]}$, 
which are  supported on $\mathcal{O}^c$ and $\mathcal{O}$, respectively.

\item There exists a Borel measurable function $\mathbf{L}_{[s,t]}$ on $\boldsymbol{\Upsilon}_{[s,t]}$ such that $\mathbf{L}_{[s,t]}^{\varepsilon} \rightarrow \mathbf{L}_{[s,t]}$ in  $L^1_{\textup{loc} }\big(\mathbf{V}^{\mathsmaller{\vartheta}}_{[s,t]} \big)$ as $\varepsilon\rightarrow 0$. 
\item For any $\vartheta'\in \R$ the path measure $\mathbf{V}^{\mathsmaller{\vartheta}'}_{[s,t]}$ is absolutely continuous with respect to $\mathbf{V}^{\mathsmaller{\vartheta}}_{[s,t]}$ with Radon-Nikodym derivative 
\[ \frac{d\mathbf{V}^{\mathsmaller{\vartheta}'}_{[s,t]}}{d\mathbf{V}^{\mathsmaller{\vartheta}}_{[s,t]}}\,=\,\textup{exp}\big\{\,(\mathsmaller{\vartheta}'-\mathsmaller{\vartheta})\, \mathbf{L}_{[s,t]}  \, \big\} \,.
\]

\item  $\mathbf{L}_{[s,t]}$ is $\mathbf{V}^{\mathsmaller{\vartheta}}_{[s,t]}$-almost surely positive on $\mathcal{O}$.

\item The set $O_{p}$ has log-Hausdorff exponent one for $\mathbf{V}^{\mathsmaller{\vartheta}}_{[s,t]}$-almost every $p\in \mathcal{O}$. 

\end{enumerate}

\end{proposition}
For the remainder of this subsection, let us set $[s,t]=[0,T]$ for  $T>0$, and put $\mathbf{V}^{T,\mathsmaller{\vartheta}} := \mathbf{V}^{\mathsmaller{\vartheta}}_{[s,t]}$ and $\mathbf{U}^{T} := \mathbf{U}_{[s,t]}$.  We can write
\begin{align}\label{VDis}
\mathbf{V}^{T,\mathsmaller{\vartheta}}\,=\,\int_{\R^2}\,\mathbf{V}^{T,\mathsmaller{\vartheta}}_x\,dx\,,
\end{align}
where $\mathbf{V}^{T,\mathsmaller{\vartheta}}_x$ for nonzero $x\in \R^2$ is the    finite measure on the path space $\boldsymbol{\Upsilon}_{[0,T]}$ corresponding to  initial position $x$ and transition density kernels $\{P^{\mathsmaller{\vartheta}}_t\}_{t\in [0,\infty)} $.  That is, $\mathbf{V}^{T,\mathsmaller{\vartheta}}_x$ is the path measure whose pushforward by $\Pi_{P}$,  for any partition $P=\{t_0,t_1,\ldots , t_m\}$ of $[0,T]$, is   given by
\begin{align*}
\mathbf{V}^{T,\mathsmaller{\vartheta}}_x\circ \Pi_{P}^{-1}\big(dx_0,dx_1,\ldots, dx_m  \big)\,=\,\prod_{j=1}^m P^{\mathsmaller{\vartheta}}_{t_j-t_{j-1}}(x_{j-1},x_j)\,\delta_x(dx_0)\,dx_1 \,\cdots\,dx_m\,, \hspace{1cm} x_j\in \R^2\,.
\end{align*}
The total mass of $\mathbf{V}^{T,\mathsmaller{\vartheta}}_x$ can be expressed as 
\begin{align}
\mathbf{V}^{T,\mathsmaller{\vartheta}}_x\big(\boldsymbol{\Upsilon}_{[0,T]}\big)\,=\,\int_{\R^2}\,P^{\mathsmaller{\vartheta}}_T(x, y)\,dy\,\stackrel{(\ref{DefK}) }{ =}\, 1\,+\,2\,\pi \,e^{\mathsmaller{\vartheta}}\,\int_{0<r<s<t   }\,  g_r(x)\,\nu'\big((s-r)e^{\mathsmaller{\vartheta}}\big)\,ds\,dr\,,
\end{align}
and it  diverges logarithmically as $x\rightarrow 0$.  The normalized path measure
\begin{align}\label{Pdef}
\mathbf{P}^{T,\mathsmaller{\vartheta}}_x\,:=\,\frac{1 }{\mathbf{V}^{T,\mathsmaller{\vartheta}}_x(\boldsymbol{\Upsilon}_{[0,T]} ) 
 }\,\mathbf{V}^{T,\mathsmaller{\vartheta}}_x
 \end{align}
has transition density kernels $\big\{d^{T,\mathsmaller{\vartheta}}_{s,t}\big\}_{0\leq s<t \leq T} $ given by the following  Doob transformation of $\{P^{\mathsmaller{\vartheta}}_t\}_{t\in [0,T]} $:
\[ \text{}\hspace{1cm}
d^{T,\mathsmaller{\vartheta}}_{s,t}(x,y)\,:=\,\frac{ \int_{\R^2} P^{\mathsmaller{\vartheta}}_{T-t}(y,z)\,dz  }{ \int_{\R^2} P^{\mathsmaller{\vartheta}}_{T-s}(x,z)\,dz } \,P^{\mathsmaller{\vartheta}}_{t-s}(x,y) \,, \hspace{1cm} x,y\in \R^2\backslash \{0\}\,.
\]
Naturally, we can define the density $y\mapsto d^{T,\mathsmaller{\vartheta}}_{s,t}(x,y)$ for $x=0$  by taking a limit of the above as $x\rightarrow 0$. 
The study~\cite{CM} focuses on the probability measures $\mathbf{P}^{T,\mathsmaller{\vartheta}}_x$, denoted therein by $\mathbf{P}^{T,\lambda}_x$ for a parameter $\lambda\in (0,\infty)$   related to $\vartheta\in \R$ through $\lambda=e^{\vartheta}$. Thus results in~\cite{CM} have direct consequences for  $\mathbf{V}^{T,\mathsmaller{\vartheta}}_x$ and hence  also for $\mathbf{V}^{T,\mathsmaller{\vartheta}}$, and in particular their existence  follows from~\cite[Proposition 2.2]{CM}.\vspace{.2cm}

The map   $ (t,y) \mapsto \mathlarger{d}_{s,t}^{T,\mathsmaller{\vartheta} }(x,y) $ for $t\in (s,T)$, $y\in \R^2\backslash \{0\}$  satisfies the forward Kolmogorov equation
\begin{align}\label{KolmogorovForJ}
\frac{\partial}{\partial t} \,\mathlarger{d}_{s,t}^{T,\mathsmaller{\vartheta}}(x,y)\,=\,\frac{1}{2}\,\Delta_y  \,\,\mathlarger{d}_{s,t}^{T,\mathsmaller{\vartheta}}(x,y)\,-\,\nabla_y \,\cdot\,\Big[\, b_{T-t}^{\mathsmaller{\vartheta}}(y) \, \mathlarger{d}_{s,t}^{T,\mathsmaller{\vartheta}}(x,y) \,\Big] \,,
\end{align}
where  $\nabla_y=(\partial_{y_1},\partial_{y_2})$ is the gradient operator and 
$b_{t}^{\mathsmaller{\vartheta}}:\R^2\rightarrow \R^2$ is defined by
\begin{align*}
b_{t}^{\mathsmaller{\vartheta}}(y)\,:=\, \nabla_y  \,\log \mathbf{V}^{t,\mathsmaller{\vartheta}}_y\big(\boldsymbol{\Upsilon}_{[0,t]} \big)   \,=\,\frac{ \nabla_y\mathbf{V}^{t,\mathsmaller{\vartheta}}_y(\boldsymbol{\Upsilon}_{[0,t]} ) }{ \mathbf{V}^{t,\mathsmaller{\vartheta}}_y(\boldsymbol{\Upsilon}_{[0,t]} ) }\,.
\end{align*}
Note that a stochastic differential equation (SDE) corresponding to the  forward Kolmogorov equation~(\ref{KolmogorovForJ}) has the form
    \begin{align}\label{SDEToSolve}
        d\mathbf{X}_t\,=\,d\mathbf{W}_t\,+\,b_{T-t}^{\mathsmaller{\vartheta}}(\mathbf{X}_t)\,dt\,,
    \end{align}
where $\{\mathbf{W}_t\}_{t\in [0,T]}$ is a 2d standard Brownian motion. The drift function $b_{t}^{\mathsmaller{\vartheta}}(y)$ points  in the direction $-y$, that is, towards the origin from $y$, and its magnitude  vanishes super-exponentially as $|y|\rightarrow \infty$ and diverges to infinity as $y\rightarrow 0$ with the asymptotics
  \begin{align}\label{bBlowUp}
\big|b_{t}^{\mathsmaller{\vartheta}}(y)\big|\, \stackrel{y\rightarrow 0 }{=}  \,\frac{ 1 }{|y|\log \frac{1}{|y|} }\left(1 \,+\,\frac{\frac{1}{2}(\vartheta-\log 2) +\gamma   }{\log\frac{1}{|y|}   } \,+\,\mathit{O}\bigg(  \frac{1  }{\log^2\frac{1}{|y|}   }   \bigg)    \right)    \,.
  \end{align}
Define $\bar{b}_{t}^{\mathsmaller{\vartheta}}:(0,\infty)\rightarrow (0,\infty)$ through the relation  $\bar{b}_{t}^{\mathsmaller{\vartheta}}(|y|)=\big|b_{t}^{\mathsmaller{\vartheta}}(y)\big|$.  The radial process $\mathbf{R}_t=|\mathbf{X}_t|$ satisfies the SDE
  \begin{align}\label{SDEToSolveR}
    d\mathbf{R}_t\,=\,d\mathbf{w}_t\,+\,\bigg(\frac{1}{2\mathbf{R}_t}\,-\,    \bar{b}_{T-t}^{\mathsmaller{\vartheta}}(\mathbf{R}_t)\bigg)\,dt\,,
    \end{align}
where $\{\mathbf{w}_t\}_{t\in [0,T]}$ is a 1d standard Brownian motion.  Without the presence of the drift term $-\bar{b}_{T-t}^{\mathsmaller{\vartheta}}(\mathbf{R}_t)\,dt$, the SDE~(\ref{SDEToSolveR}) would be that of a 2d Bessel process, which almost surely does not visit $0$ for any time $t>0$.  Interestingly, the process $\mathbf{R}$---and thus also the process $X$---has a positive probability of visiting $0$ for some $t>0$ even though $\frac{1}{2a}\gg  \bar{b}_{T-t}^{\mathsmaller{\vartheta}}(a)$ as $a\searrow 0$. \vspace{.2cm}

\begin{proof}[Proof of Proposition~\ref{PropLocTime}] (i)  From  \cite[Theorem 2.4(i)]{CM} we have $\mathbf{P}^{T,\mathsmaller{\vartheta}}_x [\mathcal{O}^c ]=\big(\mathbf{V}^{T,\mathsmaller{\vartheta}}_x[\boldsymbol{\Upsilon}_{[0,T]}] \big)^{-1}$, which with~(\ref{Pdef}) implies that $\mathbf{V}^{T,\mathsmaller{\vartheta}}_x [\mathcal{O}^c ]=1 $.   Let $\mathbf{P}^{T}_x$ denote 2d Wiener measure on $\boldsymbol{\Upsilon}_{[0,T]}$  corresponding to initial position $x$.
By  \cite[Theorem 2.4(ii)]{CM} the conditioning of $\mathbf{P}^{T,\mathsmaller{\vartheta}}_x $ to the event $\mathcal{O}^c $ is $\mathbf{P}^{T}_x$, yielding the  first equality below:
\begin{align}\label{P2V}
\mathbf{P}^{T}_x\,=\,\frac{1_{\mathcal{O}^c } }{  \mathbf{P}^{T,\mathsmaller{\vartheta}}_x [\mathcal{O}^c ] }\,\mathbf{P}^{T,\mathsmaller{\vartheta}}_x \,=\,\frac{1_{\mathcal{O}^c } }{  \mathbf{V}^{T,\mathsmaller{\vartheta}}_x [\mathcal{O}^c ] }\,\mathbf{V}^{T,\mathsmaller{\vartheta}}_x\,=\,1_{\mathcal{O}^c } \,\mathbf{V}^{T,\mathsmaller{\vartheta}}_x\,.
\end{align}
The second and third equalities above apply~(\ref{Pdef}) and that $\mathbf{V}^{T,\mathsmaller{\vartheta}}_x [\mathcal{O}^c ]=1 $. 
Thus we can write
\begin{align*}
\mathbf{V}^{T,\mathsmaller{\vartheta}}\,=\, 1_{\mathcal{O}^c } \,\mathbf{V}^{T,\mathsmaller{\vartheta}} \,+\,1_{\mathcal{O} }\, \mathbf{V}^{T,\mathsmaller{\vartheta}}\,=\,\int_{\R^2} \,\mathbf{P}^{T}_x\,dx \,+\,1_{\mathcal{O} }\, \mathbf{V}^{T,\mathsmaller{\vartheta}}\,=\,\mathbf{U}^{T}\,+\,1_{\mathcal{O} }\, \mathbf{V}^{T,\mathsmaller{\vartheta}}\,,
\end{align*}
where the second equality uses~(\ref{VDis}) and~(\ref{P2V}). Note that $\mathbf{U}^{T}$ is supported on $\mathcal{O}^c$ since a 2d Brownian motion starting away from the origin almost surely will never visit there.  Moreover, $\mathbf{V}^{T,\mathsmaller{\vartheta}}-\mathbf{U}^{T}=1_{\mathcal{O} } \mathbf{V}^{T,\mathsmaller{\vartheta}}_x$ is supported on $\mathcal{O}$.  Therefore, $\mathbf{U}^{T}$ and $\mathbf{V}^{T,\mathsmaller{\vartheta}}-\mathbf{U}^{T}$ are the components in the Lebesgue decomposition of $\mathbf{V}^{T,\mathsmaller{\vartheta}}$  with respect to $\mathbf{U}^{T}$. \vspace{.2cm}

\noindent (ii) Define $\mathbf{V}^{T,\mathsmaller{\vartheta},N}:= \int_{|x|\leq N} 
 \mathbf{V}^{T,\mathsmaller{\vartheta}}_x dx$ for $N>0$.  Equivalently, we can express this measure as $\mathbf{V}^{T,\mathsmaller{\vartheta},N}= 1_{E_N}\mathbf{V}^{T,\mathsmaller{\vartheta}}$ for $E_N:=\big\{p\in \boldsymbol{\Upsilon}_{[0,T]}: |p(0)|\leq N    \big\}$. 
 It suffices to show that $\mathbf{L}_{[0,T]}^{\varepsilon} \rightarrow \mathbf{L}_{[0,T]}$ in  $L^1\big(\mathbf{V}^{T,\mathsmaller{\vartheta},N} \big)$ as $\varepsilon\rightarrow 0$ for any $N >0$. Note that $\mathbf{V}^{T,\mathsmaller{\vartheta},N}$  has finite total mass  since $x\rightarrow \mathbf{V}^{T,\mathsmaller{\vartheta}}_x(\boldsymbol{\Upsilon}_{[0,T]})$     is bounded outside of any neighborhood of the origin and has merely logarithmic blowup there.  Thus $\mathbf{P}^{T,\mathsmaller{\vartheta},N}=\frac{1}{\mathbf{V}^{T,\mathsmaller{\vartheta},N}(\boldsymbol{\Upsilon}_{[0,T]}) }  \mathbf{V}^{T,\mathsmaller{\vartheta},N}$ is a well-defined probability measure, and $\mathbf{L}_{[0,T]}^{\varepsilon} \rightarrow \mathbf{L}_{[0,T]}$ holds in   $L^1\big(\mathbf{P}^{T,\mathsmaller{\vartheta},N} \big)$ 
 by 
~\cite[Theorem 2.9]{CM}, which obviously implies the analogous statement with $\mathbf{P}^{T,\mathsmaller{\vartheta},N}$ replaced by $\mathbf{V}^{T,\mathsmaller{\vartheta},N}$.\vspace{.2cm}

\noindent (iii), (iv), \& (v) follow from similar simple arguments respectively using~\cite[Theorem 2.10]{CM},~\cite[Proposition 2.12]{CM}, \&~\cite[Theorem 2.25]{CM}.
\end{proof}

\appendix

\section{Comment on the proof of Lemma~\ref{LemmaSubseqLimit}}\label{SecAppendix}
To prove Lemma~\ref{LemmaSubseqLimit}, it suffices to verify the conditions of the lemma below with
\begin{itemize}
\item  the metric space $(X,\rho_X)$  being $\mathcal{M}$ equipped with the Prohorov metric,

\item $S=\big\{(s,t): 0\leq s<t<\infty\big\}$,

\item $I=(0,1]$, and

\item $ \{ Y_t^{\varepsilon}\}_{t\in S}=\mathscr{Z}^{\mathsmaller{\vartheta},\varepsilon}$.

\end{itemize}
 Conditions (I) and (II) below are implied by~\cite[Lemma 4.1]{BC} and~\cite[Proposition 4.2]{BC}, respectively.
\begin{lemma}\label{LemmaScheme} Let $(X,\rho_X)$ and $(S,\rho_S)$ be polish spaces. 
For each $\varepsilon$ in an indexing set $I$,  let $ \{ Y_t^{\varepsilon}\}_{t\in S}$   be an  $X$-valued process, defined over a probability space $(\Omega^{\varepsilon},\mathcal{F}^{\varepsilon},\mathbb{P}^{\varepsilon})$.  Suppose that (I) \& (II) below hold:
\begin{enumerate}[(I)]
\item  The family $ \{ Y_t^{\varepsilon}\}_{\varepsilon\in I}$  is tight for each $t\in S$.

\item  For each  $t\in S$, $\sup_{\varepsilon\in I 
 }\mathbb{E}^{\varepsilon}\big[\rho_X\big(Y_{s}^{\varepsilon},Y_t^{\varepsilon}\big)\big]\rightarrow 0$ as $s\rightarrow t$. 

\end{enumerate}
Then every sequence in $I$ has a subsequence $\{\varepsilon_j\}_{1}^{\infty}  $ such that the $X^S$-valued random element $\big\{ Y^{\varepsilon_j}_t\big\}_{t\in S}$ converges in the sense of finite-dimensional distributions as $j\rightarrow \infty$.
\end{lemma}
\begin{proof}
Let $D $ be a countable dense subset of $S$.  As a consequence of (I), the family of  $X^D$-valued random elements $\big(\{ Y^{\varepsilon}_s\}_{s\in D} \big)_{\varepsilon\in I} $ is tight. Thus any sequence in $I$ has a further subsequence $\{\varepsilon_j\}_1^{\infty}$ such that $\big\{ Y^{\varepsilon_j}_s\big\}_{s\in D}$ converges in distribution  (as an  $X^D$-valued random element).  By Skorokhod's representation theorem, we can assume that  each $X^D$-valued process $ ( Y_s^{\varepsilon_j})_{s\in D}$ (for $j\in \mathbb{N}$) is defined on a probability space $(\Omega,\mathcal{F},\mathbb{P})$ and that $ \{ Y_s^{\varepsilon_j}\}_{s\in D}$ converges almost surely to a limit $ \{ Y_s\}_{s\in D}$.  We use $\mathbb{E}$ to denote the expectation associated to $\mathbb{P}$.  \vspace{.3cm}

\noindent \textbf{Constructing $\boldsymbol{Y_s}$ for $\boldsymbol{s\in S}$:}
Fix $s\in S$, and let $\{s_k\}_1^{\infty}\subset D$ be any sequence converging to $s$. Given $\delta>0$ pick $N>0$ large enough so that for all $k\geq N$ 
\begin{align}\label{PickN}
  \sup_{\varepsilon\in I  }\,\widetilde{\rho}\big(Y_{s_k}^{\varepsilon},Y_s^{\varepsilon}\big)\,\leq \,\frac{\delta}{2} \, ,
\end{align}
where $\widetilde{\rho}(Y_{1},Y_2):= \mathbb{E}\big[ 1\wedge \rho_X(Y_1,Y_2)\big]$ for $X$-valued random variables $Y_1$, $Y_2$. For $k,\ell>N$ and any $j\in \mathbb{N}$ we then have
\begin{align*}
  \widetilde{\rho}\big(Y_{s_k},Y_{s_l}\big)\,\leq \,&\,  \widetilde{\rho}\big(Y_{s_k},Y_{s_k}^{\varepsilon_j}\big)\,+\,\widetilde{\rho}\big(Y_{s_k}^{\varepsilon_j},Y_{s}^{\varepsilon_j}\big)\,+\,\widetilde{\rho}\big(Y_{s}^{\varepsilon_j},Y_{s_l}^{\varepsilon_j}\big)\,+\, \widetilde{\rho}\big(Y_{s_l}^{\varepsilon_j},Y_{s_l}\big) \\
  \,\leq \,  &\,\widetilde{\rho}\big(Y_{s_k},Y_{s_k}^{\varepsilon_j}\big)\,+\, \widetilde{\rho}\big(Y_{s_l}^{\varepsilon_j},Y_{s_l}\big)\,+\,\delta\,,
\end{align*}
and thus taking the limit $j\rightarrow \infty$ yields
\begin{align*}
  \widetilde{\rho}\big(Y_{s_k},Y_{s_l}\big)\,\leq \,&\, \limsup_{j}\Big( \widetilde{\rho}\big(Y_{s_k},Y_{s_k}^{\varepsilon_j}\big)\,+\, \widetilde{\rho}\big(Y_{s_l}^{\varepsilon_j},Y_{s_l}\big)\Big)\,+\,\delta \,=\,\delta\,,
\end{align*}
Since $\delta>0$ is arbitrary, the sequence of random elements $\{Y_{s_j}\}_1^{\infty}$ is Cauchy in probability, and thus converges in probability to a limit, which we denote by $Y_{s}$. \vspace{.3cm}

\noindent \textbf{Convergence in probability of $\boldsymbol{Y_s^{\varepsilon_j} }$ to $\boldsymbol{Y_s}$ for each $\boldsymbol{s\in S}$:}  Pick $N>0$ such that $ \widetilde{\rho}\big(Y_{s_k},Y_{s}\big)\leq \frac{\epsilon}{2}   $ and~(\ref{PickN}) hold for all $k\geq N$. Hence for any $k\geq N$ and $j$
\begin{align*}
  \widetilde{\rho}\big(Y_{s}^{\varepsilon_j },Y_{s}\big)\,\leq \,&\, \widetilde{\rho}\big(Y_{s}^{\varepsilon_j },Y_{s_k}^{\varepsilon_j}\big)\,+\, \widetilde{\rho}\big(Y_{s_k}^{\varepsilon_j},Y_{s_k}\big)\,+\, \widetilde{\rho}\big(Y_{s_k},Y_{s}\big) \,\leq \,  \widetilde{\rho}\big(Y_{s_k}^{\varepsilon_j},Y_{s_k}\big)\,+\,\delta\,.
\end{align*}
and so 
\begin{align*}
  \limsup_j\,\widetilde{\rho}\big(Y_{s}^{\varepsilon_j },Y_{s}\big)\,\leq \,&\,  \limsup_j\,\widetilde{\rho}\big(Y_{s_k}^{\varepsilon_j},Y_{s_k}\big)\,+\,\delta \,=\,\delta \,.
\end{align*}
Therefore, $\{Y_{s}^{\varepsilon_j}\}_1^{\infty}$ converges in probability to $Y_s$. \vspace{.3cm{}}

\noindent \textbf{Finite-dimensional distribution convergence of $\boldsymbol{ \{Y_s^{\varepsilon_j}\}_{s\in S} }$ to $\boldsymbol{ \{Y_s\}_{s\in S}  }$:} Given any $s_1,\ldots, s_n\in S$, there is convergence in probability of the $X^n$-valued random elements
\begin{align*}
 \big(Y_{s_1}^{\varepsilon_j},\ldots, Y_{s_n}^{\varepsilon_j}\big)\hspace{.3cm}\stackbin[ j\rightarrow \infty   ]{ \mathbb{P}}{\longrightarrow} \hspace{.3cm}\big(Y_{s_1},\ldots, Y_{s_n}\big)\,.
\end{align*}
Since convergence in probability implies convergence in distribution, the above convergence also holds in distribution.  Therefore   $ \{Y_s^{\varepsilon_j}\}_{s\in S} $  converges in distribution to $ \{Y_s\}_{s\in S}  $ in the finite-dimensional distributional sense, i.e., under the product topology on $X^S$.
\end{proof}

\end{document}